\documentclass[a4paper, 12pt]{amsart}
\usepackage[a4paper, margin=2cm]{geometry}
\usepackage{amsmath,amsfonts,amssymb,amsthm}
\usepackage{mathrsfs}
\usepackage{mathtools}
\usepackage{graphics}
\usepackage[table]{xcolor}
\usepackage{epsfig}
\usepackage{esint}
\usepackage[pagebackref=true]{hyperref}
\usepackage{enumerate,enumitem}
\usepackage{multirow}

\newcommand{\ignore}[1]{}

\newtheorem{proposition}{Proposition}

\newtheorem{contlem}{Continuity Lemma}

\newtheorem{gradcontlem}{Graded Continuity Lemma}

\newtheorem{mainthm}{Main Theorem}

\newtheorem{remark}{Remark}
\newtheorem{lemma}{Lemma}
\newtheorem{corollary}{Corollary}
\newtheorem{assumption}{Assumption}

\DeclareMathOperator{\supp}{supp}
\newcommand{\Deltap}{\Delta}
\newcommand{\nablap}{\nabla}

\newcommand{\N}{\mathbb{N}}

\newcommand{\R}{\mathbb{R}}

\newcommand{\di}{d}

\newcommand{\domain}{D}

\newcommand{\mop}{\lambda}
\newcommand{\Mop}{\mu}
\newcommand{\convker}{\rho}

\newcommand{\el}{\Lambda}

\newcommand{\aff}{\ell}

\newcommand{\projm}{\mathsf{P}}

\newcommand{\xv}{x}
\newcommand{\yv}{y}
\newcommand{\tv}{t}
\newcommand{\sv}{s}
\newcommand{\rv}{r}
\newcommand{\zv}{z}
\newcommand{\txv}{\mathbf{x}}
\newcommand{\syv}{\mathbf{y}}

\newcommand{\rzv}{\mathbf{z}}
\newcommand{\polv}{\mathsf{z}}
\newcommand{\poly}{\mathsf{z}}

\newcommand{\prim}{\tau}

\newcommand{\primo}{\prim^{(1)}}
\newcommand{\primz}{\prim^{(0)}}

\newcommand{\primm}[2]{\prim^{(#1,#2)}}

\newcommand{\na}{\mathsf{1}}

\newcommand{\llhs}{\Pi}
\newcommand{\lrhs}{\Pi^{-}}

\newcommand{\Ta}[1]{\mathsf{T}_{#1}}

\newcommand{\TT}{\mathsf{T}}
\newcommand{\Tplus}{\TT}
\newcommand{\Tminus}{\TT_{-}}

\newcommand{\Pminus}{\bar{\TT}}

\newcommand{\Aa}[1]{\mathsf{A}_{#1}}

\newcommand{\betaa}{\mathfrak{a}}

\newcommand{\betan}{{\beta_x}}

\newcommand{\set}[1]{\ensuremath{\{#1\}}}
\newcommand{\setc}[2]{\ensuremath{\{#1 :\ #2\}}}

\newcommand{\dd}{\,\text{d}}
\newcommand{\eps}{\varepsilon}

\newcommand{\id}{{\rm id}}

\newcommand{\Tc}[1]{\mop^{#1}}
\newcommand{\TMc}[1]{\Mop^{#1}}
\DeclareMathOperator{\dist}{dist}

\newcommand{\nn}{\mathsf{n}}

\newcommand{\DD}[2]{D^{(#1,#2)}}

\newcommand{\Iu}[1]{\|#1\|_{0}}

\newcommand{\Inu}[1]{\|#1\|_{1}'}

\newcommand{\HNw}[2]{[#1]_{#2}}

\newcommand{\HNNw}[2]{[#1]_{#2}'}
\newcommand{\HHN}[2]{[[#1]]_{#2}}

\usepackage{amsmath,amssymb}
\makeatletter
\newcommand{\opnorm}{\@ifstar\@opnorms\@opnorm}
\newcommand{\@opnorms}[1]{%
\ensuremath{\left|\mkern-1.5mu\left|\mkern-1.5mu\left|
#1
\right|\mkern-1.5mu\right|\mkern-1.5mu\right|}
}
\newcommand{\@opnorm}[2][]{%
\mathopen{#1|\mkern-1.5mu#1|\mkern-1.5mu#1|}
#2\mathclose{#1|\mkern-1.5mu#1|\mkern-1.5mu#1|}
}
\makeatother

\usepackage{color}
\definecolor{darkred}{rgb}{0.9,0.1,0.1}
\definecolor{darkblue}{rgb}{0,0,0.7}
\definecolor{darkgreen}{rgb}{0,0.5,0}

\definecolor{darkergreen}{rgb}{0.0, 0.5, 0.0}

\hypersetup{
pdfauthor={Felix Otto, Jonas Sauer, Scott Smith, Hendrik Weber},
pdftitle={Quasi-linear parabolic equations with singular forcing},
breaklinks=true,
colorlinks=true,
linkcolor=blue,
citecolor=blue,
urlcolor=blue,
filecolor=blue,
}
\title[A priori bounds in the full sub-critical regime]{A priori bounds for quasi-linear SPDEs in the full sub-critical regime} 
\author{Felix~Otto}
\address{Felix~Otto\newline Max Planck Institute for Mathematics in the Sciences \newline Inselstr.\@ 22, D-04103 Leipzig, Germany}
\email{felix.otto@mis.mpg.de}
\author{Jonas~Sauer}
\address{Jonas~Sauer\newline Friedrich-Schiller-Universit\"at Jena \newline Ernst-Abbe-Platz 2, 07737 Jena, Germany}
\email{jonas.sauer@uni-jena.de}
\author{Scott~Smith}
\address{Scott~Smith\newline Academy of Mathematics and System Sciences, Chinese Academy of Sciences \newline No. 55 Zhongguancun East Road, Beijing, China 100190 }
\email{ssmith@amss.ac.cn}
\author{Hendrik~Weber}
\address{Hendrik~Weber\newline Westf\"alische Wilhelms-Universit\"at M\"unster \newline Orléans-Ring 10, 48149 Münster, Germany}
\email{hendrik.weber@uni-muenster.de}

\begin{document}

\begin{abstract}
This paper is concerned with quasi-linear parabolic equations driven by an additive forcing $\xi \in C^{\alpha-2}$, in the full sub-critical regime $\alpha \in (0,1)$.
We are inspired by Hairer's regularity structures,
however we work with a more parsimonious model indexed by multi-indices rather than trees.  This allows us to capture additional symmetries which play a crucial role in our analysis.
Assuming bounds on this model, which is modified in agreement with the concept
of algebraic renormalization, we prove local \emph{a priori} estimates on solutions
to the quasi-linear equations modified by the corresponding counter terms. 
\end{abstract}

\maketitle

\tableofcontents

\section{Introduction}
In this article, we study the quasi-linear parabolic partial differential equation
\begin{align}\label{eq:PDEintro}
\partial_{\tv}u-a(u)\Deltap u=\xi,
\end{align}
where $u=u(\tv,\xv)$ for $(\tv,\xv) \in \R\times\R^{\di}$, $\Deltap=\sum_{i=1}^{\di} \partial_{\xv_i}^2$, and the coefficient field $u \mapsto a(u)$ is sufficiently smooth and uniformly elliptic.
Throughout the paper we use the shorthand notation $\txv:=(\tv,\xv)$, $\syv:=(\sv,\yv)$, and $\rzv:=(\rv,\zv)$ for space-time points.
In line with the pathwise approach to stochastic analysis of Lyons \cite{Lyo98}, the external forcing $\xi$ is deterministic and viewed as a realization of a singular noise\footnote{More precisely, we think of $\xi$ as a realization of a singular noise with a small regularization in space-time and hence make the qualitative assumption that $\xi$ is smooth throughout the paper.
Crucially, all quantitative estimates on the solution depend only on the $C^{\alpha-2}$-norm of $\xi$ and higher order analogues of the corresponding model, see Assumption \ref{sne1} below.} which a.s.\@ belongs to the (negative) parabolic H\"{o}lder space $C^{\alpha-2}$.  For $\alpha \in (0,\infty)$, the PDE \eqref{eq:PDEintro} is sub-critical in the sense of Hairer \cite{Hai14}.
A standard reference point is space-time white noise, which is included in this regime if $d=1$, but marginally fails if $d=2$.
For $\alpha>0$ the solution to \eqref{eq:PDEintro} should behave on small scales like the solution to the linear equation where $a$ is replaced by a constant, which belongs to $C^{\alpha}$ by Schauder theory.
Hence, we expect the same regularity for $u$, but the following difficulty arises: for $\alpha \in (0,1)$, there is no canonical definition of $a(u)\Delta u$ as a limit of smooth approximations.
Indeed, the usual power counting heuristic fails since $u \in C^{\alpha}$ implies $a(u) \in C^{\alpha}$ and $\Delta u \in C^{\alpha-2}$, but $\alpha+\alpha-2<0$.
More concretely, one can carry out explicit calculations with Gaussian noise to see that products of this type often require re-centering by suitable counter-terms, divergent as the smooth regularization is released.
As a result, \eqref{eq:PDEintro} is not expected to be well-posed in the traditional PDE sense and a similar re-centering will be needed for the non-linearity $a(u)\Delta u$, which amounts to adjusting the equation \eqref{eq:PDEintro} with certain counter-terms, known as a renormalization.
 
\medskip

There is now an extensive literature on renormalized stochastic PDE's following the development of regularity structures \cite{Hai13}, \cite{Hai14} and paracontrolled calculus \cite{GIP15}, the main applications of these seminal works being to semi-linear equations, see \emph{e.g.}\@ \cite{HaP15}.
The quasi-linear case was first considered in \cite{OtW19} and soon after in \cite{BDH19}, \cite{FuG19} in the case of $\alpha> \frac23$. The case $\alpha> \frac25$, which in one space dimension includes the case of space-time white noise, was investigated in \cite{Ger20}, \cite{GeH19}.\footnote{A number of aspects of this paper also work for arbitrary $\alpha>0$, but the authors did not identify the renormalized PDE in the full sub-critical regime.}
An alternative approach to this regime inspired by \cite{BDH19} appeared in \cite{bailleul2019paracontrolled}.
See also \cite{RaS20} for a treatment of the initial value problem using the methods of \cite{OtW19} (in the regime $\alpha>\frac23$).
The regime $\alpha>1$ corresponds to spatially colored noise, which has been studied in the articles \cite{hornung2019quasilinear}, \cite{kuehn2020pathwise}, and in the series of papers \cite{AgV22a}, \cite{AgV22b}, and \cite{AgV20c}.
We also mention the articles \cite{dubedat2019stochastic}, \cite{funaki2020asymptotics}, and \cite{LPSX23} where singular quasi-linear SPDE's arise naturally in some relevant physical models.
Finally, we mention the interesting recent work \cite{bruned2024quasigeneralised} which explores the quasi-linear generalized KPZ equation driven by space-time white noise, providing  sufficient conditions for global well-posedness and a large class of examples.

\medskip

In our prior work \cite{otto2018parabolic}, we developed two key analytic tools (see Section \ref{sec:jets}) which applied to arbitrary $\alpha>0$, but applied them in the more restricted regime $\alpha>\frac{1}{2}$.
In fact, in \cite{otto2018parabolic} we considered a more general problem of developing a well-posedness theory for the linear problem with rough coefficients\footnote{Extending the linear theory developed in \cite{otto2018parabolic} to arbitrary $\alpha>0$ remains an interesting and challenging open problem.}.
In the present article, we do not use linear well-posedness theory to treat the non-linear problem \eqref{eq:PDEintro}.
Instead, we shift our perspective and analyze the non-linear problem directly.
Our main result is an \emph{a priori} bound on smooth solutions to a renormalized version of \eqref{eq:PDEintro}.
We provide a framework that applies to all sub-critical regularities $\alpha>0$ and all space dimensions $\di$.
The aforementioned shift in perspective comes with the following merit.
Rather than arguing entirely within the class of modelled distributions (which would be forced upon us if we had to pass through e.g.~a contraction mapping principle), we show that any solution to the renormalized equation admits a local description under the mere assumption of local smallness of the supremum norm.

\medskip

The main inputs for our \autoref{theorem1} are two structural assumptions on the driver $\xi$ that would not hold for an arbitrary $\xi \in C^{\alpha-2}$, but are nonetheless very reasonable  for realizations of a large class of stationary space-time random fields.
On the basis of the approach introduced in this paper, the construction and the stochastic estimates of the renormalized model, which this paper takes as an input, have been carried out in \cite{LOTT21}, and we will comment on how the results of \cite{LOTT21} connect to the present work below Assumption \ref{sne1} and \ref{jc73}.
However, in the spirit of regularity structures, both papers are logically independent: based on purely 
deterministic arguments, this paper establishes uniform interior regularity estimates for the renormalized equation.
Taken together, both papers demonstrate the viability of the approach to regularity structures proposed in this paper.
In particular, both papers are 
written in such a way that they can be read independently.

\medskip

We now state these assumptions and motivate them with the theory of regularity structures.
Inspired by \cite{Hai14}, we rely on a triplet $(\Aa{},\TT,\mathsf{G})$ consisting of a space of homogeneities $\Aa{}\subset\mathbb{R}$, an abstract (linear) model space $\TT$, and a structure group $\mathsf{G}\subset{\rm Aut}(\TT)$, in the sense of Hairer \cite[Definition 3.1]{Hai16}.  
For the black box approach to semi-linear equations developed in  \cite{bruned2020renormalising, bruned2016algebraic, chandra2016analytic}, each $\tau \in \TT$ is a decorated rooted tree (or forest).  A natural attempt to merge the semi-linear machinery with the parametric rough path approach employed in \cite{OtW19}, as advocated in \cite{GeH19} and \cite{otto2018parabolic}, would be to utilize trees depending on one or more parameters.
In the present work, we proceed in a rather different way by using a much smaller vector space $\TT$, which is essentially indexed by multi-indices.  

\medskip

We motivate the form of the triplet $(\Aa{}, \TT, \mathsf{G})$ and its grading here by introducing our twist on Hairer's notion of a centered model, which we view as a parameterization of the solution manifold for a renormalized version of \eqref{eq:PDEintro}.
In fact, it is possible to motivate the algebraic objects that appear in this article, including both the hierarchy of PDE's determining the model and the action of the structure group, as arising from searching for a formal series solution to \eqref{eq:PDEintro}, as we discuss in Section \ref{sec:series} below.
In order to explain the role of multi-indices played in our analysis, we now give a slightly different motivation in line with the rough path perspective where the ensemble of all non-linearities $a$ is considered simultaneously.  
Thinking of $\xi$ as being fixed, we are interested in the analytical properties of the mapping
\begin{align*}
    a \mapsto u[a],
\end{align*}
where we denote by $u[a]$ the solution to \eqref{eq:PDEintro} with nonlinearity $a$. This gives rise to a solution manifold which has an important invariance:
\begin{align}\label{rr16}
u[a]+v = u[a(\cdot-v)] \qquad \mbox{for all } v\in\mathbb{R};
\end{align}
In words, if $u$ solves \eqref{eq:PDEintro}, $u+v$ solves (\ref{eq:PDEintro}) with $a$ replaced by its shifted version $a(\cdot-v)$. In case of a driven ODE like $\partial_t u=a(u)\xi$, this implies
that modulo additive constants, the solution manifold is parameterized by $a$.
Hence in the ODE case
\begin{align}\label{rr08}
z_k:=\frac{1}{k!}\frac{d^ka}{du^k}(0)\quad\mbox{for}\;k\in\mathbb{N}_0
\end{align}
provide a complete set of coordinates for the solution manifold modulo constants. Thinking of $u$ as $u[(\poly_k)_{k\ge 0}]$ in abstract variables $(\poly_k)_{k\ge 0}$ we obtain $u[a]$ upon choosing $\poly_k=z_k$ for all $k\ge 0$.
In our PDE case of (\ref{eq:PDEintro}), the coordinates (\ref{rr08}) are insufficient.
A natural ansatz is to enrich them by a linear jet in some fixed base point $\txv$, which as in (\ref{rr08}) we somewhat arbitrarily fix to be the origin.
We choose the jet to be $z_{x}\cdot\xv$ with
\begin{align}\label{rr09}
z_{x}:=``\nabla_{\xv}u(0)" ;
\end{align}
in view of the invariance (\ref{rr16}) we deliberately drop the constant jet $u(0)$. 
As is common in the theory of rough paths and regularity structures, we will need to re-interpret (\ref{rr09}) as a Gubinelli derivative, cf.~\eqref{jc71}; thus the quotation marks.
Since the coordinate $a_0:=z_0$ will play a slightly different role in our considerations (in contrast to the other coordinates $z_x$ and $(z_k)_{k\ge 1}$, we need arbitrary high powers in the ellipicity $z_0$ if we want to describe the solution $u$ to a finite order of precision in terms of a series expansion), we often make the distinction and write $u[\poly_x,(\poly_k)_{k\ge 1};a_0]$.
Formally, Taylor's formula suggests that the general solution $u$ can be recovered from its partial derivatives with respect to $\poly=(\poly_x,\poly_1,\poly_2,\ldots)$, which are parametrized by the countable set\footnote{Here the space $c_{00}(\N_0):=\setc{\beta:\N\to\N_0}{\supp\beta \text{ is finite}}$ is the space of all $\N_0$-valued sequences of finite support.} $\N_0^d\times c_{00}(\N_0)$ of multi-indices $\beta=(\beta_x,\beta(1),\beta(2),\ldots)$.
An algebraically convenient way to analyze objects labelled by multi-indices is via formal power series. We thus introduce the model space $\TT$ as the space of formal power series in infinitely many abstract variables, the coefficients of which are (complex) analytic functions of a single parameter $a_{0}$.  More precisely, each $\tau \in \TT$ is identified with a formal power series 
\begin{equation}
\sum_{\beta}\tau_{\beta}\polv^{\beta} , \nonumber
\end{equation}
where $\beta=(\beta_{x},\beta(1) ,\beta(2) ,\dots )$ is a multi-index, $\polv^{\beta}:=\polv_x^{\beta_{x} } \prod_{k=1}^{\infty}\polv_{k}^{\beta(k)}$, and each coefficient $\tau_{\beta}$ is a function of a single parameter $a_0$ from a disc $D_{\TT}$ in the right complex half plane containing a (real) interval $I=[\el,\el^{-1}]$ for some fixed ellipticity parameter\footnote{It turns out to be enough to consider functions of a single parameter rather than several parameters since we perform estimates directly on the non-linear problem rather than attempt to develop a theory for the linear problem with rough coefficients, as in \cite{otto2018parabolic}.} $\el\in(0,1)$.  

\medskip

Another, rather minor, difference with standard regularity structures lies in the fact that
we adopt a dual perspective.  In the setting of Hairer, the abstract model space would actually correspond to $\TT^{*}$ rather than $\TT$ and the correspondence between the abstract space of symbols in $\TT^{*}$ and the concrete space-time distributions is specified through a linear map $\Pi_{\txv}: \TT^{*} \to \mathcal{S}'(\R^{\di+1})$.  We alternatively view $\Pi_{\txv}$ as a $\TT$-valued distribution.  For notational reasons, it is convenient to distinguish functions in the local description of $u$, denoted $\llhs_{\txv}$, from distributions in the description of\footnote{More accurately, the components of $\lrhs_\txv$ provide a local description of the renormalized non-linearity $a(u)\Delta u+h(u)$.} $a(u)\Delta u$, denoted  $\lrhs_\txv$, which takes values in the slightly smaller space $\Tminus$ where polynomials are excluded, see Section \ref{sec:ms}.  The $\lrhs_\txv$ can be thought of as an enhancement of the noise, in the sense that for any base-point $\txv$, it holds $\lrhs_{\txv 0}=\xi-q_0$ where $0$ denotes the multi-index with all components being zero, and where we can allow for a constant $q_0$ to ensure that certain ensemble or space-time averages of the noise vanish. Similarly, $\llhs_\txv$ can be viewed as an enhancement of the classical polynomials, in the sense that
\begin{equation}\label{jc03b}
\projm\llhs_{\txv}(\syv)=\poly_x\cdot (\yv-\xv),
\end{equation}
where $\projm$ is the projection onto the polynomial sector, cf.~Section \ref{sec:ms}. 

\medskip

 To each multi-index $\beta$ one can associate a homogeneity $|\beta|$ which is dictated by the inherent scaling of \eqref{eq:PDEintro}, cf.\@ \eqref{model201}.
 This naturally generates a set of homogeneities $\Aa{}$ and a grading of $\TT$ in terms of subspaces $\Ta{|\beta|}$ which consist of those elements of homogeneity $|\beta|$, i.e., of $\tau \in \Tplus$ such that $\tau_{\gamma}=0$ for $| \gamma | \neq |\beta|$.
 These subspaces come with their norms.
 More specifically,  we fix a sequence of discs $\set{D_{|\beta|}}_{|\beta|<2}$, where all $D_{|\beta|}$ have the same center as $D_{\Tplus}$ and are such that
\begin{align*}
 I\subsetneq D_{|\beta|} \subsetneq D_{|\gamma|} \subsetneq D_{\Tplus} \quad \mbox{for } |\gamma|<|\beta|,
\end{align*}
and set
\begin{equation}
\|\tau\|_{\Ta{|\beta|} }:= \sup_{| \gamma |=|\beta| }\sup_{a_0\in D_{|\beta|}} |\tau_{\gamma}(a_0)|.\label{md3}
\end{equation}
 Observe that elements in $\TT$ are (complex) analytic functions\footnote{We remark that we could avoid the use of complex methods altogether by monitoring the number of derivatives with respect to $a_0$ more thoroughly in terms of (real) vector-valued $C^k$-spaces; this approach was used in \cite{otto2018parabolic}.} in $a_0$, so that in view of Cauchy's integral formula derivatives with respect to $a_0$ are conveniently estimated on a marginally larger disc by the function itself; whence the nested form of the $D_{|\beta|}$'s.

\medskip

We now turn to our first assumption on the noise.
For this we introduce an anisotropic distance. Anisotropy in the directions of time and space is due to the parabolic operator $\partial_{\tv}-\Deltap $ and its
mapping properties on the scale of H\"older spaces (\emph{i.e.}, Schauder theory), which thus imposes its intrinsic (Carnot-Carath\'eodory) metric given by
\begin{align}\label{1.11}
d(\txv,\syv)=\sqrt{|\tv-\sv|}+|\xv-\yv|.
\end{align}
We all use the shorthand-notation $\dist_\txv:=\dist(\txv,\partial\domain)$ for the parabolic distance of $\txv$ from the boundary of a domain $\domain$.
\begin{assumption}\label{sne1}
For all $\txv \in B_1(0)\subset\R^{\di+1}$, there exist smooth functions $\llhs_{\txv}: \R^{\di+1}\to\Tplus$ and $\lrhs_\txv: \R^{\di+1}\to\Tminus$ satisfying the compatibility conditions \eqref{jc03b} and\footnote{We may even allow for slightly more flexibility in identity \eqref{jc60} by demanding only that it holds up to an affine function $y \mapsto P_{\txv}(\syv)$, by which we mean $P_{\txv}(\syv)=p_0+p_1\cdot (\yv-\xv)$ for some $p_0\in \Tplus$, $p_1\in \Tplus^{\di}$. } 
\begin{align}\label{jc60}
(\partial_{\tv}-a_{0}\Delta)\llhs_{\txv\beta}&=
\begin{cases}
 \xi, & \text{if } \beta=0, \\
 \lrhs_{\txv\beta}, & \text{if } \beta \ne 0.
\end{cases}
\end{align}
Furthermore, it holds
\begin{align}
 \sup_{|\beta|<2}\sup_{\syv \neq \txv \in B_1(0)\subset\R^{\di+1}} \dist_{\txv}^{\langle \beta \rangle\alpha}d^{-|\beta|}(\syv,\txv)\|\llhs_{\txv}(\syv) \|_{\Ta{|\beta|} }<\infty, \label{model7a} 
 \end{align}
 where $\langle\beta\rangle$ is defined\footnote{We mention that $\langle\beta\rangle$ can be interpreted as the number of appearances of the noise in a tree. Moreover, for $|\beta|<2$, the number $\langle\beta\rangle$ is completely determined by $|\beta|$, cf.\@ Section \ref{sec:hom}.} in \eqref{model200}.
\end{assumption}
Assumption \ref{jc73} concerns the group $\mathsf{G}$, which is a subgroup of the linear endo-morphisms of $\Tplus$, together with a re-expansion map $\Gamma_{\syv\txv} \in \mathsf{G}$ associated to each pair of base-points $\txv, \syv \in \R^{d+1}$.  This is essentially the structure group in the language of \cite[Section 4.2]{Hai16}, with the caveat that due to our dual perspective mentioned above, the transformation $\Gamma_{\syv\txv} \in \mathsf{G}$ corresponds to the adjoint of the corresponding quantity in \cite{Hai14}.  Keeping in mind that elements of $\TT$ are essentially functions of an ellipticity parameter $a_{0}$ and the abstract variables $\polv$, it turns out that elements of $\mathsf{G}$ have an elegant formulation as differential operators in these variables. 
They lead to a parametrization of $\mathsf{G}$ by $\primz\in\TT$
and $\primo\in\TT^d$ (with $\primo_\beta=0$ unless $|\beta|>1$) by an exponential formula, see (\ref{as41}) in Subsection \ref{sec:sg}. Also see the recent work \cite{LOT21}, where $\mathsf{G}$ is shown to arise from a Hopf algebra in the context of our more parsimonious model. In Section \ref{sec:series}, we give a simple motivation for the definition of  the structure group $\Gamma_{\syv\txv}$ based on Taylor's formula and on a formal series solution to the PDE.

\medskip

\begin{assumption}\label{jc73}
For all $\txv,\syv\in B_1(0)$ there exists $\Gamma_{\syv\txv}\in\mathsf{G}$ determined by $\primz_{\syv\txv}\in\Tplus$, $\primo_{\syv\txv}\in\Tplus^d$ with $(\primo_{\syv\txv})_\beta=0$ unless $|\beta|>1$, such that
\begin{align}
&\Gamma_{\syv\txv}\llhs_{\syv} =\llhs_{\txv}- \primz_{\syv\txv}, \mbox{ in particular } \primz_{\syv\txv}\stackrel{\eqref{model7a}}{=}\llhs_{\txv}(\syv), \label{jc15} \\
&\sup_{|\beta| \in (1,2)}\sup_{\syv \neq \txv \in B_{1}(0) }\dist_{\txv}^{\langle \beta \rangle\alpha}d^{1-|\beta|}(\syv,\txv)\|\primo_{\syv\txv}\|_{\TT_{|\beta|} }<\infty  , \qquad \label{model101}.
\end{align}
Furthermore, we assume there exists $q\in \Tplus$ with $q_\beta=0$ for $\beta_x\ne 0$ such that for all $\txv \in \R^{\di+1}$
 \begin{align}\label{jc60a}
   \lrhs_{\txv}(\txv)=\xi(\txv)\na - q,
 \end{align}
 where $\na$ is the unit element in $\TT$ defined by $\na(\polv)=1$.
\end{assumption}
To measure the size of the model, we define
\begin{equation}
[\Pi]:=\sup_{m=0,1}\sup_{\langle \beta \rangle \geq 1,|\beta| \in (m,2)}\sup_{\syv \neq  \txv  \in B_{1}(0)  }\dist_{\txv}^{\langle \beta \rangle\alpha}d^{m-|\beta|}(\syv,\txv)\|\tau_{\syv\txv}^{(m)}\|_{\TT_{|\beta|} } \label{e194},
\end{equation}
which is finite provided Assumptions \ref{sne1} and \ref{jc73} hold.  We emphasize that these assumptions are well-justified by the results in \cite{LOTT21}, as we will explain more precisely at the end of Section \ref{subsec:renorm}. 
The most subtle point of our assumption is hidden in (\ref{jc60a}): 
the innocent looking $q\in\TT$ in fact
is a collection of functions of $a_0$ that determine the counter term $h$ in the renormalized
equation as we show below.
In the application, one should think of
$q$ as deterministic but divergent as the regularization (\emph{i.e.}, through
mollification of $\xi$) vanishes, while
the model $(\llhs_{\txv},\lrhs_{\txv},\Gamma_{\syv\txv})$ is random but stays bounded. 
Loosely speaking, $q$ is what has to be subtracted from $\lrhs_{\txv}$ in order for the latter
to stay bounded.\footnote{Using a re-expansion property of $\lrhs$, cf.~(\ref{jc65}), it is possible to verify that
$\lrhs_{\txv}$ is characterized by $\lrhs_{\txv}({\txv})$.}
The important structural assumption is that $q$ is independent of the base point ${\txv}$
and is not affected by adjoining polynomials, by which we mean that it does not depend on the variable $\poly_x$.
In order to be self-contained, we argue below in Section \ref{subsec:renorm} that these two structural assumptions and \eqref{jc60a} are realistic.

\medskip

We denote by $\|\cdot\|$ the supremum norm.  We use $\di$ for dimension, $\el$ for an ellipticity constant, and  $\alpha\in (\frac{2}{\nn+1},\frac{2}{\nn})$ for the H\"{o}lder exponent of the solution $u$.
A constant is said to be universal provided it depends only on $\di$, $\nn$, and $\el$.  The notation $A \lesssim B$ indicates an inequality that holds up to a universal constant.
The symbols $\vee$ and $\wedge$ indicate $\max$ and $\min$, respectively.
\begin{mainthm}\label{theorem1}
Let $\alpha \in (\frac{2}{\nn+1},\frac{2}{\nn})$ for some $\nn \in \N$, $\el>0$ and $a \in C^{\nn}(\R)$ satisfy $ \el  \leq a \leq \el^{-1}$ together with $\|a^{(k)}\| \leq \el^{-1} $ for $1 \leq k \leq \nn$.
Let $\xi$ satisfy Assumptions \ref{sne1} and \ref{jc73} for some $q\in\TT$.
There exists a universal constant $\eps>0$ and a function $h:\R\to \R$ depending only on $a$ and $q$
such that all smooth solutions $u:\R^{\di+1} \to \R$ to the renormalized PDE
\begin{equation}
\partial_{\tv}u- a(u) \Deltap u + h(u)=\xi \quad \mbox{ on } B_1(0)\subset \R^{\di+1} \label{intro3}
\end{equation}
with $\|u\|\le \eps$ satisfy for all $r\in(0,1)$ and all $\txv,\syv\in B_{1-r}(0)$ with $\txv\ne \syv$ the interior
H\"{o}lder bound
 \begin{align}\label{jc63}
 r^\alpha|u(\syv)-u(\txv)| \lesssim (\|u\|+[\Pi])(1 \vee [\Pi])d^{\alpha}(\syv,\txv).
\end{align}
\end{mainthm} 
The \autoref{theorem1} holds in the full sub-critical regime $\alpha \in (0,1)$ and provides bounds on $u$ which are independent of the possibly divergent constants hidden in the counter-term $u \mapsto h(u)$, which is local and identified explicitly, see \eqref{jc66} below.  En route to \eqref{jc63} we establish a much stronger bound in the flavor of controlled rough paths, which plays the role of a higher regularity theory in the setting of singular SPDE, see \eqref{jc72} below.  The most substantial difference with our prior work \cite{otto2018parabolic} is that we need to identify a suitable algebraic structure to support our local description of $u$, which becomes increasingly refined as the parameter $\alpha$ approaches zero.  This algebraic machinery is a central ingredient that must be combined in a rather delicate way with the analytical tools developed in \cite{otto2018parabolic}.  Our approach is self-contained and we believe our methods are quite robust, potentially adding a valuable alternative perspective even in the context of semi-linear equations.

\medskip

\medskip

The renormalization of $a(u)\Deltap u$ involves counter-terms which are products of derivatives of $u \mapsto a(u)$ with 'renormalization constants' that depend on the forcing $\xi$. It will follow from the proof of the \autoref{theorem1} that these `renormalization constants' are collected precisely in $q\in\Tplus$ appearing in Assumption \ref{jc73} through \eqref{jc60} (see also Section \ref{subsec:renorm}, where we argue why this form of renormalization is to be expected). To be more specific, we encode the products of derivatives of $a$ by introducing
\begin{align*}
da(v):=\bigg(\frac{1}{k!}\frac{d^ka}{du^k}(v) \bigg )_{k \in\N}
\end{align*}
and use the following shorthand notation: We write for
$\beta=(\beta_x,\beta')$
\begin{equation}
da(v)^{\beta'}:=\prod_{k \geq 1}\bigg (\frac{1}{k!}\frac{d^ka}{du^k}(v) \bigg )^{\beta(k)},\label{intro1}
\end{equation} 
and introduce a scaled norm of such a multi-index as follows
\begin{equation}
|\beta|_{s}:=\sum_{k \geq 1}k \beta(k).\label{intro2}
\end{equation}
We will show that the renormalization $h:\R\to\R$ appearing in the \autoref{theorem1} is given by
 \begin{align}\label{jc66}
  h(v):=\sum_{|\beta|_{s}=0}^{\nn-1} da(v)^{\beta'} q_{\beta'}(a(v)),
 \end{align}
where we recall that $q\in \TT$ depends on a variable $a_0$ by the definition of $\TT$.

Estimate \eqref{jc63} is only the lowest of a whole hierarchy of estimates resembling the controlled rough path condition in \cite[Definition 1]{Gub04}. In fact, we will
show that the functions $\llhs_{\txv}$ describe the solution close to $\txv\in\R^{\di+1}$ to any order $\kappa<2$, in the sense that for all $r\in(0,1)$ and all $\txv,\syv\in B_{1-r}(0)$ there holds
 \begin{align}\label{jc72}
 r^\kappa|u(\syv)-u(\txv)-\sum_{|\beta|<\kappa} \nu^{\betan}(\txv) da(u(\txv))^{\beta'} \llhs_{\txv\beta}(\syv;a(u(\txv)))|
 \lesssim (\|u\|+[\Pi])(1 \vee [\Pi]^{\frac{\kappa}{\alpha}})d^{\kappa}(\syv,\txv),
 \end{align}
 where the Gubinelli derivative $\nu$ is given by
 \begin{align}\label{jc71}
  \nu(\txv):=\nablap u(\txv)-\sum_{|\beta|<1}da(u(\txv))^{\beta'} \nablap\llhs_{\txv\beta}(\txv;a(u(\txv))).
 \end{align}
 Here $\llhs_{\txv\beta}(\syv;a_0)$ denotes the coefficient of $\llhs_{\txv}(\syv)\in \TT$ in front of $\poly^\beta$ evaluated at $a_0$.

\medskip

Since the first version of this work appeared on arXiv, there have been some further developments in this direction.  In particular, the paper \cite{bailleul2023regularity} studies renormalization of a class of quasi-linear SPDEs containing \eqref{eq:PDEintro}, focusing on the initial value problem for small times.  An advantage of \cite{bailleul2023regularity} is that the authors are able to provide an existence and uniqueness theory for the SPDE.  However, the result in \cite{bailleul2023regularity} is conditional on the existence and continuity properties of a suitable model.  The model involves heat kernels which are not translation invariant and have limited regularity near the initial time, which rules out a direct application of the results of \cite{chandra2016analytic}, so presently the hypotheses in \cite{bailleul2023regularity} have not been verified.  

\medskip

The approach in \cite{bailleul2023regularity} has a number of similarities with this paper and our previous work \cite{otto2018parabolic}.
The main departure from our methodology is that the `freezing in' of the quasi-linear term $a(u)\Delta u$ is performed globally, with respect to a reference function approximating the initial condition (similar to  \cite{BDH19}, which implements this with paracontrolled calculus).
This allows the authors to apply Hairer’s analytic results \cite{Hai13} to close a fixed point argument in a space of modelled distributions and then recover the counter-terms by modifying the arguments of \cite{bailleul2021renormalised}.
The choice of reference function has some collateral damage in terms of the form of counter-terms in the renormalized equation, which the authors can mitigate on a case by case basis.
The arguments in \cite{bailleul2023regularity} are an instance of the traditional bottom-up approach to singular SPDE via a tree-based model.  

\medskip

In contrast, our work introduces a new regularity structure and a top-down methodology.
In particular, we show that by indexing the local expansion more efficiently and defining the right structure group, the correct counter-terms in the PDE appear automatically.
This viewpoint is not limited to quasi-linear equations, and leads to a particularly transparent simplification in settings like KPZ or $\Phi^{4}$ where one can index the model (and corresponding local expansion) simply by powers of the coupling constant and polynomials.
Finally, we mention that it does not seem clear how to apply the approach of \cite{bailleul2023regularity} on a reference domain of a fixed size, which is the setup of our main result.
In fact, in order to achieve this in the present work, even for solutions small in the supremum norm, we require more refined estimates on the solution (closer to what is needed for global bounds), which is a key reason why we don't attempt to apply directly the results of \cite{Hai13} as in the approach of \cite{bailleul2023regularity}. 

\subsection{A Formal Series Expansion for the Solution}\label{sec:series}In this section, we motivate our algebraic objects and local expansion by a formal analysis of the PDE \eqref{eq:PDEintro} without renormalization.  The discussion below is completely formal, as it involves manipulating formal power series which are not expected to converge.  Let us freeze in the coefficients at a base point $\mathbf{x}$ and re-write \eqref{eq:PDEintro} as
\begin{equation}
(\partial_{t}-a_{0}\Delta) u =(a(u)-a_{0})\Delta u+\xi \label{q1},
\end{equation}
where $a_{0}:=a(u(\txv))$. Our goal is to find a formal series solution to \eqref{q1} of the form
\begin{equation}
u-u(\txv) =f(\txv).\hat{\Pi}_{\txv} \label{q2}.  
\end{equation}
We will demand that $\tau \in \TT \mapsto f(\txv).\tau$ is linear and has the morphism property: for $\sigma,\tau \in \TT$
$$
f(\txv).(\sigma\tau)=(f(\txv).\sigma)(f(\txv).\tau), 
$$
as well as $f(\txv).\poly_{k}=\frac{1}{k!}a^{(k)}(u(\txv))$.  Similar to $\Pi_{\mathbf{x}}$, one should think of $\hat{\Pi}_{\txv}$ as a $\TT$-valued function, the hat being used to distinguish the two since at this stage there will be no renormalization.  Our goal is to show that formally, $u$ satisfies \eqref{q1} provided that the coefficients of $\hat{\Pi}_{\mathbf{x}}$ satisfy a certain hierarchy of PDE's, see \eqref{q5} below.  In the next section, we will explain that by adjusting this hierarchy slightly via a suitable renormalization, we are led to a definition of $\Pi_{\mathbf{x}}$ which has been shown in \cite{LOTT21} to satisfy the bounds imposed in the present work.  We now turn to the calculation.   
First note that considering the left-hand side of \eqref{q1}, by linearity we clearly have that for $u$ satisfying \eqref{q2} it holds
\begin{equation}
(\partial_{t}-a_{0}\Delta ) u=f(\txv).(\partial_{t}-a_{0}\Delta)\hat{\Pi}_{\txv}.
\end{equation}
Furthermore, turning to the right-hand side of \eqref{q1} and applying $\Delta$ to \eqref{q2} together with Taylor's formula, 
\begin{align}
    (a(u)-a_{0})\Delta u&=\sum_{k \geq 1}\frac{1}{k!}a^{(k)}(u(\txv))(u-u(\txv))^{k}f(\txv).\Delta \hat{\Pi}_{\txv} \nonumber\\
    &=\sum_{k \geq 1}\big ( f(\mathbf{x}).\polv_{k} \big )\big (f(\mathbf{x}).\hat{\Pi}_{\mathbf{x}} \big )^{k}f(\txv).\Delta \hat{\Pi}_{\txv}
    =f(\txv).\sum_{k \geq 1}\poly_{k}\hat{\Pi}_{\txv}^{k}\Delta \hat{\Pi}_{\txv},\label{q7}
\end{align}
where we used the morphism property in the last step.  Note also that for $\na \in \TT$ being the constant power series with value $1$, it holds $f(\txv).\na=1$ for any $\txv\in\R^{\di+1}$.  Hence, matching the terms we see that \eqref{eq:PDEintro} holds provided that
\begin{equation}\label{jq004}
    (\partial_{t}-a_{0}\Delta) \hat{\Pi}_{\txv}=\sum_{k \geq 1}\poly_{k}\hat{\Pi}_{\txv}^{k}\Delta \hat{\Pi}_{\txv}+\xi \na.
\end{equation}
At the level of components, recalling the componentwise definition of multiplication of power series, this reads as
\begin{equation}
(\partial_{t}-a_{0}\Delta)\hat{\Pi}_{\txv\beta}=\hat{\Pi}_{\txv\beta}^{-} \nonumber,
\end{equation}
where
\begin{equation}
\hat{\Pi}_{\txv\beta}^{-}:=
\begin{cases}
\xi &\text{if} \ \beta=0 \label{q5} \\
{\displaystyle \sum_{k \geq 1}\sum_{\beta_{1}+\dots+\beta_{k+1}+e_{k}=\beta} \hat{\Pi}_{\txv\beta_{1}} \cdots \hat{\Pi}_{\txv\beta_{k}} \Delta \hat{\Pi}_{\txv\beta_{k+1}}} \ &\text{if} \ \beta \neq 0. 
\end{cases}
\end{equation}
Note that the sum in $k$ is effectively finite due to the appearance of $e_k$ in the second sum.
Here $e_{k}$ is a multi-index with $1$ in component $k$ and zero in all other components.
Note that with this definition, the identity \eqref{q7} turns into
\begin{equation}
(a(u)-a_{0})\Delta u+\xi=f(\txv).\hat{\Pi}_{\txv}^{-}.\label{q12}
\end{equation}
The similarity between \eqref{q2} and \eqref{q12} is effectively the reason why it suffices to use a single index set for both the positive and the negative model simultaneously.
\begin{remark}
A simple way to achieve the  properties of $\tau \mapsto f(\mathbf{x}).\tau$ demanded above is as follows.  Define the linear form $\tau \in \TT \mapsto f(\txv).\tau$ by evaluating the formal power series $\tau$, by setting the abstract variables $\poly$ to be (an \textit{a priori} unknown function) $\poly_x=\nu(\txv)$ and $\poly_k=\frac{1}{k!}\frac{d^ka}{du^k}(u(\txv))$. That is, writing $\tau=\tau(\poly_x,(\poly_k)_{k\ge 1};a_0)$ we introduce
\begin{align}\label{jq001}
    f(\txv).\tau:=\tau(\nu(\txv),(\frac{1}{k!}\frac{d^ka}{du^k}(u(\txv)))_{k\ge 1};a(u(\txv))).
\end{align} 
Note that this is in line with the motivation given for our model space $\TT$, cf.~\eqref{rr08} and \eqref{rr09}, where now we use a general base point $\txv$ rather than fixing the origin arbitrarily. 
We remark again that this definition is only formal, as the evaluation of a formal power series is unlikely to converge.  Ignoring this issue, we see that such a form $f(\txv)$ clearly has the morphism property $(f(\txv).\sigma)(f(\txv).\tau)=f(\txv).(\sigma\tau)$.
 We also emphasize that this reasoning shows that evaluation is the driving principle behind the definition of modelled distributions: it takes elements from the abstract model space to concrete objects.
\end{remark}

We now turn our attention to the structure group $\tau \mapsto \hat{\Gamma}_{\syv\txv}\tau$ and claim that its definition arises naturally from 
demanding that the modelled distribution $f(\mathbf{x})$ has the following (formal) covariance property:
\begin{equation}
f(\syv).\tau=f(\txv).\hat{\Gamma}_{\syv\txv}\tau \label{sd1}
\end{equation}
for all $\txv,\syv\in\R^{\di+1}$ and all $\tau\in\TT$.  Before deriving the action of $\hat{\Gamma}_{\syv\txv}$, we start by giving a motivation for \eqref{sd1} based on demanding the consistency of the local expansion \eqref{q2} across two different base-points, together with the re-expansion property \eqref{jc15}. In fact, for two points $\txv,\syv\in\R^{\di+1}$, using the relation \eqref{q2} first at $\syv$ and then at $\txv$ yields
\begin{align}\label{jq002}
    f(\syv).\hat{\Pi}_{\syv}&=u-u(\syv)=u-u(\txv)-(u(\syv)-u(\txv)) 
    =f(\txv).(\hat{\Pi}_{\txv}-\hat{\Pi}_{\txv}(\syv))=f(\txv).\hat{\Gamma}_{\syv\txv}\hat{\Pi}_{\syv},
\end{align}
where we used \eqref{jc15} in the last step.
The reader acquainted with regularity structures will notice that Hairer's definition of a modelled distribution $f$ is a quantification of the defect in \eqref{sd1} for $\txv$ and $\syv$ close together.  Since $f$ describes a formally exact solution $u$, it's reasonable to expect an identity rather than an inequality.

\medskip

We will now use \eqref{sd1} and the definition \eqref{jq001} to deduce the action of $\hat{\Gamma}_{\syv\txv}$ on $\poly_j$.  We take the time to spell this calculation out explicitly since the manipulation is generally very similar to the key ideas carried out in the proof of the \autoref{lem1}, a core point of the present work.  Applying Taylor's formula yields
\begin{align*}
f(\syv).\poly_{j}&=\frac{1}{j!}a^{(j)}(u(\syv)) =\sum_{k\geq j}\frac{1}{j!}(u(\syv)-u(\txv))^{k-j}\frac{1}{(k-j)!}a^{(k)}(u(\txv)).
\end{align*}
Inserting \eqref{q2} and using the morphism property of $f$, we deduce
\begin{align*}
f(\syv).\poly_{j}&=\sum_{k \geq j}\frac{1}{j!}(f(\txv).\hat{\Pi}_{\txv}(\syv))^{k-j}\frac{1}{(k-j)!}a^{(k)}(u(\txv) ) \\
&=\sum_{k \geq j}{\binom{k}{j}}  (f(\txv).\hat{\Pi}_{\txv}(\syv) )^{k-j}f(\txv).\poly_{k} =f(\txv).\sum_{k \geq j}{\binom{k}{j}}  \hat{\Pi}_{\txv}(\syv) ^{k-j}\poly_{k}.
\end{align*}
Since $f(\txv)$ is linear, we see that ensuring \eqref{sd1} amounts to defining
\begin{equation}
\hat{\Gamma}_{\syv\txv} \poly_{j}:=\sum_{k \geq j}{\binom{k}{j}}  \hat{\Pi}_{\txv}(\mathbf{y} )^{k-j}\poly_{k} \label{sss13}. 
\end{equation} 

\subsection{The Renormalized Model} \label{subsec:renorm}

Accepting that some form of counter-term in \eqref{eq:PDEintro} is necessary, we aim to choose it in a minimally intrusive way.
This is in line with the axiomatic approach common in the physics community,
see also \cite{LOTT21}.    In particular, we ask that it is scaling wise lower order and respects the symmetries of the SPDE and its solution manifold, which then restricts the possible functional dependence.
Hence in view of \cite{bruned2020renormalising}, the appropriate ansatz for a counter term in equation 
(\ref{eq:PDEintro}) is 
\begin{align}\label{rr01}
\partial_t u-a(u)\Delta u+ h(u) + H(u) \cdot \nabla u=\xi.
\end{align}
The ansatz (\ref{rr01}) is thus parameterized by the
nonlinearities $h$ and $H$ that one postulates to be deterministic, i.e., dependent only on the
law but not the realizations of $\xi$.
In making the ansatz (\ref{rr01}), we have in mind a $\xi$ whose law is  invariant
under shifts of space-time and implicitly 
use symmetry as follows.  As the original
nonlinear operator $u\mapsto\partial_t u-a(u)\Delta u$ does not depend explicitly on space-time, we may assume that the same is true for $h$ and $H$.
There is a further symmetry-related reduction: we demand that if the law of $\xi$ is invariant under the spatial reflection $x_i\mapsto -x_i$ for $i\in\{1,\ldots,\di\}$, then the same is true for the solution $u$. Since $u\mapsto\partial_t u-a(u)\Delta u$
commutes with spatial reflection, this requires $H\equiv 0$.  Hence
(\ref{rr01}) collapses to \eqref{intro3}. 

\medskip

In fact, there is a final, but crucial symmetry observation related to the (functional)
dependence of the function $h=h[a](u)$ of $u\in\mathbb{R}$ on the nonlinearity $a$. For this, recall that the solution manifold has the important shift invariance \eqref{rr16}.
By our principle of minimal intrusiveness regarding the symmetries of the SPDE we therefore assume the following covariance under shifts
\begin{align}\label{rr04}
h[a(\cdot+v)](u)=h[a](u+v).
\end{align}
This implies that $h$ is determined by a functional $q=h[\cdot](0)$ on the space of nonlinearities via
\begin{align}\label{rr11}
h[a](v)=q[a(\cdot+v)].
\end{align}
At least heuristically, there is a one-to-one correspondence between functionals $a\mapsto \tau[a]$ and $\mathsf{S}:=\setc{\tau\in\TT}{\tau \text{ independent of }\poly_x}$.
In particular, \eqref{rr11} can be recast for fixed $a$ as
\begin{align}\label{jq003}
    h(v)=g(v).q, \quad v\in\R,
\end{align}
where $\mathsf{S}\ni\tau\mapsto g(v).\tau$ acts via $g(v).\tau=\tau[a(\cdot+v)]$.
On the one hand, \eqref{jq003} is a non-truncated version of \eqref{jc66}, as can be seen via \eqref{rr08}, \eqref{intro1} and
\begin{align*}
    \frac{1}{k!}\frac{d^ka(\cdot+v)}{du^k}(0)=\frac{1}{k!}\frac{d^ka}{du^k}(v), \qquad \mbox{for } k\in\N_0.
\end{align*}
On the other hand, Taylor's formula gives for $\txv\in\R^{\di+1}$
\begin{align*}
\begin{split}
h(u)&=\sum_{k\ge 0} \frac{1}{k!} (u-u(\txv))^k (\partial_u^k)|_{u=u(\txv)} h=\sum_{k\ge 0}\frac{1}{k!}(u-u(\txv))^k(\partial_u^k)|_{u=u(\txv)} g.q \\
&=\sum_{k\ge 0}\frac{1}{k!}(u-u(\txv))^k g(u(\txv)).(D^{(0)})^kq,
\end{split}
\end{align*}
where the derivation $D^{(0)}$ is the infinitesimal generator of shifts on the algebra $\TT$ determined via $\partial_u g.\tau=g.D^{(0)}\tau$ for all $\tau\in\TT$.
We remark that $D^{(0)}$ is given by \eqref{as40} as shown in Lemma \ref{lem:uspace} below.
Furthermore, by \eqref{jq001} we formally obtain $g(u(\txv)).\tau=\tau[a(\cdot+u(\txv))]=f(\txv).\tau$ for $\tau\in \mathsf{S}$.
Using \eqref{q2} and the morphism property of $f(\txv)$, we therefore expect
\begin{align*}
    h(u)=\sum_{k\ge 0}\frac{1}{k!}(f(\txv).\llhs_{\txv})^k f(\txv).(D^{(0)})^kq = f(\txv).\sum_{k\ge 0}\frac{1}{k!}\llhs_{\txv}^k (D^{(0)})^kq.
\end{align*}
Following the same steps as in Section \ref{sec:series}, we see that if we demand
\begin{align*}
    (\partial_{t}-a_{0}\Delta) \llhs_{\txv}=\sum_{k \geq 1}\poly_{k}\llhs_{\txv}^{k}\Delta \llhs_{\txv} - \sum_{k\ge 0}\frac{1}{k!}\llhs_{\txv}^k (D^{(0)})^kq +\xi \na =:\lrhs_\txv.
\end{align*}
then $u$ given by \eqref{q2} is formally a solution to the renormalized equation \eqref{intro3}.
In particular, keeping in mind that $\llhs_{\txv}(\txv)=0$, we expect \eqref{jc60a}.

\medskip

Now that we have finished motivating our assumptions, we comment on the precise connection with the work \cite{LOTT21} where the model is constructed.  Note that in \cite{LOTT21}, the letter $c$ is used in place of our $q$.  In particular,
\begin{itemize}
    \item analytical dependence on the parameter $a_0$ of all relevant objects as imposed by the definition of $\TT$ is obtained in \cite{LOTT21} in Remark 2.7;
    \item \eqref{jc03b} corresponds to (2.21) in \cite{LOTT21} in the relevant case $\beta_x\in\set{0,1}$,
    \item \eqref{jc60} corresponds to (2.35) in \cite{LOTT21}, where we spell out explicitly the $(\beta=0)$-component. Indeed, the $(\beta=0)$-component of (2.18) in \cite{LOTT21} reads $\lrhs_{\txv 0}=\xi-q_0$, which is \eqref{jc60} up to a constant $q_0$ which we can allow for in view of the footnote to \eqref{jc60};
    \item \eqref{model7a} corresponds to (2.36) in \cite{LOTT21};
    \item the explicit form of elements in $\mathsf{G}$ postulated in \eqref{as41} corresponds to (2.44) in \cite{LOTT21};
    \item \eqref{jc15} corresponds to (2.61) in \cite{LOTT21};
    \item \eqref{model101} corresponds to (2.55) in \cite{LOTT21};
    \item and \eqref{jc60a} corresponds to evaluating (2.18) in \cite{LOTT21} at the base point $\txv$ and observing that the sums collapse due to (2.36) in form of $\llhs_{\txv\beta}(\txv)=0$.
\end{itemize}

\section{Model Space and Structure Group} \label{sec:model}
In this section, we introduce the algebraic framework which underlies our local expansion for the solution and is used to quantify our assumptions on the forcing $\xi$.
For a multi-index $\betaa=(\betaa_1,\ldots,\betaa_{\di})$ we use the standard notation
\begin{align*}
 \betaa!&:=\betaa_1!\cdots\betaa_{\di}!, \quad |\betaa|:=\betaa_1+\ldots+\betaa_{\di}, \quad x^{\betaa}:=x_1^{\betaa_1}\cdots x_{\di}^{\betaa_{\di}},
\end{align*}
with the convention that $0^{0}=1$.
\subsection{The Model Space}\label{sec:ms}
Recall the definition of the model space $\Tplus$ as the linear space of formal power series $\sum_\beta \tau_\beta \polv^\beta$ in the abstract variables $\polv=(\polv_x,\polv_{1},\polv_{2},\ldots ) \in \R^{\di} \times \R^{\mathbb{N} }$.  It will be important that $\Tplus$ forms an algebra with unit element $\na$ defined via $\na(\polv):=1$, and given $\sigma,\tau \in \Tplus$ the product $\sigma\tau$ is identified with its coefficients via
\begin{equation}
(\sigma\tau)_{\beta}:=\sum_{\beta_{1}+\beta_{2}=\beta}\sigma_{\beta_{1}}\tau_{\beta_{2}}\label{model40}.
\end{equation}
A special role is played by the monomials $\polv_x,\{\poly_{j}\}_{j \geq 1}$, and in addition we define $\polv_0:=a_0\na$.  It will be convenient to separate the polynomial sector $\Pminus$ of $\Tplus$ from the rest, that is we write
\begin{align*}
 \Tplus=\Pminus \oplus \Tminus
\end{align*}
with
 \begin{align*}
  \Pminus&:=\{\tau\in\Tplus:\tau_{(\betan,\beta')}=0 \quad \mbox{unless}\; \betan\ne 0,\beta'=0\}, \\
  \Tminus&:=\{\tau\in\Tplus:\tau_{(\betan,0)}=0 \quad \mbox{for all}\; \betan\ne 0\}.
 \end{align*}
We denote the projection of $\Tplus$ onto $\Pminus$ by $\projm$.  Notice that in particular $\na \in \Tminus$, which should be compared with the $\beta=0$ constraint of \eqref{jc60}, and we warn the reader that we are departing from the notational convention in \cite{Hai14}. 

\medskip

\subsection{Homogeneities}\label{sec:hom}
We now define a grading of $\TT$ by assigning a homogeneity to each elementary monomial $\polv^{\beta}$, or equivalently to each multi-index $\beta$. Specifically, recalling \eqref{intro2} we define an integer-valued function $\beta \mapsto \langle \beta \rangle$ by
\begin{equation}
\langle \beta \rangle:=|\beta|_{s}+1_{\beta_{x}=0}, \label{model200}
\end{equation}
cf.~\eqref{intro2} and use this to define the homogeneity
\begin{equation}
|\beta|:=\langle \beta \rangle \alpha + |\betan|. \label{model201}
\end{equation}
\begin{remark} \label{rem:homog}
We want to emphasize that we are not making any departures from the traditional homogeneity counting in regularity structures.  For example, for $\beta_{x} \in \N_{0}^{d}$ with $|\beta_{x}|=1$ we have $|(\beta_{x},0)|=1$ to be compatible with \eqref{jc03b} and $|0|=\alpha$ which should be compared with the $\beta=0$ constraint of \eqref{jc60}, keeping in mind $\xi \in C^{\alpha-2}$.
\end{remark}

\begin{remark}
The definition can also be motivated by a scaling argument.  Indeed, observe that the $C^{\alpha-2}$-norm of $\xi$ is invariant under the scaling $\xi(\txv)\mapsto \tilde\xi(\txv):=\mop^{\alpha-2}\xi(\tilde\txv)$, $\tilde\txv:=(\mop^{-2}\tv,\mop^{-1}\xv)$. Such scaling leaves \eqref{eq:PDEintro} invariant if we define $\tilde u(\txv):=\lambda^\alpha u(\tilde \txv)$ and $\tilde a(v)=a(\lambda^{-\alpha} v)$. In view of \eqref{rr08} and \eqref{rr09}, this leads to the rescaled coordinates $\tilde\polv:=(\mop^{\alpha-1}\polv_x, \mop^{-\alpha}\polv_1, \mop^{-2\alpha}\polv_2,\ldots)$. Thus, for the partial derivatives $\tilde\llhs_\beta$ of $\tilde u$ with respect to $\tilde\poly$ we have
\begin{align*}
\tilde\llhs_\beta(\txv)=\partial_{\tilde\polv}^\beta|_{\tilde\polv=0} \tilde u(\txv) = \mop^{\alpha + \alpha\sum_{k\ge 1}k\beta(k) + (1-\alpha)|\beta_x|}\partial_\polv^\beta|_{\polv=0} u(\tilde\txv) = \mop^{|\beta|} \llhs_\beta(\tilde\txv),
\end{align*}
which is exactly corresponding to \eqref{model200} in the relevant case $|\beta_x|=0,1$.
\end{remark}
This assignment of homogeneities naturally generates a finite set of homogeneities 
$$
\Aa{}:=\big \{ |\beta|<2 \big \} \subset  \N_{0}\alpha+\N_{0}.
$$ 
The reader should keep in mind that the least element of $\Aa{}$ is $\alpha$, the homogeneity of the multi-index $\beta=0$, cf.~Remark \ref{rem:homog}.  A distinguished role will also be played by $ \lceil \alpha^{-1} \rceil \alpha$,
the least homogeneity in $\Aa{}$ larger than 1.   Moreover we note that if $\beta$ and $\gamma$ are such that $|\beta|=|\gamma|\in\Aa{}$, then the choice of $\alpha$ implies that $|\beta|_s=|\gamma|_s$, $\langle\beta\rangle=\langle\gamma\rangle$, and $\beta_{x}=\gamma_{x}$. The notion of homogeneity leads to a grading on $\TT$ as follows: for each $\delta \in \Aa{}$ we define
\begin{align*}
\TT_{\delta}&:=\{\tau \in \TT : \, \, \tau_{\beta}=0 \quad \text{if} \quad |\beta| \neq \delta \}, \quad  \quad 
\TT_{\geq \delta }:=\bigcup_{|\beta| \geq \delta}\TT_{|\beta|},
\end{align*}
and analogously $\TT_{>\delta}$.  It will be important to keep in mind how the multiplication of power series in the sense of \eqref{model40} interacts with the grading resulting from the definition \eqref{model201}.  The reader should be careful to note that although $|\beta_{1}+\beta_{2}|_{s}=|\beta_{1}|_{s}+|\beta_{2}|_{s}$, in most cases of interest in this article, $|\beta_{1}+\beta_{2}| \neq |\beta_{1}|+|\beta_{2}|$.  This is easily seen to be false by considering $\beta_{1}=\beta_{2}=0$.  For the typical cases that interest us, the two sides differ by $\alpha$, and for convenience we record this in the following lemma.
\begin{lemma}\label{lem_hom}
Let $\beta_{1},\beta_{2}$ be multi-indices.
\begin{enumerate}
\item The following identity holds:
\begin{equation} \label{jc28}
|\beta_{1}+\beta_{2}|=
\begin{cases}
|\beta_{1}|+|\beta_{2}|-\alpha \quad  &\text{if} \quad |\beta_{1,x}|\cdot |\beta_{2,x}|=0   . \\
|\beta_{1}|+|\beta_{2}| \quad &\text{else}
\end{cases}
\end{equation}
\item The following implication holds:
\begin{align}\label{prodHomog}
 \left.\begin{array}{rl}
        \sigma &\in \Ta{\ge |\beta_{1}|} \\
        \tau  &\in \Ta{ \ge |\beta_{2}|}
       \end{array}\right\}
       \Rightarrow \sigma \tau \in \Ta{ \ge |\beta_{1}|+|\beta_{2}|-\alpha} \subset \Ta{\ge |\beta_{1}|\vee |\beta_{2}|}.
\end{align}
Moreover, if $|\beta_{1,x}|\cdot |\beta_{2,x}|=0$, then
\begin{align}\label{prodHomog2}
 \left.\begin{array}{rl}
        \sigma &\in \Ta{|\beta_{1}|} \\
        \tau  &\in \Ta{|\beta_{2}|}
       \end{array}\right\}
       \Rightarrow \sigma \tau \in \Ta{|\beta_{1}|+|\beta_{2}|-\alpha}.
\end{align}
\end{enumerate}
\end{lemma}
\begin{proof}
To see the identity \eqref{jc28}, notice that
\begin{align}
|\beta_{1}+\beta_{2}|&=\alpha \big ( |\beta_{1}|_{s}+|\beta_{2}|_{s} +1_{\beta_{1,x}+\beta_{2,x}=0 } \big )+|\beta_{1,x}|+|\beta_{2,x}| \nonumber \\
&=|\beta_{1}|+|\beta_{2}|+\alpha \big (1_{\beta_{1,x}+\beta_{2,x}=0 } -1_{\beta_{1,x}=0 }-1_{\beta_{2,x}=0 }\big ) \nonumber.
\end{align}
Hence, if at least one of $\beta_{1,x},\beta_{2,x}$ is zero, then $|\beta_{1}+\beta_{2}|=|\beta_{1}|+|\beta_{2}|-\alpha$ and otherwise all indicator functions above vanish and $|\beta_{1}+\beta_{2}|=|\beta_{1}|+|\beta_{2}|$.  This establishes \eqref{jc28}.  
We now turn to the implication \eqref{prodHomog} and suppose $\sigma \in \TT_{ \geq |\beta_{1}|}$, $\tau \in \TT_{ \geq |\beta_{2}| }$.  We now argue that $(\sigma \tau)_{\gamma}=0$ if $|\gamma|<|\beta_{1}|+|\beta_{2}|-\alpha$.  Indeed, keeping in mind \eqref{model40}, if $\gamma_{1}+\gamma_{2}=\gamma$, then by \eqref{jc28} it holds $|\gamma_{1}|+|\gamma_{2}| \leq |\gamma|+\alpha<|\beta_{1}|+|\beta_{2}|$.  Hence, $|\gamma_{i}|<|\beta_{i}|$ for at least one of $i=1,2$, which implies that $\sigma_{\gamma_{1}}\tau_{\gamma_{2}}=0$, yielding the claim.  The inclusion $\Ta{ \ge |\beta_{1}|+|\beta_{2}|-\alpha} \subset \Ta{\ge |\beta_{1}|\vee |\beta_{2}|}$ follows immediately since all homogeneities in $\Aa{}$ are at least $\alpha$.
To show \eqref{prodHomog2}, let $\sigma \in \TT_{ |\beta_{1}|}$, $\tau \in \TT_{ |\beta_{2}| }$ and consider $(\sigma\tau)_\gamma$, where $\gamma=\gamma_{1}+\gamma_{2}$.  Then $(\sigma \tau)_{\gamma}=0$ unless $|\gamma_1|=|\beta_1|$ and $|\gamma_2|=|\beta_2|$, and hence $|\gamma_{1,x}|\cdot|\gamma_{2,x}|=|\beta_{1,x}|\cdot |\beta_{2,x}|=0$.
Thus, by \eqref{jc28} it holds $|\gamma|=|\gamma_1|+|\gamma_2|-\alpha=|\beta_1|+|\beta_2|-\alpha$.
\end{proof}

\medskip

\subsection{The Structure Group}\label{sec:sg}
We now define a subgroup $\mathsf{G}$ of the linear endo-morphisms $\Gamma$ of $\Tplus$. Namely, each $\Gamma$ is required to be of the ``exponential" form
\begin{align}\label{as41}
\Gamma=\sum_{k,|\betaa|\ge 0}\frac{1}{k!\betaa!} \primm{k}{\betaa} \DD{k}{\betaa},
\end{align}
where $(\primz,\primo) \in \Tplus\times\Tplus^{\di}$.\footnote{At this stage, $\tau^{(0)},\tau^{(1)}$ are arbitrary since we are describing a generic group element.  Given base-points $\txv,\syv$ we write $\tau^{(0)}_{\syv\txv}$,$\tau^{(1)}_{\syv\txv}$ for the specific choices leading to the group element $\Gamma_{\syv\txv}$ described in Assumption 2.}
Here, we have used the notation
\begin{align}\label{jc38}
 \primm{k}{\betaa}:=(\primz)^{k}(\primo)^{\betaa}, \quad \mbox{and} \quad \DD{k}{\betaa}:=(D^{(0)})^k(D^{(1)})^\betaa,
\end{align}
with the linear operators
$D^{(0)}$ and $D^{(1)}$ given by
\begin{align}\label{as40}
D^{(0)}:=\polv_1\partial_{a_0}+\sum_{k=1}^\infty(k+1)\polv_{k+1}\partial_{\polv_k},
\quad D^{(1)}:=\nablap_{\polv_{x}}.
\end{align}
These are ``derivations'' in the sense that they satisfy 
\begin{align}\label{as28}
D(\tau\sigma)=(D\tau)\sigma+\tau(D\sigma)\;\mbox{for all}\;\tau,\sigma\in\Tplus,
\end{align}
in particular $D\mathsf{1}=0$.
It will be convenient to record their value on the linear monomials (recall $\polv_0=a_0\na$):
\begin{align}
\begin{cases} 
D^{(0)}\poly_{j}=(j+1)\poly_{j+1} &\mbox{for}\;j\in\N_0, \\
D^{(0)}\poly_{x}=0,
\end{cases}
\label{as30}
\end{align}
and for $0\ne\betaa\in\N_0^{\di}$
\begin{align}
\begin{cases}
(D^{(1)})^{\betaa}\poly_{j}=0 &\mbox{if}\;j\in\N_0, \\
(D^{(1)})^{\betaa}\poly_{x}^\betaa=\betaa!\na.
\end{cases}
\label{as37}
\end{align}
In particular, if we introduce for $\tau=(\tau^1,\ldots,\tau^{\di})\in\Tplus^{\di}$ the notation $\Gamma \tau:=(\Gamma\tau^1,\ldots,\Gamma\tau^{\di})\in\Tplus^{\di}$, then we have
\begin{align}
\Gamma\poly_{j}&=\sum_{k \geq j} \binom{k}{j}(\primz)^{k-j}\poly_{k}, \qquad \mbox{for}\; j\ge 0, \label{model80} \\
\Gamma\poly_{x}&=\poly_{x}+\primo.\label{model81}
\end{align}
By a short calculation using the binomial formula and \eqref{as28}, it follows that $\Gamma$ is
an algebra morphism, that is for $\sigma, \tau\in\Tplus$ we have
\begin{align}
 \Gamma(\sigma \tau)=(\Gamma\sigma)(\Gamma\tau), \quad \Gamma\na=\na. \label{sFin2}
\end{align}
\begin{remark}
The reader might wonder where the definition \eqref{as41} based on \eqref{as40} comes from.  We first note that the relation \eqref{model80} is consistent with our heuristic expectation \eqref{sss13}.  In fact, we could equivalently define $\Gamma$ by starting with \eqref{model80}-\eqref{model81} and extending to the rest of $\TT$ by demanding \eqref{sFin2}.  However, the rather explicit expression \eqref{as40} is useful in the proof of the \autoref{lem1}, where the starting point is a general $\tau \in \TT$.  To motivate the $D^{(0)}$ operator, note that we could alternatively define it by starting with \eqref{as30} and extending to a general $\tau$ by demanding \eqref{as28}.  Furthermore, \eqref{as30} has a simple interpretation: thinking of $\poly_{j}$ as a placeholder for $\frac{1}{j!}a^{(j)}(u)$, cf.~\eqref{rr08}, the $D^{(0)}$ operator simply corresponds to differentiation in $u$.  This will be important in the proof of Corollary \ref{cor:taylor_identity}.
\end{remark}

We mention that the set $\mathsf{G}$ of all $\Gamma$ given in the form \eqref{as41}, where $(\primz,\primo)$ runs through $\Tplus\times \Tplus^{\di}$, forms a subgroup of the endomorphisms on $\TT$. In particular, every $\Gamma\in \mathsf{G}$ is invertible. In the present work, we do not need these properties and hence refer the reader\footnote{We mention for the convenience of the reader two basic transformation rules: if $(\primz,\primo)$ generates $\Gamma$, then $(-\Gamma^{-1}\primz,-\Gamma^{-1}\primo)$ generates $\Gamma^{-1}$. If additionally $({\primz}',{\primo}')$ generates $\Gamma'$, then $(\primz+\Gamma{\primz}',\primo+\Gamma{\primo}')$ generates $\Gamma\Gamma'$, cf.~Proposition 5.1(iii) in \cite{LOT21}.} to \cite{LOT21}, where the group structure is established in a more general situation.

\medskip

Since the coefficients $\tau_\beta$ are analytic in $a_0$, we may estimate higher derivatives with respect to $a_0$ on $D_{|\beta|}$ by lower ones on a larger set, and hence it follows from the definition of the operators $D^{(0)}$ and $D^{(1)}$ in \eqref{as40} and \eqref{jc38}, 
and from the nestedness of the discs $D_{|\beta|}$ that for $|\gamma|\in \Aa{}$, $k\in\N_0$ and $|\betaa|\le 1$ with $|\gamma|+\mu(k,\betaa) \in \Aa{}$, we have
\begin{align}\label{jc36}
 \DD{k}{\betaa}:\TT_{|\gamma|}\to \TT_{|\gamma|+\mu(k,\betaa)}, \quad \mu(k,\betaa):=k\alpha+|\betaa|(\alpha-1), \quad \|\DD{k}{\betaa}\tau\|_{|\gamma|+\mu(k,\betaa)}\lesssim \|\tau\|_{|\gamma|},
\end{align}
where the implicit constant is universal (indeed, it does not depend on the specific $k\in \N_0$ since there are only finitely many $k$ that fulfill the proviso).

\subsection{Linear forms on $(u,\nu)$ space}
We define a family of linear forms parametrized by $\R \times \R^{d}$ which will be used to quantify our estimates on the solution $u$, cf.~\eqref{js100}.  Namely, for each $(u,\nu) \in \R \times \R^{d}$ we define
\begin{equation}
g(u,\nu).\tau:=\sum_{\beta }\nu^{\betan}da(u)^{\beta'}\tau_{\beta}(a(u) ), \label{md0g}
\end{equation}
for all $\tau \in \TT$ which have at most finitely many non-zero coefficients $\tau_{\beta}$. We will often omit the dependence of $g$ on $(u,\nu)$ and simply use the shorthand notation $g.\tau$.  We will make extensive use of the fact that $\tau \mapsto g.\tau$ is an algebra morphism, that is
\begin{equation}
g.(\sigma \tau) =\big ( g.\sigma \big )\big ( g.\tau \big )\label{morphismP}, 
\end{equation}
which can be seen from recalling \eqref{model40} and writing
\begin{align}
g.(\sigma \tau)&=\sum_{\beta}\nu^{\betan}da(u)^{\beta'}\sum_{\beta_{1}+\beta_{2}=\beta} \sigma_{\beta_{1}}(a(u) )\tau_{\beta_{2}}(a(u) )\nonumber \\
&=\sum_{\beta_{1},\beta_{2}}\nu^{\beta_{1,x}}\nu^{\beta_{2,x}}da(u)^{\beta_{1}'}da(u)^{\beta_{2}'}\sigma_{\beta_{1}}(a(u) )\tau_{\beta_{2}}(a(u) )=\big ( g.\sigma \big )\big ( g.\tau \big ). \nonumber
\end{align}
A further property of $g$ is the interaction between differentiation in $(u,\nu)$ space and application of the operators $\DD{k}{\betaa}$ used in the definition of $\mathsf{G}$, cf.~\eqref{jc38}.
\begin{lemma}\label{lem:uspace}
For all $k \geq 0$ and $\betaa \in \N_{0}^{d}$ it holds
\begin{align}
\partial_u^{k}\partial_\nu^{\betaa} ( g.\tau)=g.\DD{k}{\betaa}\tau.\label{snew1}
\end{align}
\end{lemma}
\begin{proof}
Note that it suffices to show \eqref{snew1} for the special cases $k=1,\betaa=0$ and $k=0,|\betaa|=1$ which read
\begin{align}
\partial_u ( g .\tau)=g.D^{(0)} \tau, \quad \partial_{\nu}^\betaa ( g .\tau)=g.(D^{(1)})^{\betaa} \tau \quad \mbox{for} \; |\betaa|=1. \label{snew2}
\end{align}
The general case then follows by iteration.
To establish \eqref{snew2} we start with monomials and then use the morphism property \eqref{morphismP} to extend to a general $\tau$.
Indeed, note that for monomials $\poly_{x}$ and $\poly_{j}$, $j\ge 0$, we find 
\begin{align*}
\partial_u (g.\poly_{j})&=\frac{a^{(j+1)}(u)}{j!} =(j+1) g.\poly_{j+1} \stackrel{\eqref{as30}}{=}g.D^{(0) }\poly_{j}, \\
\partial_u (g.\poly_{x})&=\partial_u\nu=0 \stackrel{\eqref{as30}}{=}g.D^{(0) }\poly_{x},
\end{align*}
as well as
\begin{align*}
\partial_\nu^{\betaa} (g.\poly_{j})&=\partial_\nu^{\betaa}\frac{a^{(j)}(u)}{j!} =0\stackrel{\eqref{as37}}{=}g.(D^{(1) })^{\betaa}\poly_{j}, \\
\partial_\nu^{\betaa} (g.\poly_{x}^{\betaa})&=\partial_\nu^{\betaa}\nu^\betaa=1 \stackrel{\eqref{as37}}{=}g.(D^{(1) })^{\betaa}\poly_{x}^{\betaa}.
\end{align*}
Since $D^{(0)}$ is a derivation, cf.\@ \eqref{as28}, and since $\tau \mapsto g.\tau$ is a morphism, if \eqref{snew2} holds for given $\tau, \tau'$, applying the product rule gives
\begin{align*}
\partial_u\big ( g.\tau \tau' \big ) &=\partial_u\big ( g.\tau g.\tau' \big )=\big (\partial_u g.\tau\big )\big ( g.\tau' \big ) + \big ( g.\tau\big ) \big (\partial_u g.\tau' \big )\\
&\stackrel{\eqref{snew2} }{=} \big(g.D^{(0)} \tau\big) \big( g.\tau'\big) + \big(g.\tau\big) \big( g.D^{(0)}\tau'\big)\\
&= g.\big((D^{(0)}\tau)\tau' + \tau(D^{(0)}\tau')\big)\stackrel{ \eqref{as28} }{=}g.D^{(0)} (\tau \tau').
\end{align*} 
Similarly, $\partial_\nu^{\betaa} ( g.\tau \tau')=g.(D^{(1)})^\betaa (\tau \tau')$ for $|\betaa|=1$. This shows \eqref{snew2} for all $\tau\in\TT_{|\gamma|}$ that are polynomial in $a_0$, and hence by density for all $\tau\in\TT_{|\gamma|}$.
\end{proof}

\subsection{Projections}
For each $\beta$ we define the projection $P_\beta:\TT\to \Ta{|\beta|}$ via
\begin{align*}
 \tau=\sum_{\beta} \tau_{\beta}\poly^{\beta} \mapsto P_{\beta}\tau:=\tau_{\beta}\poly^{\beta},    
\end{align*}
 and for each $\eta>0$, we define the projection $Q_{\eta}:=\sum_{|\beta|<\eta}P_\beta: \TT \mapsto \TT_{<\eta}$, i.e. 
\begin{equation}
\tau=\sum_{\beta} \tau_{\beta}\poly^{\beta} \mapsto Q_{\eta}\tau:=\sum_{|\beta|<\eta}\tau_{\beta}\poly^{\beta} \label{sFin3}.
\end{equation}
We will need a variation of the exponential formula \eqref{as41} for the composition of a group element with a projection onto homogeneities below a given level.  This will be employed in the proof of Corollary \ref{cor:taylor_identity}, which is the starting point for the proof of the Graded Continuity Lemma.  In preparation for truncating the infinite summation in \eqref{as41}, it is convenient to introduce the following notation: for $|\betaa|=0,1$ and $\kappa \in \R$
\begin{equation}
 K(\kappa,|\betaa|):=\lceil \alpha^{-1}(\kappa+|\betaa|) \rceil -\lceil \alpha^{-1}|\betaa| \rceil-1. \label{e187}
\end{equation}
Throughout the article, we will often use the abbreviation $K(|\betaa|)$.
\begin{lemma}\label{lem:truncatedGamma}
Let $\eta>0$ and let $\Gamma \in \mathsf{G}$ be associated with $(\primz,\primo) \in \TT \times (\TT_{>1})^{d}$. 
For each $\tau \in \TT_{|\gamma|}$ with $|\gamma| \in \Aa{}$ and $|\gamma|<\eta$ it holds
\begin{align}\label{truncatedGamma}
Q_{\eta}\Gamma\tau
&=Q_{\eta}\sum_{|\betaa|=0,1}\sum_{k=0}^{K(\eta-|\gamma|,|\betaa|)}\frac{1}{k!\betaa!} \primm{k}{\betaa} \DD{k}{\betaa}\tau,
\end{align}
cf.~\eqref{e187}.
\end{lemma}
\begin{proof}
We apply $Q_{\eta}$ on both sides of \eqref{as41} and our goal is to truncate the summation in $k,|\betaa|$.
Observe that for $\tau \in \TT_{|\gamma|}$ with $|\gamma| \in \Aa{}$ we have $|\gamma_{x}| \in \{0,1\}$ and therefore $\DD{k}{\betaa}\tau=0$ if $|\betaa| \geq 2$, cf.~\eqref{as40}.  Note that due to the structure of the set of homogeneities, $\tau^{(1)} \in \TT_{>1}^d$ implies $\tau^{(1)} \in \TT_{ \geq  \lceil \alpha^{-1} \rceil \alpha  }^d$, which together with $(\tau^{(0)})^{k} \in \TT_{\geq \alpha}$ implies via \eqref{prodHomog} that  $\primm{k}{\betaa} \in \TT_{ \geq (1+|\betaa|\lfloor \alpha^{-1} \rfloor)\alpha}$.
 Combining this with $\DD{k}{\betaa}\tau \in \TT_{|\gamma|+k\alpha+|\betaa|(\alpha-1) }$, cf.~\eqref{jc36}, \eqref{prodHomog} yields that $\primm{k}{\betaa} \DD{k}{\betaa}\tau \in \TT_{\ge|\gamma|+k\alpha+|\betaa|( \lceil \alpha^{-1} \rceil \alpha -1 ) }$ and therefore $Q_{\eta}(\primm{k}{\betaa} \DD{k}{\betaa}\tau)=0$ provided that $k \geq \alpha^{-1}(\eta-|\gamma|+|\betaa|)-|\betaa| \lceil \alpha^{-1} \rceil$ and hence for $k \geq K(\eta-|\gamma|,|\betaa|)+1$.  
\end{proof}
Next we introduce a truncated version of $g$ denoted $g_{\eta}$, via
\begin{align}\label{sFin3g}
\tau \in \TT \mapsto g_{\eta}.\tau:=g.Q_{\eta}\tau=\sum_{|\beta|<\eta}g.P_\beta\tau.
\end{align}
We also need a variant of the morphism property \eqref{morphismP} for $g_{\eta}$ with $\eta \in (0,2)$.
Let us first consider a special case and let  $\sigma,\tau \in \TT$ with $\tau \in \TT_{|\gamma|}$ for some $|\gamma|<\eta$, then by \eqref{jc28} 
\begin{align}
g_{\eta}.(\sigma \tau)&=\sum_{|\beta|<\eta}\sum_{\beta_{1}+\beta_{2}=\beta}\nu^{\beta_{1,x}+\beta_{2,x}}da(u)^{\beta_{1}'+\beta_{2}' }\sigma_{\beta_{1}}(a(u))\tau_{\beta_{2}}(a(u)) \nonumber \\
&=\sum_{|\beta_{2}|=|\gamma|}\nu^{\beta_{2,x}}da(u)^{\beta_{2}'}\tau_{\beta_{2}}(a(u))\sum_{|\beta_{1}|<\eta+\alpha-|\gamma|}\nu^{\beta_{1,x}}da(u)^{\beta_{1}'} \sigma_{\beta_{1}}(a(u)) \nonumber \\
&=\big ( g_{\eta+\alpha-|\gamma|}.\sigma \big ) \big ( g_{\eta}.\tau \big ). \label{truncatedMorphism}
\end{align}
Note that the right-hand side is generally not the same as $(g_{\eta}.\sigma)(g_{\eta}.\tau)$, but instead involves a truncation at a (potentially) lower level $\eta+\alpha-|\gamma| \leq \eta$.  More generally, we find that for any $\sigma,\tau \in \TT$ it holds
\begin{equation}
    g_{\eta}.(\sigma \tau)=\sum_{|\gamma|<\eta}\big ( g_{\eta+\alpha-|\gamma|}.\sigma \big ) \big ( g.P_{\gamma}\tau \big ) \label{truncatedMorphism2},
\end{equation}
which follows from $g_{\eta}.(\sigma \tau)=g_{\eta}.(\sigma Q_{\eta}\tau)$, cf.~the inclusion \eqref{prodHomog}, then by decomposing the projection cf.~\eqref{sFin3} and applying \eqref{truncatedMorphism}.  We now combine Lemma \ref{lem:uspace}, Lemma \ref{lem:truncatedGamma}, and \eqref{truncatedMorphism} to obtain the following corollary, for which we recall the definition of $\mu(k,\betaa)$ in \eqref{jc36}.
\begin{corollary}\label{cor:taylor_identity}
Let $\eta \in (0,2)$, $\tau \in \TT_{|\gamma|}$ for $|\gamma|<\eta$, and $g$ be given by \eqref{md0g}. For all $u,u' \in \R$, $\nu, \nu ' \in \R^{d}$, there exist some $u^{0},u^{1}\in \R$ between $u$ and $u'$ such that
    \begin{align*}
    g_{\eta}&(u',\nu').\tau -g_{\eta}(u,\nu).\Gamma \tau \\
    &=\sum_{|\betaa|=0,1}\sum_{k=1-|\betaa|}^{K(\eta-|\gamma|, |\betaa|)}\frac{1}{k!\betaa!} \big ((u'-u)^{k}(\nu'-\nu)^{\betaa}-g_{\eta+\alpha-|\gamma|-\mu(k,\betaa)}(u,\nu).\tau^{(k,\betaa)}\big )g(u,\nu).D^{(k,\betaa)}\tau \nonumber \\
&+\sum_{|\betaa|=0,1}\frac{1}{K(\eta-|\gamma|,|\betaa|)!}(u'-u)^{K(\eta-|\gamma|,|\betaa|)}(\nu'-\nu)^{\betaa}\big(g(u^{|\betaa|},\nu)-g(u,\nu)\big).D^{(K(\eta-|\gamma|,|\betaa|), \betaa)}\tau.
\end{align*}
\end{corollary}
\begin{proof}
Since $\tau \in \TT_{|\gamma|}$ for $|\gamma|<\eta$, it holds that $g_{\eta}(u',\nu').\tau=g(u',\nu').\tau$ and hence
\begin{equation}
 g_{\eta}(u',\nu').\tau-g_{\eta}(u,\nu).\Gamma \tau =\big ( g(u',\nu')-g(u,\nu) \big ). \tau-g_{\eta}(u,\nu).(\Gamma \tau-\tau) \label{e190}.  
\end{equation}
We will analyze the first term using Taylor's theorem and \eqref{snew1}, while for the second term we will appeal to Lemma \ref{lem:truncatedGamma}.  Since $\eta$ and $|\gamma|$ are fixed throughout the proof, we will simply write $K(|\betaa|)$ instead of $K(\eta-|\gamma|,|\betaa|)$.  Indeed, for the second term in \eqref{e190} applying $g$ on both sides of \eqref{truncatedGamma}, cf.~\eqref{sFin3g} we find
\begin{align}
g_{\eta}.(\Gamma\tau-\tau)&=\sum_{|\betaa|=0,1}\sum_{k=1-|\betaa|}^{K(|\betaa|)}\frac{1}{k!\betaa!}g_{\eta}.( \primm{k}{\betaa} \DD{k}{\betaa}\tau) \nonumber \\
&=\sum_{|\betaa|=0,1}\sum_{k=1-|\betaa|}^{K(|\betaa|)}\frac{1}{k!\betaa!}g_{\eta+\alpha-|\gamma|-\mu(k,\betaa)}. \primm{k}{\betaa}g. \DD{k}{\betaa}\tau. \label{e1}
\end{align}
In the second equality, we used \eqref{truncatedMorphism} and $\DD{k}{\betaa}\tau \in \TT_{|\gamma|+\mu(k,\betaa)}$, cf.~\eqref{jc36}.
We also used that $g_{\eta}. \DD{k}{\betaa}\tau=g. \DD{k}{\betaa}\tau$ due to $|\gamma|+\mu(k,\betaa) < \eta$ for $k \leq K(|\betaa|)$, which is a result of the inequality $K(|\betaa|)<\alpha^{-1}(\eta-|\gamma|+|\betaa| )-|\betaa|$, c.f. \eqref{e187}.

\medskip

For the first term in \eqref{e190}, we write
\begin{align}\label{js109}
    \big ( g(u',\nu')-g(u,\nu) \big ). \tau&= \big ( g(u',\nu)-g(u,\nu) \big ). \tau + \big ( g(u',\nu')-g(u',\nu) \big ). \tau
\end{align}
and analyze both terms separately.
By Taylor's formula in $u$ to order $K(0)$ and \eqref{snew1}, there exists a $u^{0} \in \R$ between $u$ and $u'$ such that 
\begin{align}
    \big ( g(u',\nu)-g(u,\nu) \big ). \tau & =\sum_{k=1}^{K(0)}\frac{1}{k!}(u'-u)^{k}g(u,\nu).\DD{k}{0}\tau \nonumber\\
    &\quad + \frac{1}{K(0)!}(u'-u)^{K(0)}\big(g(u^{0},\nu)-g(u,\nu)\big). \DD{K(0)}{0}\tau \label{e191}.
\end{align}
For the second term in \eqref{js109}, we first use that $\partial_\nu^\betaa g=0$ for $|\betaa|>1$ to write
\begin{align*}
    \big ( g(u',\nu')-g(u',\nu) \big ). \tau = \sum_{|\betaa|=1}(\nu'-\nu)^\betaa \partial_\nu^\betaa g(u',\nu).\tau = \sum_{|\betaa|=1}(\nu'-\nu)^\betaa g(u',\nu).\DD{0}{\betaa}\tau.
\end{align*}
Now Taylor's formula in $u$ to order $K(1)$ and \eqref{snew1} yield a $u^1 \in \R$ between $u$ and $u'$ such that
\begin{align}
    \big ( g(u',\nu')-g(u',\nu) \big ). \tau &= \sum_{|\betaa|=1}\sum_{k=0}^{K(1)}(\nu'-\nu)^\betaa (u-u')^k g(u,\nu).\DD{k}{\betaa}\tau \nonumber\\
    & \quad + \sum_{|\betaa|=1}\frac{1}{K(1)!} (\nu'-\nu)^\betaa (u-u')^{K(1)} \big( g(u^1,\nu)-g(u,\nu) \big).\DD{K(1)}{\betaa}\tau. \label{e192}
\end{align}
Combining the identities \eqref{e190}-\eqref{e192} the proof is complete.
\end{proof}
\begin{remark}\label{rem:morphismExt}
In the above, we used that $\eta<2$ in our appeal to \eqref{truncatedMorphism}.  We remark that Corollary \ref{cor:taylor_identity} continues to hold if we relax this assumption to $\eta<1+\lceil \alpha^{-1}\rceil \alpha$ and $\tau \in \TT_{-}$.
Recall that in  the derivation of \eqref{truncatedMorphism}, this ensures that there is no contribution of the form $\sigma_{\beta_{1}}(a(u))\tau_{\beta_{2}}(a(u))$ where both $\beta_{1,x}\neq 0$ and $\beta_{2,x} \neq 0$, so the application of \eqref{jc28} is valid. 
\end{remark}
In our main application of the previous lemma, we will need to further simplify the quantity 
\begin{equation}
    (u'-u)^{k}(\nu'-\nu)^{\betaa}-g_{\eta+\alpha-|\gamma|-\mu(k,\betaa)}(u,\nu).\tau^{(k,\betaa)}
\end{equation}
in order to relate it to the semi-norms \eqref{js100} defined in the next section.  This is accomplished with the following lemma. 
\begin{lemma}\label{lem_identity}
Let $\kappa\in (0,2)$, $J\in \N$, $u_m\in\R$ and $\tau_m\in\TT$ for $m\in\set{1,\ldots,J}$.
The following identity holds:
\begin{align}
    &\bigg ( \prod_{m=1}^J u_m \bigg )-g_\kappa.(\tau_1\cdots\tau_J) \nonumber \\
    &= (u_1-g_{\kappa}.\tau_1)\prod_{m=2}^J u_m
     +\sum_{|\beta|<\kappa}\sum_{j=2}^J  (u_j-g_{\kappa+\alpha-|\beta|}.\tau_j) \bigg (  \prod_{m=j+1}^J u_m \bigg )g.P_{\beta}(\tau_{1}\cdots \tau_{j-1} ) \label{e2}  .
\end{align}
\end{lemma}
\begin{proof}
    We give a proof by induction on $J$.  For $J=1$, the claim 
follows immediately, noting that by convention the empty sum is zero and the empty product is one.  To provide some additional intuition, let us also consider the case $J=2$ (the reader can also skip directly to the case of general $J$ below).  Using \eqref{truncatedMorphism2},
\begin{align}
u_{1}u_{2}-g_{\kappa}.(\tau_{1}\tau_{2})&=(u_{1}-g_{\kappa}.\tau_{1})u_{2}+u_{2}(g_{\kappa}.\tau_{1})-\sum_{|\beta|<\kappa}(g_{\kappa+\alpha-|\beta|}.\tau_{2} )( g.P_{\beta}\tau_{1}) \nonumber \\
&=(u_{1}-g_{\kappa}.\tau_{1})u_{2}+\sum_{|\beta|<\kappa} (u_{2}-g_{\kappa+\alpha-|\beta|}.\tau_{2})( g.P_{\beta}\tau_{1}) \nonumber,
\end{align}
which is precisely \eqref{e2}.  We now proceed to the general inductive proof.  Let $J>1$ be such that the statement is true for $J-1$.
We multiply the induction hypothesis by $u_J$ 
resulting in 
\begin{align}
    &\bigg ( \prod_{m=1}^J u_m \bigg )-u_{J}g_\kappa.(\tau_1\cdots\tau_{J-1})  \nonumber \\
    &= (u_1-g_{\kappa}.\tau_1)\prod_{m=2}^{J} u_m
     +\sum_{|\beta|<\kappa}\sum_{j=2}^{J-1}  (u_j-g_{\kappa+\alpha-|\beta|}.\tau_j) \bigg (  \prod_{m=j+1}^J u_m \bigg )g.P_{\beta}(\tau_{1}\cdots \tau_{j-1} )\label{ee1}   .
\end{align}
Now we further analyze the LHS of the equality above and write
\begin{align}
    u_{J}g_\kappa.(\tau_1\cdots\tau_{J-1}) =\sum_{|\beta|<\kappa} u_{J}g.P_{\beta}(\tau_1\cdots\tau_{J-1})&=\sum_{|\beta|<\kappa} (u_{J}-g_{\kappa+\alpha-|\beta|}.\tau_{J} ) g.P_{\beta}(\tau_1\cdots\tau_{J-1}) \nonumber \\
    &+\sum_{|\beta|<\kappa}  g.P_{\beta}(\tau_1\cdots\tau_{J-1}) (g_{\kappa+\alpha-|\beta|}.\tau_{J}). \nonumber
\end{align}
The first term is now incorporated in the RHS of \eqref{ee1} to give the contribution from $j=J$, and it only remains to argue that the second term simplifies via
\begin{equation}
    g_{\kappa}.(\tau_1\cdots\tau_{J})=\sum_{|\beta|<\kappa} g.P_{\beta}(\tau_1\cdots\tau_{J-1}) (g_{\kappa+\alpha-|\beta|}.\tau_{J}) \nonumber,
\end{equation}
which follows immediately from \eqref{truncatedMorphism2}.
\end{proof}
\begin{remark}
Let us explain how we intend to apply Lemma \ref{lem_identity} in the context of Corollary \ref{cor:taylor_identity}.  We claim that for $|\betaa|=0,1$ and any $\kappa \in (0,2)$ it holds
\begin{align}
 &(u'-u)^{k}(\nu'-\nu)^{\betaa}-g_{\kappa}(u,\nu).\tau^{(k,\betaa)} \nonumber \\
 &= \big ((u'-u)^{1-|\betaa|}(\nu'-\nu)^{\betaa}-g_{\kappa}(u,\nu).(\tau^{(0)})^{1-|\betaa|}(\tau^{(1)})^{\betaa} \big ) (u'-u)^{k+|\betaa|-1}\nonumber \\
 &+\sum_{|\beta|<\kappa} \big ( u'-u-g_{\kappa+\alpha-|\beta|}(u,\nu).\tau^{(0)} \big )\sum_{j=2}^{k+|\betaa|}(u'-u)^{k+|\betaa|-j}g.P_{\beta}\tau^{(j-1-|\betaa|,\betaa)}\label{e8}
\end{align}
Indeed, for $|\betaa|=0,1$ this follows from Lemma \ref{lem_identity} with $J=k+|\betaa|$, $u_{1}=(u'-u)^{1-|\betaa|}(\nu'-\nu)^{ \betaa}$, $\tau_{1}=(\tau^{(0)})^{1-|\betaa|}(\tau^{(1)})^{\betaa}$ and $u_{m}=u'-u$, $\tau_{m}=\tau^{(0)}$  for $m=2,\cdots, J$.  The identity above makes it easy to estimate the LHS in terms of the semi-norms \eqref{js100} defined below.
\end{remark}

\section{Modelled Distributions}\label{sec:md}
Let $\domain\subset \R^{\di+1}$ be a bounded, open, convex domain.
Given functions $u:\domain\to\R$, $\nu=(\nu_1,\ldots,\nu_{\di}):\domain\to\R^{\di}$, and a cut-off value $\eta>0$, recalling \eqref{md0g}, \eqref{sFin3} we define a map $f_\eta:\domain\to \Tplus^*$ via
\begin{equation}
f_{\eta}(\txv).\tau:=g(u(\txv),\nu(\txv)).Q_\eta\tau.\label{md0}
\end{equation}
An important consequence of truncating $f_{\eta}$ at a finite level $\eta$ is the loss of the covariance property \eqref{sd1}, which has to be replaced by a corresponding continuity property.
To quantify this type of continuity of $f_{\eta}$ with respect to the base point $\txv\in \domain$ in the case of a finite cut-off level $\eta>0$, we take inspiration from \cite[Definition 3.7]{Hai14} and define the quantity\footnote{For ease of notation we do not explicitly state the dependence of $\opnorm{f}$, $[u]_\eta$ and $[\nu]_\eta'$ on $\domain$.} $\opnorm{f_{\eta}}$ to be the minimal $M>0$ such that for all $\tau \in \Tminus$, $\txv\in \domain$ and $\syv\in B_{\frac12\dist_\txv}(\txv)$ it holds
\begin{align}
\dist_{\txv}^{\eta} \big |f_{\eta}(\syv).\tau-f_{\eta}(\txv).\Gamma_{\syv\txv}\tau \big | \leq M \sum_{|\beta|<\eta}d^{\eta-|\beta|}(\syv,\txv)\dist_{\txv}^{\langle\beta\rangle\alpha}\|\tau \|_{\Ta{|\beta|} },\label{md19}
\end{align}
where we recall the shorthand-notation $\dist_\txv:=\dist(\txv,\partial\domain)$ for the parabolic distance of $\txv$ from the boundary of $\domain$.\footnote{On a first reading, we would advise the reader to ignore the factor of $\dist_{\txv}^\eta$ and think of $\syv \in \domain$ in the definitions \eqref{md19} and \eqref{js100}. Many of the core ideas in the paper are largely unrelated to this additional weight.  In fact, the weights could be completely avoided if one restricted attention to solutions to \eqref{intro3} which are space-time periodic, though in general due to the renormalization these may be difficult to construct even with smooth noise.}
By analogy to \eqref{jc78}, we seek to control $u$ and $\nu$ through the (weighted) nonlinear quantities
\begin{align}\label{js100}
\begin{split}
[u]_{\kappa}&:=\sup \bigg\{ \dist_{\txv}^\kappa 
\frac{ |u(\syv)-u(\txv)-f_{\kappa}(\txv).\primz_{\syv\txv} |}{d^{\kappa}(\syv,\txv)}: \txv\in \domain, \, \syv\in B_{\dist_{\txv}}(\txv)
\bigg \}, \quad \kappa>0.  \\
[\nu]_{\kappa}'&:=\sup \bigg\{ \dist_{\txv}^\kappa 
\frac{ |\nu(\syv)-\nu(\txv)-f_{\kappa}(\txv).\primo_{\syv\txv}  |}{d^{\kappa-1}(\syv,\txv)}: \txv\in \domain, \, \syv\in B_{\frac{1}{2}\dist_{\txv}}(\txv)
\bigg \}, \, \, \, \kappa>1.  \\
\end{split}
\end{align}

This control relaxes the formal identity \eqref{q2} and draws on the ideas of (controlled) rough paths developed in \cite{Gub04}, \cite{Lyo98}.
Observe that both quantities depend
not only on $u$ and $\nu$, but also on on $\xi$ and the nonlinearity $a$ via $f$. Note that $[u]_{\alpha}$ is just a weighted $\alpha$-H\"older semi-norm of $u$, while for higher values of $\kappa$, the quantity $[u]_{\kappa}$ is nonlinear in $u$. Similarly, $[\nu]'_{ \lceil \alpha^{-1} \rceil \alpha }$ is a weighted $(\lceil \alpha^{-1} \rceil \alpha-1)$-H\"older norm of $\nu$, while for higher values of $\kappa$, nonlinear effects come into play. We additionally introduce
\begin{equation}
\|\nu\|_{1}':=\sup_{\txv\in \domain} \dist_{\txv} |\nu(\txv)|    \label{js101a}
\end{equation}
The following lemma is at the core of the article.  It can be understood as a control on the semi-norm defined via \eqref{md19} in terms of \eqref{js100}-\eqref{js101a} in a way that scales optimally with respect to $u$ and $\nu$. This lemma is crucial in order to meet the requirements of reconstruction and integration, cf.~Propositions \ref{rec_lem} and \ref{int lem} which we will recall in Section \ref{sec_jets} below.  
\begin{contlem}\label{lem1}
Let $u:\domain\to\R$ and $\nu:\domain\to\R^{\di}$ be smooth functions
and define $f_{\eta}$ via \eqref{md0} with $\eta<1+ \lfloor \alpha^{-1} \rfloor \alpha$. 
Assume that for some $\delta \in (0,\frac{1}{2})$ it holds
$\|u\|+\delta^{\alpha}[\Pi] \leq 1$.  Then the following estimate holds:
\begin{equation}
    \opnorm{f_{\eta+\alpha} } \lesssim [u]_{\eta}+[\nu]_{\eta}'+ \delta^{-\eta} \nonumber.
\end{equation}
\end{contlem}
\begin{remark}
The assumption that $\|u\|+\delta^{\alpha}[\Pi] \leq 1$ is purely for convenience.
In fact, the two main inputs for the  \autoref{lem1} are the
\autoref{gradcontlem1}, which gives a more general bound on the above semi-norm, together with interpolation inequalities, cf.~Lemma \ref{lem: positiveInt}, and neither requires this assumption, but it slightly simplifies the combined output.
\end{remark}

\subsection{Graded Continuity Lemma}
We start by analyzing for each $|\gamma|<\eta$ the optimal $M_{\eta,|\gamma|}$ such that \eqref{md19} holds for all  $\tau \in \TT_{|\gamma|} \cap \Tminus$.  We refer to this as the Graded Continuity Lemma, since the estimate for $M_{\eta,|\gamma|}$ involves the semi-norms \eqref{js100} for variable $\kappa$, in comparison to the Continuity Lemma which involves only the semi-norm of order $\eta$.
\begin{gradcontlem}\label{gradcontlem1}
Let $\eta<1+\lceil \alpha^{-1} \rceil\alpha$.  For all $\tau \in \TT_{|\gamma|} \cap \Tminus$ with $|\gamma|<\eta$, $\txv\in \domain$ and $\syv\in B_{\frac12\dist_\txv}(\txv)$ it holds 
\begin{align}\label{coreEstimate_experimental}
\dist_{\txv}^{\eta-\langle\gamma\rangle\alpha} \big |f_{\eta}(\syv).\tau-f_{\eta}(\txv).\Gamma_{\syv\txv}\tau \big | \lesssim M_{\eta,|\gamma|} d^{\eta-|\gamma|}(\syv,\txv)\|\tau \|_{\Ta{|\gamma|} },
\end{align}
where 
\begin{align}
&M_{\eta,|\gamma|} \nonumber \\
&\lesssim 1+[\Pi]^{\frac{\eta-\langle \gamma \rangle \alpha}{\alpha}}+ \sup_{ \alpha \leq \kappa \leq \eta-\langle \gamma\rangle \alpha} [u]_{\kappa}^{\frac{\eta-\langle \gamma\rangle \alpha}{\kappa}} +  (\|\nu\|_{1}')^{\eta-\langle \gamma \rangle \alpha}+1_{|\gamma_{x}|=1}\sup_{\lceil \alpha^{-1} \rceil \alpha \leq \kappa \leq \eta-\langle \gamma \rangle \alpha}([\nu]'_{\kappa})^{\frac{\eta-\langle \gamma \rangle \alpha}{\kappa}} \label{e197}
\end{align}
\end{gradcontlem}
\begin{proof}[Proof of the Graded Continuity Lemma]
\newcounter{FCP} 
\refstepcounter{FCP} 

\medskip

Without loss of generality, we may assume $\|\tau\|_{\TT_{|\gamma|}}=1$.  For brevity, we will write  $K(|\betaa|)$ for $K(\eta-|\gamma|,|\betaa|)$.  We apply Corollary \ref{cor:taylor_identity}, taking into account Remark \ref{rem:morphismExt}, with $u'=u(\syv),u=u(\txv)$, $\nu'=\nu(\syv),\nu=\nu(\txv)$, and $\Gamma=\Gamma_{\syv\txv}$ to obtain
\begin{align}
&f_{\eta}(\syv).\tau -f_{\eta}(\txv).\Gamma_{\syv\txv} \tau    \nonumber \\
&=\sum_{|a|=0,1}\sum_{k=1-|a|}^{K(\betaa)}\frac{1}{k!\betaa!} \big ((u(\syv)-u(\txv))^{k}(\nu(\syv)-\nu(\txv))^{\betaa}-f_{\eta+\alpha-|\gamma|-\mu(k,\betaa)}(\txv).\tau^{(k,\betaa)}_{\syv\txv}\big )f(\txv).D^{(k,\betaa)}\tau \label{e9} \\
&+\sum_{|\betaa|=0,1}\frac{1}{K(\betaa)!}(u(\syv)-u(\txv))^{K(\betaa)}(\nu(\syv)-\nu(\txv))^{\betaa}\big(g(u^{|\betaa|},\nu(\txv))-f(\txv)\big).D^{(K(\betaa),\betaa)}\tau \label{e12},
\end{align}
for some intermediary point $u^{\betaa}$ with $|u^{\betaa}-u(\txv)| \leq |u(\syv)-u(\txv)|$. We will start by arguing the following bound, which is used in estimating both \eqref{e9} and \eqref{e12}: for all $|\betaa|\le 1$ and $1 \leq k \leq K(\betaa)+1$
\begin{align}
    |f(\txv).\DD{k}{\betaa}\tau|\le \|g(\cdot,\nu(\txv)).\DD{k}{\betaa}\tau\|_{C^{0}(\R)} \lesssim (1-1_{|\gamma_{x}|=0,|\betaa|=1})\big(\frac{\|\nu\|_1'}{\dist_\txv}\big)^{|\gamma_x|(1-|\betaa|)} \label{taylorRemainder}.
\end{align}
The first inequality is immediate.  For the second, notice that if $|\gamma_{x}|=0$ and $|\betaa|=1$, then $g(\cdot,\nu(\txv)).\DD{k}{\betaa}\tau=0$ since $\DD{k}{\betaa}\tau=0$, cf.~\eqref{as37}.  The other cases follow easily from the definitions \eqref{md0g} and \eqref{md0}, taking into account the convention
$\|a^{(m)}(\cdot)\|_{C^{0}(\R)} \lesssim 1$ for $m \leq \nn$, \eqref{jc36}, and the normalization $\|\tau\|_{\TT_{|\gamma|}}=1$.

\medskip

We now turn to estimating \eqref{e12}.  By \eqref{taylorRemainder}, and using that for $\delta\in [0,1]$ and $|\betaa|\le 1$ we have the (standard H\"older space) interpolation inequality
\begin{align*}
\|g(\cdot,\nu(\txv)).D^{(K(|\betaa|),\betaa)}\tau\|_{C^{\delta}(\R)}
 &\le \|g(\cdot,\nu(\txv)).D^{(K(|\betaa|),\betaa)}\tau\|_{C^{0}(\R)}+\|g(\cdot,\nu(\txv)).D^{(K(|\betaa|),\betaa)}\tau\|_{C^{1}(\R)} \\  
 &\le \|g(\cdot,\nu(\txv)).D^{(K(|\betaa|),\betaa)}\tau\|_{C^{0}(\R)}+\|g(\cdot,\nu(\txv)).D^{(K(|\betaa|)+1,\betaa)}\tau\|_{C^{0}(\R)} \\
 &\lesssim (1-1_{|\gamma_{x}|=0,|\betaa|=1})\big(\frac{\|\nu\|_1'}{\dist_\txv}\big)^{|\gamma_x|(1-|\betaa|)},
\end{align*}
setting $\delta=\frac{\eta-|\gamma|+|\betaa|}{\alpha}-\lceil \frac{\eta-|\gamma|+|\betaa|}{\alpha} \rceil+1\in [0,1]$ cf.~\eqref{e187}, we can bound each summand of \eqref{e12} by
\begin{align*}
  &|u(\syv)-u(\txv)|^{K(|\betaa|)}|u^{\betaa}-u(\txv)|^{\delta}|\nu(\syv)-\nu(\txv)|^{|\betaa|}\|g(\cdot,\nu(\txv)).D^{(K(|\betaa|),\betaa)}\tau\|_{C^{\delta}(\R)}\\
    &\lesssim (1-1_{|\gamma_{x}|=0,|\betaa|=1})d^{\eta-|\gamma|}(\syv,\txv)\dist_\txv^{\langle\gamma\rangle\alpha-\eta}[u]_\alpha^{\alpha^{-1}(\eta-|\gamma|)+|\betaa|(\alpha^{-1}-\lceil \alpha^{-1} \rceil ) }([\nu]_{\lceil \alpha^{-1} \rceil\alpha}')^{|\betaa|}(\|\nu\|_1')^{|\gamma_x|(1-|\betaa|)} \nonumber \\
    &\lesssim d^{\eta-|\gamma|}(\syv,\txv)\dist_\txv^{\langle\gamma\rangle\alpha-\eta} \bigg ( [u]_\alpha^{\frac{\eta-\langle \gamma\rangle \alpha}{\alpha}}+ 1_{|\gamma_{x}|=1}  (\|\nu\|_{1}')^{\eta-\langle \gamma \rangle \alpha}+1_{|\gamma_{x}|=1}([\nu]_{\lceil \alpha^{-1} \rceil\alpha}')^{\frac{\eta-\langle \gamma \rangle \alpha}{\lceil \alpha^{-1} \rceil\alpha} } \bigg ), \nonumber
\end{align*}
where we used Young's inequality in the last step, keeping in mind \eqref{model201}. 

\medskip

Now we turn to the estimate for \eqref{e9}: for the $(k,\betaa)$ summand we bound $f(\txv).\DD{k}{\betaa}\tau$ with \eqref{taylorRemainder} and estimate the other part of the product
using identity \eqref{e8} with $\kappa=\eta+\alpha-|\gamma|-\mu(k,\betaa)$.  Note that due to \eqref{taylorRemainder}, we are free to exclude the case $|\betaa|=1,|\gamma_{x}|=0$, so that our assumption $\eta<1+ \lfloor \alpha^{-1} \rfloor \alpha$ ensures $\kappa<2$.  We therefore obtain
\begin{align}
 &(u(\syv)-u(\txv))^{k}(\nu(\syv)-\nu(\txv))^{\betaa}-f_{\eta+\alpha-|\gamma|-\mu(k,\betaa)}(\txv).\tau^{(k,\betaa)}_{\syv \txv}  \nonumber \\
 &= \big ((u(\syv)-u(\txv))^{1-|\betaa|}(\nu(\syv)-\nu(\txv))^{\betaa}-f_{\eta+\alpha-|\gamma|-\mu(k,\betaa)}(\txv).(\tau^{(0)}_{\syv \txv} )^{1-|\betaa|}(\tau^{(1)}_{ \syv \txv})^{\betaa} \big ) \label{e10} \\
 & \qquad \qquad \qquad \qquad \qquad \qquad \qquad \qquad \qquad \qquad \qquad \qquad \times (u(\syv)-u(\txv))^{k+|\betaa|-1} \nonumber \\
 &+\sum_{|\beta|<\eta+\alpha-|\gamma|-\mu(k,\betaa)} \big ( u(\syv)-u(\txv)-f_{\eta+2\alpha-|\gamma|-\mu(k,\betaa)-|\beta|}(\txv).\tau^{(0)}_{\txv \syv} \big )\sum_{j=2}^{k+|\betaa|}(u(\syv)-u(\txv))^{k+|\betaa|-j} \label{e11} \\
 & \qquad \qquad \qquad \qquad \qquad \qquad \qquad \qquad \qquad \qquad \qquad \qquad \times f(\txv).P_{\beta}\tau^{(j-1-|\betaa|,\betaa)}_{\syv \txv} \nonumber.
\end{align}
We insert the above identity into \eqref{e9} and then estimate the resulting contribution, taking into account \eqref{taylorRemainder}.  The contribution from \eqref{e10} for $k \geq 1-|\betaa|$ is estimated 
by $d^{\eta-|\gamma|}(\syv,\txv)\dist_\txv^{\langle\gamma\rangle\alpha-\eta}$ multiplied by
\begin{align}
&(1-1_{|\gamma_{x}|=0,|\betaa|=1})\big ( [u]_{\eta+\alpha-|\gamma|-\mu(k,\betaa)} (\|\nu\|_1')^{|\gamma_x|}\big )^{1-|\betaa|}([\nu]_{\eta+\alpha-|\gamma|-\mu(k,\betaa) }')^{|\betaa|}[u]_{\alpha}^{k+|\betaa|-1} \nonumber \\
& \lesssim [u]_\alpha^{\frac{\eta-\langle \gamma\rangle \alpha}{\alpha}}+[u]_{\eta+\alpha-|\gamma|-k\alpha}^{\frac{\eta-\langle \gamma \rangle \alpha}{\eta+\alpha-|\gamma|-k\alpha}}+ 1_{|\gamma_{x}|=1}  (\|\nu\|_{1}')^{\eta-\langle \gamma \rangle \alpha}+1_{|\gamma_{x}|=1}([\nu]'_{\eta-(\langle \gamma\rangle+k)\alpha})^{\frac{\eta-\langle \gamma \rangle \alpha}{\eta-(\langle \gamma \rangle+k)\alpha}} \nonumber,
\end{align}
where the second line above is the result of Young's inequality.  Note that $\eta+\alpha-|\gamma|-k\alpha \leq \eta-\langle \gamma \rangle \alpha$, simply using $k \geq 0$ if $|\gamma_{x}|=1$ and $k \geq 1$ if $|\gamma_{x}|=0$, so the above is contained in $M_{\eta,|\gamma|}$.

\medskip

The contribution to \eqref{e9} from  \eqref{e11} is estimated similarly, but requires the following additional estimate
\begin{equation}
|f(\txv).P_{\beta}\tau^{(j-1-|\betaa|,\betaa)}_{\syv \txv} | \lesssim d^{|\beta|+(j-2)\alpha-|\betaa| }(\syv,\txv)\text{dist}_{\txv}^{-|\beta|-(j-2)\alpha} 
\begin{cases}
[\Pi]^{j-1}(\|\nu\|_1')^{|\beta_{x}|} , &\quad |\beta| \neq 1  \\
[\Pi]^{j-2}\|\nu\|_1', &\quad |\beta|=1 
\end{cases}
\label{e195}
\end{equation}
which follows from \eqref{e194}, \eqref{model40}, and \eqref{jc28}. In the case $|\beta|=1$, the improved exponent on $[\Pi]$ is a consequence of \eqref{jc03b}.  Indeed, if $|\betaa|=1$, then the left-hand side vanishes by Assumption \ref{jc73}, while if $|\betaa|=0$, Assumption \ref{sne1} implies via \eqref{jc03b} that the factor of $[\Pi]^{j-1}$ in \eqref{e195} can be replaced by $[\Pi]^{j-2}$ since $\beta_1+\ldots+\beta_{j-1}=\beta$ and $|\beta|=1$ implies that $|\beta_i|=1$ for one $i\in\set{1,\ldots,j-1}$ and $\beta_k=0$ for all $k\ne i$.  

\medskip

Applying \eqref{e195}, the contribution to \eqref{e9} from a summand in  \eqref{e11}   with $|\gamma|+|\beta|+\mu(k,\betaa)<\eta+2\alpha$, $|\beta| \neq1$ is estimated by $d^{\eta-|\gamma|}(\syv,\txv)\dist_\txv^{\langle\gamma\rangle\alpha-\eta}$ multiplied by 
\begin{align}
   &(1-1_{|\gamma_{x}|=0,|\betaa|=1})[u]_{\eta+2\alpha-|\gamma|-|\beta|-\mu(k,\betaa)}[u]_{\alpha}^{k-j+|\betaa|}(\|\nu\|_{1}')^{|\beta_{x}|+|\gamma_x|(1-|\betaa|)} [\Pi]^{j-1}   \nonumber \\
   & \lesssim [u]_\alpha^{\frac{\eta-\langle \gamma\rangle \alpha}{\alpha}}+[u]_{\eta+2\alpha-|\gamma|-|\beta|-\mu(k,\betaa)}^{\frac{\eta-\langle \gamma\rangle \alpha}{\eta+2\alpha-|\gamma|-|\beta|-\mu(k,\betaa)}}+(\|\nu\|_{1}')^{\eta-\langle \gamma \rangle \alpha}+[\Pi]^{\frac{(j-1)(\eta-\langle \gamma \rangle \alpha) }{(j+\langle \beta \rangle-2)\alpha} }  \nonumber,
\end{align}
where we used Young's inequality.  Notice that for $\beta$ satisfying the constraints above it holds $\langle \beta \rangle \geq 1$, so that again by Young's inequality it holds $[\Pi]^{\frac{(j-1)(\eta-\langle \gamma \rangle \alpha) }{(j+\langle \beta \rangle-2)\alpha} } \leq 1+{\Pi}^{\frac{\eta-\langle \gamma \rangle \alpha}{\alpha}}$.  Finally, we note that if $|\beta|=1$ so that $\langle \beta \rangle=0$, the same estimate holds by the same argument, simply accounting for the change in the exponent of $[\Pi]$ resulting from the second case in \eqref{e195}.
\end{proof}

\subsection{Interpolation Inequalities}
The final ingredient to pass from the graded continuity lemma to the continuity lemma is interpolation inequalities.
\begin{lemma}\label{lem: positiveInt}
Let $D\subset\R^d$ be a domain.
Let $u:D\to\R$ be a smooth function and define $\nu:D\to\R$ via \eqref{jc71}.
Let $\alpha<\kappa<\eta<2$. For all $\delta \in (0,\frac{1}{2})$ it holds
\begin{align}
[u]_{\kappa} &\lesssim  [u]_{\eta}^{\frac{\kappa}{\eta}}(\|u\|+\delta^{\alpha}[\llhs] )^{1-\frac{\kappa}{\eta}}+\big ( \|u\|+\delta^{\alpha}[\llhs] \big )(\delta^{-1} \vee [\llhs]^{\frac{1}{\alpha}})^{\kappa}, \label{e174} \\
 \|\nu\|_{1}' &\lesssim  [u]_{\eta}^{\frac{1}{\eta}}(\|u\|+\delta^{\alpha}[\llhs] )^{1-\frac{1}{\eta}}+\big ( \|u\|+\delta^{\alpha}[\llhs] \big )(\delta^{-1} \vee [\llhs]^{\frac{1}{\alpha}}) \qquad \qquad && \eta>1, \label{e175} \\
  [\nu]_{\kappa}' &\lesssim ([\nu]_{\eta}')^{\frac{\kappa-1}{\eta-1}}\big (\| \nu\|_{1}'+\delta^\alpha [\Pi] \big )^{\frac{\eta-\kappa}{\eta-1}} + (\|\nu\|_1'+\delta^\alpha[\llhs])\big ( \delta^{-1} \vee [\llhs]^{\frac{1}{\alpha}} \big )^{\kappa-1} && \kappa>1. \label{e176}
\end{align}
The implicit constants are universal and independent of $\delta$.
\end{lemma}
\begin{proof}
We start with the following claim: for any $R\in(0,1)$ the following inequalities hold
\begin{align}
    [u]_{\kappa} &\lesssim  [u]_\eta R^{\eta-\kappa}+\big ( \|u\| + R^{\alpha}[\Pi]  \big )R^{-\kappa}+ \|\nu\|_1' R^{1-\kappa}+  [\Pi]\|\nu\|_{1}' R^{1+\alpha-\kappa}, \label{e172} \\
\|\nu\|_{1}' &\lesssim [u]_{\eta}R^{\eta-1}+(\|u\|+ R^{\alpha}[\Pi])R^{-1}+[\llhs]\|\nu\|_1'R^{\alpha}, \label{e173} \\
[\nu]_{\kappa} &\lesssim [\nu]_\eta R^{\eta-\kappa}+  \big ( \|\nu\|_1' + R^{\alpha}[\Pi]  \big )R^{1-\kappa}  + [\Pi]\|\nu\|_{1}' R^{1+\alpha-\kappa}\label{e171}.
\end{align}
We recall in advance that $\|\primz_{\txv \syv}\|_{\TT_{|\beta|}} \leq \dist_\txv^{|\beta_x|}[\llhs]\big ( \frac{d(\txv,\syv)}{\dist_{\txv}} \big )^{|\beta|}$
for all $\beta\in \Aa{}$, and hence in particular $|\nu(x)|\|\primz_{\txv \syv}\|_{\TT_{|\beta|}} \leq \|\nu\|_{1}'[\Pi]\big ( \frac{d(\txv,\syv)}{\dist_{\txv}} \big )^{|\beta|}$, cf.~\eqref{js101a}.
The reader should also keep in mind that by convention $|da(u)^{\beta}| \lesssim 1$ for all $\beta\in \Aa{}$.
To show \eqref{e172}, pick $R\in(0,1)$ and $\txv,\syv\in D$ with $\syv\in B_{\dist_\txv}(\txv)$.
Notice that for $\frac{d(\txv,\syv)}{\dist_{\txv}} \geq R$ it holds that 
\begin{align}
&\big ( \frac{\dist_{\txv}}{d(\txv,\syv)} \big )^{\kappa} \big |u(\syv)-u(\txv)-f_{\kappa}(\txv).\primz_{\syv\txv} \big | \nonumber \\
&\lesssim R^{-\kappa}\|u\|+[\llhs] \sum_{\substack{\alpha\le|\beta|<\kappa \\|\beta_{x}|=0}}R^{|\beta|-\kappa}+R^{1-\kappa}\|\nu\|_1'+[\Pi]\|\nu\|_1' \sum_{\substack{1<|\beta|<\kappa\\ |\beta_{x}|=1}}R^{|\beta|-\kappa}. \nonumber
\end{align}
Alternatively, for $\frac{d(\txv,\syv)}{\dist_{\txv}} \leq R$ it holds that
\begin{align}
&\big ( \frac{\dist_{\txv}}{d(\txv,\syv)} \big )^{\kappa} \big |u(\syv)-u(\txv)-f_{\kappa}(\txv).\primz_{\syv\txv} \big | \nonumber \\
&\le \big ( \frac{\dist_{\txv}}{d(\txv,\syv)} \big )^{\kappa} \big |u(\syv)-u(\txv)-f_{\eta}(\txv).\primz_{\syv\txv} \big | +\big ( \frac{\dist_{\txv}}{d(\txv,\syv)} \big )^{\kappa} \big |(f_{\eta}-f_{\kappa})(\txv).\primz_{\syv\txv} \big | \nonumber \\
& \lesssim R^{\eta-\kappa}[u]_{\eta}+[\Pi] \sum_{\substack{\kappa \leq |\beta|<\eta\\ |\beta_x|=0}} R^{|\beta|-\kappa}+[\Pi]\|\nu\|_{1}' \sum_{\substack{\kappa \leq |\beta|<\eta\\|\beta_{x}|=1}} R^{|\beta|-\kappa} \nonumber.
\end{align}
Combining the two observations and using that $R^{|\beta|} \leq R^{\alpha}$ for $|\beta|>0$ and
$R^{|\beta|} \leq R^{1+\alpha}$ for $|\beta|>1, |\beta_{x}|=1$, we obtain \eqref{e172}.
The bound \eqref{e171} follows by an analogous argument.
To show \eqref{e173}, given $\txv \in D$ and $1\le i\le d$, let $\syv:=\txv+R\dist_{\txv}e_{i} \in D$, so that $\frac{d(\txv,\syv)}{\dist_{\txv}}=R$.  
Applying this, we find with $\dist_{\txv}|\nu(\txv)|\le \|\nu\|_1'$ that
\begin{align}
R &\dist_{\txv}|\nu_{i}(\txv)|=|\nu_{i}(\txv)(y-x)_{i} | \nonumber \\
&\lesssim |u(\syv)-u(\txv)-f_{\eta}(\txv).\primz_{\syv\txv}|\|u\|+[\llhs]\sum_{\substack{\alpha\le|\beta|<\eta\\ |\beta_x|=0}}  R^{|\beta|} +[\llhs] \sum_{\substack{1+\alpha\le |\beta|<\eta \\ |\beta_x|=1}}(\|\nu\|_{1}')^{|\beta_{x}|} R^{|\beta|}\nonumber \\
&\lesssim [u]_{\eta}R^{\eta}+\|u\|+ R^{\alpha}[\Pi]+[\llhs]\|\nu\|_1'R^{1+\alpha} \nonumber.
\end{align}
Taking the supremum over $1 \leq i \leq d$ and dividing by $R$ yields \eqref{e173}.

\medskip

Recall that $\delta \in (0,\frac{1}{2})$ has been fixed in advance.  We now claim that there is a universal $\epsilon>0$ that for any for any $R$ satisfying both
\begin{equation}
    R \leq \delta, \quad R^{\alpha}[\Pi] \leq \epsilon
\label{e4},
\end{equation}
the following inequalities hold
\begin{align}
     [u]_{\kappa} &\lesssim [u]_{\eta}R^{\eta-\kappa}+ \big ( \|u\|+ \delta^\alpha [\Pi] \big )R^{-\kappa} \label{e17}, \\
    \|\nu \|_{1}' &\lesssim  [u]_{\eta}R^{\eta-1}+ \big (\|u\|+\delta^{\alpha}[\Pi] \big )R^{-1} \label{e193}, \\
    [\nu]_{\kappa} &\lesssim [\nu]_{\eta}R^{\eta-\kappa}+ \big ( \|\nu\|_1'+ \delta^\alpha [\Pi] \big )R^{1-\kappa} \label{e170}.
\end{align}
Indeed, it is clear from \eqref{e173} that a sufficiently small and universal $\epsilon$ can be chosen to obtain \eqref{e193}.  Now \eqref{e17} and \eqref{e170} follow in virtue of \eqref{e4} by inserting \eqref{e193} into \eqref{e172} and \eqref{e171}.

\medskip

We would like to choose $R$ to balance the two terms in \eqref{e17} and \eqref{e193}.
In both cases this corresponds to choosing $R^{\eta}=(\|u\|+\delta^{\alpha}[\Pi] )[u]_{\eta}^{-1}$.  If \eqref{e4} holds with this choice of $R$, then we obtain the (homogeneous) interpolation inequalities
\begin{align}    
[u]_{\kappa} \lesssim  [u]_{\eta}^{\frac{\kappa}{\eta}}(\|u\|+\delta^{\alpha}[\llhs] )^{1-\frac{\kappa}{\eta}}, \quad
\|\nu\|_{1}' \lesssim   [u]_{\eta}^{\frac{1}{\eta}}(\|u\|+\delta^{\alpha}[\llhs] )^{1-\frac{1}{\eta}}, \nonumber
\end{align} 
which imply \eqref{e174} and \eqref{e175}.  If instead \eqref{e4} fails, then it follows that 
\begin{align}
    [u]_{\eta} \leq \big ( \|u\|+\delta^{\alpha}[\llhs] \big ) [\delta^{-\eta} \vee (\epsilon^{-1}[\llhs])^{\frac{\eta}{\alpha}}] 
     \lesssim \big ( \|u\|+\delta^{\alpha}[\llhs] \big ) (\delta^{-1} \vee [\llhs]^{\frac{1}{\alpha}})^{\eta} \label{e177} .
\end{align}
In this case, we apply \eqref{e193} and \eqref{e17} with $R=\delta \wedge \big (  \epsilon[\Pi]^{-1} \big )^{\frac{1}{\alpha}} \approx \delta \wedge [\Pi]^{-\frac{1}{\alpha}}$, which satisfies \eqref{e4} by design, then we insert \eqref{e177}.  
This yields
\begin{equation}
\|\nu \|_{1}' \lesssim \big ( \|u\|+\delta^{\alpha}[\llhs] \big )[R^{\eta-1}(\delta^{-1} \vee [\Pi]^{\frac{1}{\alpha}})^\eta +R^{-1}] \lesssim \big ( \|u\|+\delta^{\alpha}[\llhs] \big )(\delta^{-1} \vee [\llhs]^{\frac{1}{\alpha}}), \nonumber     
\end{equation}
which implies \eqref{e175}.  Similarly, we obtain 
\begin{equation}
[u]_{\kappa} \lesssim \big ( \|u\|+\delta^{\alpha}[\llhs] \big )(\delta^{-1} \vee [\llhs]^{\frac{1}{\alpha}})^{\kappa} \nonumber,
\end{equation}
which implies \eqref{e174}. 
This completes the proof of the estimates on $[u]_\kappa$ and $\|\nu\|_1'$.

\medskip
The proof of \eqref{e176} follows a similar argument.
To balance terms in \eqref{e170}, we would need to choose 
$R^{\eta-1}=(\|\nu\|_1'+\delta^\alpha [\llhs])[\nu]_{\eta}^{-1}$.  If \eqref{e4} holds with this choice of $R$, we find
\begin{equation}
[\nu]_{\kappa}' \lesssim \big (\| \nu\|_{1}'+\delta^\alpha [\Pi] \big )^{\frac{\eta-\kappa}{\eta-1}}([\nu]_{\eta}')^{\frac{\kappa-1}{\eta-1}},\nonumber
\end{equation}
which implies \eqref{e176}. Otherwise, we find that
\begin{align}
    [\nu]_{\eta} 
     \lesssim \big ( \| \nu\|_{1}'+\delta^{\alpha}[\llhs] \big ) (\delta^{-1} \vee [\llhs]^{\frac{1}{\alpha}})^{\eta-1}, \nonumber
\end{align}
 so that choosing again $R \approx \delta \wedge [\Pi]^{-\frac{1}{\alpha}}$ leads us to 
\begin{equation}
    [\nu]_{\kappa}  \lesssim \big ( \| \nu\|_{1}'+\delta^{\alpha}[\llhs] \big ) (\delta^{-1} \vee [\llhs]^{\frac{1}{\alpha}})^{\kappa-1}, \nonumber
\end{equation}
which implies \eqref{e176}.  
\end{proof}

\subsection{Proof of the Continuity Lemma}
\begin{proof}[Proof of the \autoref{lem1}]
In light of the Graded Continuity Lemma (applied with $\eta+\alpha$ in place of $\eta$), it suffices to show that for all $|\gamma|<\eta+\alpha$ with $\langle \gamma \rangle \geq 1$ (since $\tau \in \TT_{-} \cap \TT_{|\gamma|}$ implies $\langle \gamma\rangle \geq 1$) it holds
\begin{align}
M_{\eta+\alpha,|\gamma|} &\lesssim [u]_{\eta}+  [\nu]'_{\eta} + \delta^{-\eta}    \label{e185}.
\end{align}
We will make use of the following interpolation inequalities:
\begin{align}
[u]_{\kappa} \lesssim [u]_{\eta}^{\frac{\kappa}{\eta}}+\delta^{-\kappa} , \quad
 \|\nu\|_{1}' \lesssim  [u]_{\eta}^{\frac{1}{\eta}}+\delta^{-1}  , \quad [\nu]_{\kappa}' \lesssim ([u]_{\eta} \vee [\nu]_{\eta}' )^{\frac{\kappa}{\eta} }+\delta^{-\kappa}.   \label{e182}
\end{align}
The first two follow immediately from \eqref{e174}-\eqref{e175} in light of our assumption $\|u\|+\delta^{\alpha}[\Pi] \leq 1$.  We now argue the third inequality.   
Indeed, inserting the second inequality in \eqref{e182} into \eqref{e176} and using $\delta^{\alpha}[\Pi] \leq 1 \leq \delta^{-1}$ we find
\begin{align}
[\nu]_{\kappa}' &\lesssim ([\nu]_{\eta}')^{\frac{\kappa-1}{\eta-1}}\big ( [u]_{\eta}^{\frac{1}{\eta}}+\delta^{-1}  \big )^{\frac{\eta-\kappa}{\eta-1}} + \big ( [u]_{\eta}^{\frac{1}{\eta}}+\delta^{-1}  \big ) \delta^{-(\kappa-1)} \nonumber \\
&\lesssim  \big ( [u]_{\eta} \vee [\nu]_{\eta}' \big )^{\frac{\kappa}{\eta}}+([\nu]_{\eta}')^{\frac{\kappa-1}{\eta-1}} \delta^{-\frac{\eta-\kappa}{\eta-1}}+[u]_{\eta}^{\frac{1}{\eta}} \delta^{-(\kappa-1) }+\delta^{-\kappa} \nonumber,
\end{align}
which implies the claim via Young's inequality applied to the second and third terms.

\medskip

In particular, from \eqref{e182} it follows that for any exponent $p \leq \eta$ it holds
\begin{equation}
[u]_{\kappa}^{\frac{p}{\kappa}}+(\|\nu\|_{1}')^{p}+([\nu]_{\kappa}')^{\frac{p}{\kappa}} \lesssim [u]_{\eta}+  [\nu]'_{\eta} +  \delta^{-\eta}  \label{e186}.
\end{equation}
Hence, using \eqref{e197} followed by \eqref{e186} with $p=\eta-(\langle \gamma \rangle-1)\alpha \leq \eta$,  we obtain \eqref{e185}.
\end{proof}

\section{Jets and Proof of the Main Theorem}\label{sec_jets}
In this section, we present the two main analytic ingredients required for our method, reconstruction and integration in the language of \cite{Hai14}, then show how to apply them to deduce our main result.
The general strategy is to prove \eqref{jc72}, 
from which the (weaker) bound \eqref{jc63} easily follows from our interpolation inequalities, cf.~\eqref{e174}.
Hence, our main focus is on estimating the $[u]_{\eta}$ and $[\nu]_{\eta}'$ semi-norms, which we view abstractly as a semi-norm on a specific jet of smooth functions.
By a jet, we mean a family of functions $U_\txv:\domain\to\R$ indexed by a base point $\txv\in \domain\subset \R^{\di+1}$.
These are used to describe the remainder that appears in the formal identities \eqref{q2} and \eqref{q12} when the linear form $f(\txv)$ is replaced by its truncated version $f_{\eta}(\txv)$.   
More precisely, given $u$ satisfying \eqref{intro3} and $\nu$ defined by \eqref{jc71}, we define
\begin{align}\label{def_UF}
\begin{split}
U_\txv&:=u - u(\txv) - f_{\eta}(\txv).\llhs_{\txv},\\
F_\txv&:=\big ( a(u)-a(u(\txv)) \big )\Deltap u+\xi - h(u) - f_{\eta+\alpha}(\txv).\lrhs_{\txv},
\end{split}
\end{align}
where $\eta \in (1,2)$ will be chosen sufficiently close to $2$ depending on how close $\alpha$ is to $0$, cf.~\eqref{jc40}.
Our goal is then to estimate $\text{dist}_{\txv}^{\eta}d^{-\eta}(\syv,\txv)|U_\txv(\syv)|$, which is done using Proposition \ref{int lem}, a generalization of classical Schauder theory to jets.

\medskip

To describe the input for Proposition \ref{int lem}, let us fix our notational conventions for convolutions: we say that $\convker$ is a symmetric convolution kernel if it is a Schwartz function with integral $1$ satisfying $\convker(t,x)=\convker(t,-x)$.
For a fixed $\rho$, we use $\lambda>0$ to denote a convolution parameter and write $(\cdot)_{\mop}$ for the convolution with $\convker_\mop$, where $\convker_\mop(t,x):=\mop^{-(\di+2)}\convker(\mop^{-2}t,\mop^{-1}x)$.  Specifically, given a (regular) tempered distribution $F$ and a kernel $\convker$, we define $F_{\lambda}(\txv):=\int_{\R^{\di+1}} F(\syv)\rho_{\lambda}(\txv-\syv)d\syv$, 
and omit the specific kernel from the notation.\footnote{In the notation of Hairer \cite{Hai14}, $F_{\lambda}(\txv)=\langle F, \convker_{\txv}^{\mop} \rangle$.}

\medskip

The most important input for Proposition \ref{int lem} is \eqref{KS1New}, which requires for each base-point $\txv \in \domain$ a control on $\text{dist}_{\txv}^{\eta} \lambda^{2-\eta}|(\partial_{t}-a(u(\txv))\Delta)U_{\txv\lambda }(\txv)|$, the natural extension of the $C^{\eta-2}$ semi-norm applied to the jet $\{ (\partial_{t}-a(u(\txv))\Delta)U_{\txv}\}_{\txv}$.  At this point,
the compatibility between $\Pi_{\mathbf{x}}$ and $\Pi_{\mathbf{x}}^{-}$ in the form of \eqref{jc60} becomes important, as it
leads to a compatibility between $U_\txv$ and $F_\txv$ up to a correction.  Specifically, it follows from \eqref{jc60} and \eqref{intro3} that
\begin{equation}\label{js902}
(\partial_{t}-a(u(\txv))\Delta)U_{\mathbf{x}}=F_{\mathbf{x}}+(f_{\eta}(\mathbf{x})-f_{\eta+\alpha}(\mathbf{x}) ).\Pi_{\mathbf{x}}^{-}.
\end{equation}
The second term on the right-hand side is straightforward to estimate to order $\eta-2$ since it only involves $\Pi^{-}_{\txv\beta}$ for $|\beta| \geq \eta$, and we can easily combine \eqref{jc60} and \eqref{model7a} for this purpose, cf.~\eqref{model7}.  Hence, our main task is to estimate $\dist_{\txv}^{\eta} \lambda^{2-\eta}|F_{\txv \lambda}(\txv)|$, which is where Proposition \ref{rec_lem} comes into play.  The main input for Proposition \ref{rec_lem} is \eqref{R1}, which requires a bound on $\big | F_{\txv \mu}(\syv)-F_{\syv \mu}(\syv) \big |$ in terms of a sum of terms of the form $d^{\kappa_{1}}(\syv,\txv)\mu^{\kappa_{2}}$ where $\kappa_{1}+\kappa_{2}=\delta>0$.  It follows immediately from the definition \eqref{def_UF} that 
\begin{equation}
F_{\txv \mu}(\syv)-F_{\syv \mu}(\syv)=(a(u(\syv))-a(u(\txv)) )\Delta u_{\mu}(\syv)+f_{\eta+\alpha}(\syv).\Pi_{\syv \mu}^{-}(\syv)-f_{\eta+\alpha}(\txv).\Pi^{-}_{\txv \mu}(\syv). \end{equation}
This identity needs to be further re-arranged: for instance the first term could at best be bounded by $[u]_{\alpha}^{2}d^{\alpha}(\syv,\txv)\mu^{\alpha-2}$, but adding the exponents gives $2\alpha-2<0$ for $\alpha<1$.  Fortunately, there are many cancellations between the above quantities, and we establish that in the following way.  We start by arguing the following change of base-point for the negative model
\begin{equation}
\lrhs_{\txv}=\Gamma_{\syv\txv}\lrhs_{\syv}+\sum_{k \geq 1}\polv_k\llhs_\txv(\syv)^k \Delta\llhs_{\txv}. \end{equation}
Inserting this above we find that
\begin{align}
F_{\txv \mu}(\syv)-F_{\syv \mu}(\syv)&=\big(f_{\eta+\alpha}(\syv).\id - f_{\eta+\alpha}(\txv).\Gamma_{\syv\txv}\big)\lrhs_\syv \nonumber \\
&+\big(a(u(\syv))-a(u(\txv) ) \big )\Delta u_{\mu}(\syv)-\sum_{k\ge 1}f_{\eta+\alpha}(\txv).(\polv_k\llhs_\txv(\syv)^k \Delta\llhs_{\txv}) \nonumber.
\end{align}
The first term is exactly tailored to the modelled distribution norm \eqref{md19} and leads to a sum of terms of the form $\opnorm{f_{\eta+\alpha}}[\Pi]d^{\eta+\alpha-|\beta|}(\syv,\txv)\mu^{|\beta|-2}$, which explains the role of the \autoref{lem1}.
The second contribution above is precisely the truncated version of \eqref{q7}.
We show below that by a Taylor expansion and further application of Lemma \ref{lem_identity} (similar to the proof of the Graded Continuity Lemma), this quantity can be estimated by a sum of terms: the ones of highest order taking the form $[u]_{\kappa}^{\frac{\eta+\alpha}{\kappa}}d^{\eta+\alpha-|\beta|}(\syv,\txv) \Mop^{|\beta|-2}$, where $\alpha \leq \kappa \leq \eta$.
This suggests to apply Proposition \ref{rec_lem} with $\delta=\eta+\alpha-2$ provided that $\eta>2-\alpha$.
The output from reconstruction is of order $\eta+\alpha-2$, which in particular implies a control of order $\eta-2$.
If we take the above estimates as an input for integration, we would find 
\begin{equation}
[u]_{\eta}+[\nu]_{\eta}' \lesssim ([u]_{\eta}+[\nu]_{\eta}')[\Pi]+\sum_{\alpha \leq \kappa \leq \eta} [u]_{\kappa}^{\frac{\eta+\alpha}{\kappa}}+\text{l.o.t} \label{sFin6},  \end{equation}
where the first term comes from the \autoref{lem1}.  Unfortunately, even with the help of our interpolation inequalities, this estimate does not close without an additional a priori smallness constraint on $[u]_{\eta}$ and $[\Pi]$. For typical pertubative results, this is usually achieved via a continuity argument or exploiting an additional small factor of the time interval. 
However, our main result is of a different character and takes for granted only the rather weak input: smallness of $\|u\|$.
To improve on \eqref{sFin6}, we need to be more careful in passing from the $\eta+\alpha-2$ output from reconstruction to the $\eta-2$ input required for integration.
Namely, in Lemma \ref{jq007} we interpolate between the (optimal)  $\eta+\alpha-2>0$ description (on small scales) with a rather coarse $\alpha-2$ bound (on large scales) to obtain an $\eta-2<0$ description where the quantities on the right-hand side of \eqref{sFin6} are effectively raised to the power $\frac{\eta}{\eta+\alpha}$, allowing us to apply our interpolation inequalities and  buckle under smallness of $\|u\|$.

\medskip

\subsection{Approximation by Jets}\label{sec:jets}
We are interested in jets that are uniformly locally bounded, which we monitor using the quantity
\begin{align}\label{jc79}
 \Iu{U}&:=\sup\setc{|U_\txv(\syv)|}{\txv\in \domain, \syv\in B_{\dist_{\txv}}(\txv)}, 
\end{align}
where we recall the shorthand-notation $\dist_\txv:=\dist(\txv,\partial\domain)$ for the distance of $\txv$ from the boundary of $\domain$. Moreover, we measure higher regularity of order $\eta>0$ via the weighted quantities
\begin{align}\label{jc78}
\begin{split}
 [U]_{\eta}&:=\sup\setc{\dist_{\txv}^\eta \frac{|U_\txv(\syv)|}{d^\eta(\syv,\txv)}}{\txv\in \domain, \syv\in B_{\dist_{\txv}}(\txv)}, \\
 [U]_{\eta}'&:=\sup\setc{\dist_{\txv}^\eta \frac{|U_\txv(\syv)|}{d^{\eta-1}(\syv,\txv)}}{\txv\in \domain, \syv\in B_{\frac12\dist_{\txv}}(\txv)}.
\end{split}
\end{align}
The second definition is used to monitor jets related to analogues of Gubinelli derivatives, which explains the subscript $\eta$ despite $U$ being measured against $d^{\eta-1}(\syv,\txv)$. It also hints to why the supremum over $\syv$ is taken over a smaller ball.  The definition \eqref{jc78} is designed to be compatible with \eqref{js100} (for a specific choice of jet).

\medskip

We first cite a local reconstruction assertion, which can be found in essentially this form in \cite{MoW20}; it is a local version of the reconstruction theorem in \cite{OtW19}, both being inspired by related results in \cite{Hai14}.

\medskip

\begin{proposition}\label{rec_lem} (Reconstruction)
Let $\delta>0$ and $\mathsf{K} \subset (-\infty,\delta)$ finite. There is a symmetric convolution kernel $\convker$ with $\supp\convker\subset B_1(0)$ with the following property. Fix $\mathbf{z}\in \domain$, $\mop\in(0,1)$ with $\mop<\dist_\rzv$. Assume that for a jet of smooth functions $\{F_\txv\}_{\txv}$ such that $F_\rzv(\rzv)=0$, there is $C>0$ such that for all $\Mop\in (0,\mop)$ and $\txv,\syv\in B_{\mop-\Mop}(\rzv)$ it holds
\begin{equation}
\big | F_{\mathbf{x}\Mop}(\mathbf{y})-F_{\mathbf{y}\Mop}(\mathbf{y}) \big | \leq C \sum_{\kappa \in \mathsf{K}}d^{\delta-\kappa}(\syv,\txv) \TMc{\kappa}\label{R1},
\end{equation}
Then we have
\begin{equation}
\big |F_{\mathbf{z}\mop}(\mathbf{z} ) \big | \lesssim C \Tc{\delta}\label{R2},
\end{equation}
where the implicit constant depends only on $\eta$, $\mathsf{K}$, and the dimension $\di$.
\end{proposition}
The following proposition on integration of jets is a local variant of Lemma 5 in \cite{otto2018parabolic} and extends Lemma 2.11 in \cite{MoW20}. Within the proof of the \autoref{theorem1}, the output of Proposition \ref{rec_lem} will be used as an input of Proposition \ref{int lem}.
At first glance, the conditions of Proposition \ref{int lem} appear to be stronger than expected if compared to integration lemmas from multi-level Schauder theory that are used in semi-linear contexts, as we have to include the continuity condition \eqref{KS2}.
However, we emphasize that condition \eqref{KS1New} is rather weak: It does not rely at all on regularity information of the coefficient field $a$ relative to the base point or the convolution scale.
This condition alone can therefore not suffice to give a result that is applicable to quasi-linear equations.
Proposition \ref{int lem} can thus be understood as providing a surprisingly weak supplementary condition to \eqref{KS1New}, namely the three-point continuity \eqref{KS2}, which is sufficient to conclude the regularity statement \eqref{KS34}.
\begin{proposition}[Integration]\label{int lem}
Let $\el\in(0,1)$, $\eta\in  (1,2)$ and let $\Aa{}\subset (0,\eta)$ be finite.  Consider a jet of smooth functions $\{U_\txv \}_{\txv}$ on $\domain$ such that for all $\txv\in \domain$ it holds $U_\txv(\txv)=0$ and $\nablap_{\yv}|_{\syv=\txv}U_\txv(\syv)=0$, and assume that $[U]_{\eta}<\infty$.
Let $\convker$ be a symmetric convolution kernel with compact support in $B_1(0)$, and let the following two conditions be satisfied for some $M>0$.
\begin{enumerate}
\item For all $\syv\in \domain$ and $\lambda \in (0,\frac{1}{10} \dist_{\syv} )$ it holds
\begin{equation}
\inf_{a_{0}, c_{0}}\dist_{\syv}^{\eta}\big |(\partial_{t}-a_{0}\Delta)U_{\mathbf{y}\lambda}(\mathbf{y})-c_{0} \big | \leq M \lambda^{\eta-2}\label{KS1New}
\end{equation}
where the infimum runs over all $a_0\in I:=[\el,\el^{-1}]$ and
all constants $c_0\in\R$.
\item (Three-point continuity) For all $\txv\in \domain$, $\syv\in B_{\frac12\dist_{\txv}}(\txv)$, $\rzv\in B_{\frac12\dist_{\txv}}(\syv)$ it holds
\begin{align} 
\begin{split}
\dist_{\txv}^{\eta}\big | U_\txv(\rzv)- & U_\txv(\syv)-U_\syv(\rzv) -\gamma_\txv(\syv)\cdot (\zv-\yv) \big | 
 \le  M \sum_{\kappa \in \Aa{}}d^{\eta-\kappa}(\syv,\txv)d^{\kappa}(\rzv,\syv), \label{KS2}
\end{split}
\end{align}
for some jet $\{\gamma_\txv\}_{\txv}$ with $\gamma_{\txv}: \domain\to\R^{\di}$.
\end{enumerate}
Then it holds
\begin{align}\label{KS34}
 [U]_{\eta} + [\gamma]_{\eta}' \lesssim M + \Iu{U}.
\end{align}
Here the implicit constant in \eqref{KS34} depends only on $\el$, $\eta$, $\Aa{}$, the dimension $\di$ and the convolution kernel $\convker$.
\end{proposition}
\begin{proof}
 \newcounter{PDE_P} 
\refstepcounter{PDE_P} 

\medskip

{\sc Step} \arabic{PDE_P}.\label{localSplitStep}\refstepcounter{PDE_P}
We claim that for all $\txv\in \domain$, and all $\mop\in (0,\frac{1}{10}\dist_\txv)$, $R\in (0,\frac12\dist_\txv)$ with $\mop\le \frac12 R$ it holds
\begin{align}
\dist_{\txv}^{\eta} \inf_{a_{0}\in I, c_{0}\in\R} \|(\partial_{\sv}-a_{0}\Deltap ) & U_{\txv\mop}-c_{0}\|_{B_{R}(\txv)} \le  M \Tc{\alpha-2}(\mop+R)^{\eta-\alpha},\label{KS1}
\end{align}
where $\|\cdot\|_M$ denotes the supremum norm restricted to a subset $M\subset \R^{\di+1}$.
To this end, we let $\txv,\syv$ satisfy $d(\syv,\txv) \leq R$ and write
\begin{align*}
&(\partial_{t}-a_{0}\Delta)(U_{\txv\mop}-U_{\syv\mop})(\syv) \\
&=\int \big ( U_\txv(\rzv)-U_\txv(\syv)-U_\syv(\rzv) -\gamma_\txv(\syv)\cdot (\zv-\yv) \big)(\partial_{t}-a_{0}\Delta)\rho_{\lambda}(\mathbf{y}-\mathbf{z})d \mathbf{z}.
\end{align*}
Using the three-point continuity condition \eqref{KS2}, which is valid since $R, \lambda <\frac{1}{2}\dist_{\txv}$ we find that 
\begin{align}
&\dist_{\txv}^{\eta}\big | (\partial_{\sv}-a_{0}\Delta)(U_{\mathbf{x} \lambda}-U_{\mathbf{y} \lambda})(\mathbf{y}) \big | \nonumber \\
& \lesssim M \sum_{\kappa \in \Aa{}}\int d^{\eta-\kappa}(\syv,\txv)d^{\kappa}(\rzv,\syv)|(\partial_{t}-a_{0}\Delta)\rho_{\lambda}(\mathbf{y}-\mathbf{z})|d\mathbf{z}
\lesssim M \sum_{\kappa \in \Aa{}} R^{\eta-\kappa}\lambda^{\kappa-2} \label{sFin1}.
\end{align}
To pass from \eqref{sFin1} to \eqref{KS1} we use the triangle inequality and Young's inequality together with \eqref{KS1New} and note that $\frac{\dist_{\mathbf{x}}}{\dist_{\mathbf{y}}}\leq 2 $ as a consequence of $\dist_{\mathbf{y}} \geq \dist_{\mathbf{x}} -R$ and $R \leq \frac{1}{2}\dist_{\mathbf{x}}$.  

\medskip

{\sc Step} \arabic{PDE_P}.\label{pde_scale_St}\refstepcounter{PDE_P} We claim that for all base points $\txv\in \domain$ and scales $\mop\in (0,\frac{1}{10}\dist_\txv)$, $R,L\in (0,\frac12\dist_\txv)$ with $\mop,R\le \frac12 L$ it holds
\begin{align}
\begin{split}
 \left(\frac{\dist_{\txv}}{R}\right)^{\eta}& \inf_{\aff}\|U_{\txv\mop}- \aff\|_{B_R(\txv)}  
  \lesssim \left(\frac{R}{L}\right)^{2-\eta} [U]_{\eta} + \frac{L^2 M}{R^\eta} \Tc{\alpha-2}(\mop+L)^{\eta-\alpha}, \label{p46}
\end{split}
\end{align}
where the infimum runs over all affine functions $\aff$, by which we mean functions of the form $\aff(\syv)=c+\nu\cdot (\yv-\xv)$ for some $c\in\R$ and $\nu\in\R^{\di}$.
Towards this end, we define for $a_0\in I$ and $c_0$ that are near optimal in the estimate \eqref{KS1} a decomposition $U_{\txv\mop}=u_{<}(\cdot) + u_{>}(\cdot)$ by setting $u_{>}$ to be the (decaying) solution to
\begin{align*}
 (\partial_\sv - a_0\Deltap)u_{>} = I(B_L(\txv)) \left((\partial_\sv-a_0\Deltap)U_{\txv\mop}-c_0\right),
\end{align*}
where $I(B_L(\txv))$ is the characteristic function of $B_L(\txv)$.  Observe that on $B_L(\txv)$ it holds
\begin{align}\label{KS131}
 (\partial_\sv - a_0\Deltap)u_{<} =c_0.
\end{align}
By standard estimates for the heat equation and \eqref{KS1} we have
\begin{align}\label{KS400}
 \|u_{>}\|_{B_L(\txv)}&\lesssim L^2 \|(\partial_\sv - a_0\Deltap)U_{\txv\mop}-c_0\|_{B_L(\txv)}
 \le L^2  \dist_{\txv}^{-\eta}M \Tc{\alpha-2} (\mop+L)^{\eta-\alpha}, 
\end{align} 
together with
\begin{align}
 \|\{\partial_\sv,\nablap^2\}u_{<}\|_{B_R(\txv)} &\lesssim  L^{-2} \|u_{<}\|_{B_L(\txv)}, \label{KS5}
\end{align}
where we used that $R\le \frac12L$.
In fact, \eqref{KS5} is slightly non-standard due to the presence of a constant $c_0$ on the right-hand side of \eqref{KS131}. However, this can be reduced to the case $c_0=0$ as observed in \cite[Lemma 3.6]{OtW19}.
Next we define a concrete affine function via $\aff_{<}(\syv):=u_{<}(\txv)+\nablap u_{<}(\txv)\cdot(\yv-\xv)$ and observe that  Taylor's formula, \eqref{KS5} and $R\le L$ give
\begin{align*}
 \|u_{<}-\aff_{<}\|_{B_R(\txv)}&\lesssim R^2\|\partial_\sv u_{<}\|_{B_R(\txv)} + R^{2}\|\nabla^{2} u_{<}\|_{B_R(\txv)} \\
 & \stackrel{\mathclap{\eqref{KS5}}}{\lesssim} \left(\frac{R}{L}\right)^{2} \|u_{<}\|_{B_L(\txv)} \le \left(\frac{R}{L}\right)^{2} \|U_{\txv\mop}\|_{B_L(\txv)} + \|u_{>}\|_{B_L(\txv)}.
\end{align*}
Combining this observation with \eqref{KS400} gives
\begin{align*}
 \|&U_{\txv\mop}-\aff_{<}\|_{B_R(\txv)} \le \|u_{>}\|_{B_R(\txv)} + \|u_{<}-\aff_{<}\|_{B_R(\txv)} \\
 & \lesssim \left(\frac{R}{L}\right)^{2} \|U_{\txv\mop}\|_{B_L(\txv)} + \|u_{>}\|_{B_L(\txv)} \\
 & \lesssim \left(\frac{R}{L}\right)^{2} \|U_{\txv\mop}\|_{B_L(\txv)} + L^2 \dist_{\txv}^{-\eta}M\Tc{\alpha-2}(\mop+L)^{\eta-\alpha},
\end{align*}
which implies \eqref{p46}, since
\begin{align*}
 \frac{1}{L^\eta}\|U_{\txv\mop}\|_{B_L(\txv)} \lesssim \frac{1}{L^\eta}\|U_{\txv}\|_{B_{L+\mop}(\txv)} \lesssim \frac{1}{(2L)^\eta}\|U_\txv\|_{B_{2L}(\txv)} \le \dist_{\txv}^{-\eta}[U]_{\eta}
\end{align*}
by the definition of $[U]_{\eta}$, the constraints on $\mop, L$ and the fact that $U_{\txv}(\txv)=0$ by assumption.

\medskip

{\sc Step} \arabic{PDE_P}.\label{pde_conv_St}\refstepcounter{PDE_P} We claim that for all base points $\txv\in \domain$ and all scales $\mop,R\in(0,\frac12\dist_{\txv})$ it holds
\begin{align}\label{p49}
 \dist_{\txv}^{\eta}\|U_{\txv\mop}- &U_\txv\|_{B_R(\txv)} \lesssim  [U]_{\eta}\Tc{\eta} + M\sum_{\kappa\in\Aa{}} R^{\eta-\kappa}\Tc{\kappa}.
\end{align}
For $\syv\in B_R(\txv)$ we write
\begin{align*}
(U_{\txv\mop}-U_\txv)(\syv)=\int (U_\txv(\rzv)-U_\txv(\syv))\convker_{\mop}(\syv-\rzv){\rm d}\rzv.
\end{align*}
By the symmetry of the convolution kernel under the involution $x\mapsto -x$, we have in particular $\int \nu\cdot (\yv-\xv)\convker_{\mop}(\syv-\rzv){\rm d}\rzv=0$ for any $\nu\in\R^\di$, so that we may rewrite the above identity as
\begin{align*}
(U_{\txv\mop}-U_\txv)(\syv)
&=\int U_\syv(\rzv)\convker_{\mop}(\syv-\rzv){\rm d}\rzv\nonumber\\
&\quad +\int (U_\txv(\rzv)-U_\txv(\syv)-U_\syv(\rzv)  -\gamma_\txv(\syv)\cdot(\zv-\yv))\convker_{\mop}(\syv-\rzv){\rm d}\rzv.
\end{align*}
By the choice of $R$, the triangle inequality and the definition of $\dist_{\txv}$, we have that $\syv\in B_R(\txv)$ implies $\frac12 \dist_{\txv}\le \dist_{\syv}$. Hence, by the choice of $\mop$ and since the support of $\convker$ is contained in $B_1(0)$, we have
$
 d(\rzv,\syv)\le \frac12\dist_{\txv}\le \dist_{\syv},
$
so that $\rzv\in B_{\dist_{\syv}}(\syv)$. Thus, the definition \eqref{jc78} of $[U]_{\eta}$ and \eqref{KS2} give
\begin{align*}
|(U_{\txv\mop}-  U_\txv)(\syv)|&\le \dist_{\syv}^{-\eta} [U]_{\eta}\int d^{\eta}(\rzv,\syv)|\convker_{\mop}(\syv,\rzv)|{\rm d}z\nonumber\\
&+\dist_{\txv}^{-\eta}M\sum_{\kappa\in\Aa{}}d^{\eta-\kappa}(\txv,\syv)\int d^{\kappa}(\rzv,\syv)|\convker_{\mop}(\syv,\rzv)|{\rm d}z.
\end{align*}
This implies by virtue of the scaling properties $\convker_\mop$ and once more $\frac12 \dist_{\txv}\le \dist_{\syv}$ the desired
\begin{align*}
\dist_{\txv}^{\eta}|(U_{\txv\mop}- & U_\txv)(\syv)|\lesssim  [U]_{\eta}\Tc{\eta} +M\sum_{\kappa\in\Aa{}}d^{\eta-\kappa}(\syv,\txv)\Tc{\kappa}.
\end{align*}

\medskip

{\sc Step} \arabic{PDE_P}.\label{pde_normeq_St}\refstepcounter{PDE_P} We claim the norm equivalence
\begin{align}\label{3.8}
 [U]_{\eta} \sim \HHN{U}{\eta},
\end{align}
where we have set
\begin{align}\label{3.6}
 \HHN{U}{\eta}:=\sup_{\txv\in B_1(0)}\dist_{\txv}^{\eta}\sup_{R\in(0,\dist_{\txv})} R^{-\eta}\inf_\aff \|U_\txv-\aff\|_{B_R(\txv)},
\end{align}
and where $\sim$ means that both inequalities with $\lesssim$ and $\gtrsim$ are true. Here, the infimum runs over all affine functions $\aff$.
We first argue that these $\aff$ may be chosen to be independent of $R$, that is, for all $\txv\in \domain$
\begin{align}\label{3.7}
\inf_{\aff}\sup_{R\in(0,\dist_{\txv})}R^{-\eta} \|U_\txv-\aff\|_{B_R(\txv)}\lesssim \dist_{\txv}^{-\eta}\HHN{U}{\eta}=:C,
\end{align}
where we denote the right-hand side momentarily by $C$ for better readability.
Indeed, let $\aff_R(\syv)=c_R + \nu_R\cdot (\yv-\xv)$ be (near) optimal in \eqref{3.6}.
Then by definition of $\HHN{U}{\eta}$ and the triangle inequality, 
\begin{equation*}
R^{-\eta}\|\aff_{2R}-\aff_{R}\|_{B_R(\txv)}\lesssim C.
\end{equation*}
This implies $R^{-(\eta-1)}|\nu_{2R}-\nu_{R}|+R^{-\eta}|c_{2R}-c_R|\lesssim C$.  Since $\eta>1$, telescoping gives $R^{-(\eta-1)}|\nu_{R}-\nu_{R'}|+R^{-\eta}|c_{R}-c_{R'}|\lesssim C$
for all $R'\le R$ and thus the existence of $\nu\in\R^{\di}$ and $c\in\mathbb{R}$ such that 
\begin{equation*}
R^{-(\eta-1)}|\nu_{R}-\nu|+R^{-\eta}|c_{R}-c|\lesssim C,
\end{equation*}
so that $\aff(\syv):=c + \nu \cdot (\yv-\xv)$ satisfies
\begin{align*}
R^{-\eta}\|\aff_R-\aff\|_{B_R(\txv)}\lesssim C.
\end{align*}
Hence we may pass from \eqref{3.6}
to \eqref{3.7} by the triangle inequality.

\medskip

It is clear from \eqref{3.7} and the assumptions on $U$ that necessarily for any $\txv\in \domain$
the optimal $\aff$ must be of the form $\aff(\syv) =0$.  Thus
\begin{align}
\left |U_\txv(\syv)\right | \lesssim \HHN{U}{\eta} \Big(\frac{d(\syv,\txv)}{\dist_\txv}\Big)^{\eta} \label{p43} 
\end{align} 
for $\syv\in B_{\dist_{\txv}}(\txv)$, which establishes the
nontrivial direction of \eqref{3.8}.

\medskip

{\sc Step} \arabic{PDE_P}.\label{pde_conc_St}\refstepcounter{PDE_P} We now give the estimate of $[U]_{\eta}$ in \eqref{KS34}, that is, we will show
\begin{align}\label{jc76}
  [U]_{\eta} \lesssim M + \Iu{U}.
\end{align}
Combining Steps \ref{pde_scale_St} and \ref{pde_conv_St}, we obtain by the triangle inequality for each base point $\txv\in \domain$ and all scales $\mop\in(0,\frac{1}{10}\dist_{\txv})$, $R,L\in(0,\frac12\dist_{\txv})$ with $\mop,R\le \frac12 L$ 
\begin{align*}
 \left(\frac{\dist_{\txv}}{R}\right)^{\eta}\inf_\aff  \|U_\txv-&\aff\|_{B_R(\txv)} \lesssim \HNw{U}{\eta}\left(\bigg(\frac{R}{L}\bigg)^{2-\eta} +\left (\frac{\mop}{R} \right)^{\eta}\right)  \\
   & +  \frac{L^2 M}{R^{\eta}} \Tc{\alpha-2}(\mop+L)^{\eta-\alpha}+ M\sum_{\kappa\in\Aa{}}R^{-\kappa}\Tc{\kappa}.
\end{align*}
Now we link the scales $L$ and $\mop$ to $R$ by introducing a small $\eps\in(0,\frac15)$ and choosing $L=\frac{1}{\eps}R$ and $\mop = \eps R$.  Then for all $R\in (0,\frac{\eps}{2}\dist_{\txv})$ we have
\begin{align*}
 \left(\frac{\dist_{\txv}}{R}\right)^{\eta}&\inf_\aff \|U_\txv-\aff\|_{B_R(\txv)}\lesssim [U]_{\eta}(\eps^{2-\eta} + \eps^{\eta}) + M \big(\eps^{\eta-4}+\eps^{2\alpha-4-\eta}+\sum_{\kappa\in\Aa{}} \eps^{\kappa}\big).
\end{align*}
Since for $R\in [\frac{\eps}{2}\dist_{\txv},\dist_{\txv})$ we have by the definition \eqref{jc79} of $\Iu{U}$
\begin{align*}
 \left(\frac{\dist_{\txv}}{R}\right)^{\eta}\inf_\aff \|U_\txv-\aff\|_{B_R(\txv)} &\lesssim \eps^{-\eta}\|U_\txv\|_{B_R(\txv)} \le \eps^{-\eta}\Iu{U},
\end{align*}
Step \ref{pde_normeq_St} implies
\begin{align}
\begin{split}\label{jc74}
[U]_{\eta} &\lesssim \Iu{U}\eps^{-\eta} + [U]_{\eta}(\eps^{2-\eta} + \eps^{\eta}) 
+ M\sum_{\kappa\in\Aa{}} \left (\eps^{-\eta+2\kappa-4}+\eps^{\kappa}\right ).
\end{split}
\end{align}
Taking into account $\eta \in (0,2)$ and using the qualitative property that $[U]_{\eta}<\infty$, we may choose $\eps$ small enough to ensure \eqref{jc76}.

\medskip

{\sc Step} \arabic{PDE_P}.\label{nu}\refstepcounter{PDE_P}
Finally, we show the full estimate \eqref{KS34}.
For $\txv\in \domain$ and $\syv\in B_{\frac12\dist_{\txv}}(\txv)$ choose $\rzv:=\syv+d(\syv,\txv)e_i$ for each $i\in\{1,\ldots,\di\}$. Observe that $(\zv-\yv)_i=d(\rzv,\syv)=d(\syv,\txv)$, so that in particular $\rzv\in B_{\frac12\dist_{\txv}}(\syv)$. Using
\begin{align}\label{jc87}
 d(\rzv,\txv)\le d(\syv,\txv)+d(\rzv,\syv) = 2 d(\syv,\txv)<\dist_{\txv}
\end{align}
and $\frac12\dist_{\txv}\le \dist_{\syv}$, we see $\rzv\in B_{\dist_{\txv}}(\txv) \cap B_{\dist_{\syv}}(\syv)$. Hence, the definition \eqref{jc78} of $\HNw{U}{\eta}$ and the triangle inequality yield
\begin{align*}
 \dist_{\txv}^{\eta} &|U_\txv(\rzv)-U_\txv(\syv)-U_\syv(\rzv) | \\
 & \lesssim [U]_{\eta}(d^{\eta}(\rzv,\txv)+d^{\eta}(\syv,\txv)+d^{\eta}(\rzv,\syv)) \lesssim [U]_{\eta} d^{\eta}(\syv,\txv),
\end{align*}
where in the last step we used \eqref{jc87} again.
We now combine this with the three-point continuity condition \eqref{KS2} and the triangle inequality, using again $d(\rzv,\syv)=d(\syv,\txv)$ to the effect of
\begin{align*}
\dist_{\txv}^{\eta}&\left|\gamma_\txv(\syv)\cdot(\zv-\yv)\right | \lesssim (M + [U]_{\eta}) d^{\eta}(\syv,\txv).
\end{align*}
Noting $|\gamma_\txv(\syv)\cdot(\zv-\yv)|=|\gamma_{i,\txv}(\syv)|d(\syv,\txv)$, we have together with \eqref{jc76}
\begin{align*}
\dist_{\txv}^{\eta}\left|\gamma_{i,\txv}(\syv)\right | \lesssim (M+\Iu{U})d^{\eta-1}(\syv,\txv).
\end{align*}
Since $i\in\{1,\ldots,\di\}$ was arbitrary, this yields
\begin{align*}
 \HNNw{\gamma}{\eta} \lesssim M+\Iu{U},
\end{align*}
which together with \eqref{jc76} implies \eqref{KS34}.
\end{proof}

\subsection{Application of Reconstruction}
We now need to place a constraint on the height of $\eta$. A lower bound is required for reconstruction and an upper bound is required to use the Continuity Lemma.  We will additionally need an upper bound in terms of $\nn\in \N$.
More specifically, recalling that $\nn \in \N$ is defined such that $\nn\alpha<2<(\nn+1)\alpha$ we select $\eta$ subject to
\begin{align}\label{jc40}
2-\alpha<\eta\le \min\{\nn\alpha ,1+\lfloor \alpha^{-1} \rfloor\alpha \}<2.
\end{align}

\begin{lemma}\label{MTPS1}
For all $\syv\in \domain$ and $\lambda\in (0,\frac{1}{5}\dist_{\syv} )$ we have
\begin{align}\label{jq005}
    \dist_{\syv}^{\eta+\alpha}\big | F_{\syv\mop}(\syv) \big | \lesssim \tilde M \mop^{\eta+\alpha-2},
\end{align}
where
\begin{align*}
 \tilde M:=\sup_{ \alpha \leq \kappa \leq \eta} [u]_{\kappa}^{\frac{\eta+\alpha}{\kappa}} +(\opnorm{f_{\eta+\alpha}}+1 \vee [\Pi]^{\frac{\eta}{\alpha}})[\Pi] 
 +(\|\nu\|_1')^{\eta+\alpha}.
\end{align*}

\end{lemma}
\begin{proof}
We divide the proof in three steps.
In Step \ref{ARS1} we show an identity for the difference of $F_\txv$ and $F_\syv$.
This identity is used in Step \ref{ARS2} to obtain an estimate which we recognize as the input \eqref{R1} of Proposition \ref{rec_lem}.
In Step \ref{ARS3} we apply Proposition \ref{rec_lem} and obtain \eqref{jq005}.  

\medskip

\newcounter{AR} 
\refstepcounter{AR} 
{\sc Step} \arabic{AR}.\label{ARS1}\refstepcounter{AR}
In this step, we show that for all $\txv,\syv\in\domain$ there is $u^*\in \R$ between $u(\txv)$ and $u(\syv)$ with
\begin{align}
    F_\txv-F_\syv 
    &=\big(f_{\eta+\alpha}(\syv).\id - f_{\eta+\alpha}(\txv).\Gamma_{\syv\txv}\big)\lrhs_\syv \label{MT90}\\
    &\quad + \frac{1}{(\nn-1)!}\big(a^{(\nn-1)}(u^*)-a^{(\nn-1)}(u(\txv))\big)\big(u(\syv)-u(\txv)\big)^{\nn-1}\Delta u \label{MT91}\\
    &\quad + \sum_{k=1}^{\nn-1} \frac{1}{k!}a^{(k)}(u(\txv))\big((u(\syv)-u(\txv))^k\Delta u - f_{\eta-(k-1)\alpha}(\txv).\llhs_\txv(\syv)^k\Delta\llhs_\txv\big). \label{MT92}
\end{align}
By definition \eqref{def_UF}, keeping in mind that $a(u)\Delta u+\xi-h(u)$ is independent of the base point $\txv$,
\begin{equation}
    F_{\txv}-F_{\syv}=\big ( a(u(\syv))-a(u(\txv)) \big) \Delta u +f_{\eta+\alpha}(\syv).\lrhs_{\syv}-f_{\eta+\alpha}(\txv).\lrhs_{\txv}.\label{e13}
\end{equation}
Furthermore, we claim that
\begin{align}\label{jc65}
\lrhs_{\txv}=\Gamma_{\syv\txv}\lrhs_{\syv}+\sum_{k \geq 1}\polv_{k}( \primz_{\syv\txv }\big )^{k}\Delta \llhs_{\txv}.
\end{align}
To see this, write the compatibility condition \eqref{jc60} in the form $\Pi_{\syv}^{-}=\partial_{t}\Pi_{\syv}-\polv_{0}\Delta \Pi_{\syv}$, then apply $\Gamma_{\syv \txv}$ on both sides and use the morphism property, the $j=0$ item of \eqref{model80}, and the re-expansion property \eqref{jc15}. 
In light of \eqref{jc65} and $\primz_{\syv\txv}=\llhs_\txv(\syv)$, we obtain
\begin{align*}
f_{\eta+\alpha}(\syv ).\lrhs_{\syv }-f_{\eta+\alpha}(\txv ).\lrhs_{\txv}
=f_{\eta+\alpha}(\syv ).\lrhs_{\syv }-f_{\eta+\alpha}(\txv ).\Gamma_{\syv\txv}\lrhs_{\syv}-\sum_{k\ge 1}f_{\eta+\alpha}(\txv).(\polv_k\llhs_\txv(\syv)^k \Delta\llhs_{\txv}),
\end{align*}
and thus inserting \eqref{jc65} into \eqref{e13} reveals
\begin{align*}
    F_{\txv}-F_{\syv}&= f_{\eta+\alpha}(\syv).\lrhs_{\syv}-f_{\eta+\alpha}(\txv).\Gamma_{\syv\txv}\lrhs_{\syv} \\
    &\quad +\big ( a(u(\syv))-a(u(\txv)) \big ) \Delta u -\sum_{k=1}^{\nn-1}\frac{1}{k!}a^{(k)}(u(\txv))f_{\eta-(k-1)\alpha}(\txv).(\llhs_{\txv}(\syv)^{k} \Delta \llhs_{\txv}),
\end{align*}
where we have also applied \eqref{truncatedMorphism} in form of Remark \ref{rem:morphismExt} to $\kappa:=\eta+\alpha$, $\sigma:=\polv_k\in \Ta{(k+1)\alpha}\cap \TT_-$ and $\tau:=\llhs_\txv(\syv)^k\Delta\llhs_\txv$, relying on $f_{\eta+\alpha}.\polv_k=\frac{1}{k!}a^{(k)}(u(\txv))$ for $k\le \nn-1$ and $f_{\eta+\alpha}.\polv_k=0$ otherwise since $\eta+\alpha \leq (\nn+1)\alpha$ by assumption \eqref{jc40}, cf.~\eqref{intro1}, \eqref{md0g} and \eqref{md0}.
Hence applying Taylor's formula in $u$ to order $\nn-1$ to the term $a(u(\syv))-a(u(\txv))$ yields the claimed identity.

\medskip

{\sc Step} \arabic{AR}.\label{ARS2}\refstepcounter{AR}
In this step we show that for all $\txv\in \domain$, ${\syv}\in B_{\frac12\dist_{\txv}}(\txv)$ and $0<\Mop \le \frac{1}{2}\dist_{\txv}$ it holds
\begin{align}
\dist_{\txv}^{\eta+\alpha} \big | F_{\txv\Mop}(\syv)-F_{\syv\Mop}(\syv) \big | \lesssim  \tilde M \sum_{|\beta|<\eta+\alpha  } d^{\eta+\alpha-|\beta|}({\syv},\txv)\TMc{|\beta|-2} \label{mt1}.
\end{align}
We mention in passing that
\begin{align*}
    \frac{1}{2}\dist_\txv\le \dist_\syv\le \frac{3}{2}\dist_\txv,
\end{align*}
which we will use multiple times without further mentioning. Indeed, for the first inequality, note that for all $\rzv\in \partial\domain$ using the triangle inequality and ${\syv}\in B_{\frac12\dist_{\txv}}(\txv)$ gives $\text{dist}_{\txv} \leq d(\rzv,\txv)\le d(\rzv,\syv)+\frac12\dist_\txv$, which implies the claim by taking the infimum in $\rzv$.  The second inequality is argued in a similar way.

\medskip

We now apply the convolution kernel $(\cdot)_{\mu}$ to the identity from Step \ref{ARS1}, evaluate at ${\syv}$, and estimate each contribution.  For \eqref{MT90} we first observe that \eqref{model7a} combined with \eqref{jc60} implies due to $\supp\convker\subset B_1(0)$ and $\Mop\le \frac12\dist_\txv\le \dist_\syv$, that
\begin{align}
    \dist_\txv^{\langle\beta \rangle \alpha}\|\lrhs_{\syv\Mop}(\syv)\|_{\Ta{|\beta|} } &
    \lesssim \dist_\syv^{\langle\beta \rangle \alpha} \int \|\llhs_\syv(\rzv)\|_{\Ta{|\beta|}} (\partial_\sv-\Delta)\rho_\mu(\syv-\rzv) \dd \rzv
    \lesssim [\Pi]\TMc{|\beta|-2}. \label{model7}
\end{align}
Now we use the definition \eqref{md19} together with \eqref{model7} to obtain
\begin{align*}
\dist_{\txv}^{\eta+\alpha} &\big |(f_{\eta+\alpha}(\txv).\Gamma_{{\syv}\txv} - f_{\eta+\alpha}({\syv}).\id)\lrhs_{{\syv}\Mop}({\syv}) \big | \\
& \quad \lesssim \sum_{|\beta|<\eta+\alpha} \opnorm{f_{\eta+\alpha}} d^{\eta+\alpha-|\beta| }({\syv},\txv)\dist_\txv^{ \langle \beta \rangle \alpha}\| \lrhs_{{\syv}\Mop}({\syv}) \|_{ \Ta{|\beta|} }  \\
& \quad \lesssim \sum_{|\beta|<\eta+\alpha} \opnorm{f_{\eta+\alpha}}[\Pi] d^{\eta+\alpha-|\beta| }({\syv},\txv)\TMc{|\beta|-2},
\end{align*}
which is contained in the right-hand side of \eqref{mt1}.

\medskip

Next we estimate \eqref{MT91}.
First note that  \eqref{model7} implies
\begin{align}\label{js502}
    \dist_\txv^{\alpha}|\Delta u_\Mop(\syv)|
    \lesssim [u]_\alpha \Mop^{\alpha-2}.
\end{align}
Since $u^*$ is between $u(\txv)$ and $u(\syv)$, our assumption  $\|a\|_{C^{\nn}}\lesssim 1$ with $\delta:=\frac{\eta}{\alpha}-(\nn-1)\in (0,1]$, cf.~\eqref{jc40} implies
\begin{align*}
    \dist_\txv^{\eta+\alpha}&|\big(a^{(\nn-1)}(u^*)-a^{(\nn-1)}(u(\txv))\big)\big(u(\syv)-u(\txv)\big)^{\nn-1}\Delta u_\mu| \\
    &\lesssim \dist_\txv^{\eta+\alpha}|u(\syv)-u(\txv)|^{\nn-1+\delta} |\Delta u_\Mop(\syv)| \lesssim [u]_\alpha^{\frac{\eta+\alpha}{\alpha}} d^{\eta}(\syv,\txv)\Mop^{\alpha-2},
\end{align*}
which is contained in the right-hand side of \eqref{mt1}.

\medskip

Finally, we estimate \eqref{MT92}.
For this, it is enough to show that for $k\in \set{1,\ldots,\nn-1}$
\begin{align}\label{js501}
\begin{split}
    &\dist_\txv^{\eta+\alpha}\big|\big (u(\syv)-u(\txv) \big )^{k} \Delta u_\Mop(\syv) -f_{\eta-(k-1)\alpha}(\txv).(\llhs_{\txv}(\syv)^{k} \Delta \llhs_{\txv\Mop}(\syv))\big| \\
    &\lesssim \tilde M \sum_{|\beta|<\eta+\alpha} d^{\eta+\alpha-|\beta|}(\syv,\txv) \Mop^{|\beta|-2}.
\end{split}
\end{align}
We apply Lemma \ref{lem_identity} with $g:=f(\txv)$, $\kappa:=\eta-(k-1)\alpha\in (0,2)$, $J=k+1$, $u_{1}=\Delta u_\mu(\syv)$, $\tau_1:=\Delta \llhs_{\txv\mu}(\syv)$ and $u_{m}=u(\syv)-u(\txv)$ as well as $\tau_m:=\llhs_\txv(\syv)$ for $m\in\set{2,\ldots, k+1}$.
Then for $|\beta|<\eta-(k-1)\alpha$, we have 
\begin{align}
    &\Delta u_\Mop(\syv)\big (u(\syv)-u(\txv) \big )^{k}  -f_{\eta-(k-1)\alpha}(x).(\Delta \llhs_{\txv\Mop}(\syv)\llhs_{\txv}(\syv)^{k} ) \nonumber \\
    &= \big (\Delta u_\Mop(\syv)-f_{\eta-(k-1)\alpha}(\txv).\Delta \llhs_{\txv\Mop}(\syv) \big)(u(\syv)-u(\txv))^{k-1} \label{e14} \\
     &+\sum_{|\beta|<\eta-(k-1)\alpha}\sum_{j=2}^{k+1}  \big (u(\syv)-u(\txv)-f_{\eta-(k-2)\alpha-|\beta|}.\llhs_{\txv}(\syv) \big ) \big (  u(\syv)-u(\txv) \big )^{k-j+1} \label{e2a} \\
     & \qquad \qquad \qquad \qquad \qquad \qquad \qquad \qquad \qquad \qquad \qquad \qquad \times f(\txv).P_{\beta}(\Delta \llhs_{\txv\Mop}(\syv) \llhs_{\txv}(\syv)^{j-2} ) \nonumber  .
\end{align}
To estimate \eqref{e14}, we notice that for any $\kappa=\eta-(k-1)\alpha$ it holds
\begin{align}\label{jc85}
\begin{split}
\dist_{\txv}^{\kappa}\big |&\Deltap u_{\Mop}({\syv})- f_{\kappa}({\txv}).\Deltap  \llhs_{{\txv}\Mop}({\syv}) \big | \\
&\stackrel{\eqref{jc15} }{\lesssim} \dist_{{\txv}}^{\kappa}\int |(u(\rzv)-u({\txv})-f_{\kappa}({\txv}).\primz_{\rzv{\txv}}) \Deltap \convker_{\Mop}({\syv}-\rzv)| \dd \rzv \nonumber \\ 
&\lesssim [u]_{\kappa} \int d^{\eta}(\txv,\rzv) |\Deltap \convker_{\Mop}({\syv}-\rzv)| \dd \rzv 
\lesssim [u]_{\kappa} (d^\kappa(\syv,\txv)+\TMc{\kappa})\TMc{-2},
\end{split}
\end{align}
where we used that $\mathbf{z} \in B_{\text{dist}_{\mathbf{x} }}(\mathbf{x})$ since $\supp\rho\subset B_1(0)$ and $d(\mathbf{z},\mathbf{x}) \leq \mu+ d(\syv,\txv) \leq \dist_{\mathbf{x}}$.  Hence, the contribution of \eqref{e14} to $\tilde{M}$ is
\begin{equation}
    [u]_{\eta-(k-1)\alpha }[u]_{\alpha}^{k} \lesssim [u]_{\eta-(k-1)\alpha }^{\frac{\eta+\alpha}{\eta-(k-1)\alpha} }+[u]_{\alpha}^{\frac{\eta+\alpha}{\alpha} } \nonumber.
\end{equation}
The estimate for \eqref{e2a} follows directly from the definition \eqref{js100} together with the bound  
\begin{align}
&\dist_{\txv}^{|\beta|+(j-2)\alpha}\big | f(\txv).P_{\beta}(\Delta \llhs_{\txv\Mop}(\syv) \llhs_{\txv}(\syv)^{j-2} ) \big |  \nonumber \\
&\lesssim d(\syv,\txv)^{|\beta|+(j-2)\alpha}
    \begin{cases}
(\|\nu\|_1')^{\beta_{x}} [\Pi]^{j-1}  &\quad |\beta| \neq 1 \\
(\|\nu\|_1')^{\beta_{x}} [\Pi]^{j-2}1_{j \geq 3}  &\quad |\beta| = 1 \\
    \end{cases}
     \nonumber.
\end{align}
The argument is entirely analagous to the proof of \eqref{e195} for $|\betaa|=0$, but also taking into account that $\Delta \Pi_{x} \in \TT_{-}$ and hence $P_{\beta}\Delta \Pi_{x}=0$ if $|\beta|=1$ (which explains the $1_{j \geq 3}$).  Hence, the contribution to $\tilde{M}$ from a summand in \eqref{e2a} with $\langle \beta \rangle \geq 1$ is
\begin{equation}
    [u]_{\eta-(k-2)\alpha-|\beta|}[u]_\alpha^{k+1-j}(\|\nu\|_1')^{|\beta_x|} [\Pi]^{j-1} \lesssim [u]_{\eta-(k-2)\alpha-|\beta|}^{\frac{\eta+\alpha}{\eta-(k-2)\alpha-|\beta|}}+[u]_{\alpha}^{\frac{\eta+\alpha}{\alpha}}+  (\|\nu\|_1')^{\eta+\alpha}+[\Pi]^{\frac{(j-1)(\eta+\alpha)}{(j+\langle \beta \rangle-2)\alpha} }.
\end{equation}
Note that the exponent of $[\Pi]$ satisfies the following  inequality:  $\frac{\eta+\alpha}{\eta} \leq \frac{(j-1)(\eta+\alpha)}{(j+\langle \beta \rangle-2)\alpha} \leq \frac{\eta+\alpha}{\alpha}$.   
Indeed, the upper bound follows from $\langle \beta \rangle \geq 1$, and the lower uses the constraint $|\beta|<\eta-(k-1)\alpha$ followed by $2 \leq j \leq k+1$ to bound from below by $\frac{(j-1)(\eta+\alpha)}{\eta+(j-k-1)\alpha} \geq (j-1)\frac{\eta+\alpha}{\eta} \geq \frac{\eta+\alpha}{\eta}$.  The case $|\beta|=1$, for which  $\langle \beta \rangle=0$, is treated in an identical way, simply taking into account the change in the exponent of $[\Pi]$.

\medskip

{\sc Step} \arabic{AR}.\label{ARS3}\refstepcounter{AR}
In this step we will obtain \eqref{jq005} as a consequence of Proposition \ref{rec_lem} and the estimate \eqref{mt1}.
We first argue that $F_\txv(\txv)=0$ for all ${\txv}\in B_{1}(0)$.
To show this, note that
\begin{align}\label{jc67}
 f_{\eta+\alpha}(\txv).q=\sum_{|\beta'|_s=0}^{\nn-1} da(u(\txv))^{\beta'} q_{\beta}(a(u(\txv)))=h(u(\txv)).
\end{align}
Indeed, since $q_\beta=0$ unless $\betan=0$, \eqref{md0} gives
\begin{align*}
 f_{\eta+\alpha}(\txv).q=\sum_{\substack{|\beta|<\eta +\alpha \\ \betan=0}} da(u(\txv))^{\beta'}q_{\beta}(a(u({\txv}) )),
\end{align*}
and for $\beta$ with $\betan=0$, \eqref{model201}, \eqref{jc40} and $\nn\alpha<2<\eta+\alpha$ imply
\begin{align*}
 0\le|\beta'|_s\le \nn-1 \quad \Leftrightarrow \quad \alpha\le |\beta|<\eta+\alpha.
\end{align*}
Together with $\lrhs_{\txv}(\txv)=\xi(\txv)\na -q$, cf.\@ \eqref{jc60a},
\eqref{jc67} yields
\begin{align*}
F_\txv(\txv)=0.
\end{align*}
Our last step is to use \eqref{mt1} to obtain \eqref{R1}, so that \eqref{R2} turns into our desired output 
\begin{equation}
\big |F_{\rzv\lambda}(\rzv) \big | \lesssim \dist_{\rzv}^{-(\eta+\alpha)}\tilde{M} \lambda^{\eta+\alpha-2},\nonumber
\end{equation}
for $\rzv \in B_{1}(0)$ and $\mop \in (0,\frac{1}{5}\dist_{\rzv})$.
Note that we simply relabelled $\syv$ as $\rzv$ in comparison to \eqref{jq005}.
To do so, we use Proposition \ref{rec_lem} and now fix $\rzv \in \domain$ and $\mop \in (0,\frac{1}{5}\dist_{\rzv})$.
We need to argue that for any $\Mop \in (0,\mop)$ and $\txv,\syv \in B_{\mop-\Mop}(\rzv)$ it holds
\begin{equation}
\big | F_{{\txv} \mu}({\syv} )-F_{{\syv}\mu}({\syv}) \big | \lesssim \tilde{M}(\dist_{\rzv})^{-(\eta+\alpha)}\sum_{\kappa \in \Aa{}}d^{\eta+\alpha-\kappa}(\syv,\txv)\mu^{\kappa-2} \nonumber. 
\end{equation}
To apply \eqref{mt1}, we need to show that ${\txv},{\syv} \in B_{\lambda-\mu}({\rzv} )$ and $\lambda<\frac{1}{5}\dist_{\rzv}$ imply ${\syv} \in B_{\frac{1}{2}\dist_{\txv}}({\txv})$.   
Indeed, first note that $d(\syv,\txv) \leq d({\txv},{\rzv})+d({\syv},{\rzv}) \leq 2(\lambda-\mu)\leq \frac{2}{5} \dist_{\rzv}$.
Furthermore, since $\dist_\rzv\le \dist_\txv + d(\rzv,\txv)\le \dist_\txv+\lambda\le \dist_\txv+\frac15\dist_\rzv$, we have $\frac45\dist_\rzv\le \dist_\txv$ so that combining the two yields $d(\syv,\txv) \leq \frac{1}{2} \dist_{\txv}$.
Hence, \eqref{mt1} is applicable, and since our argument also showed that $\dist_{\rzv} \lesssim \dist_{\txv}$, we may change the factor of $\dist_{\txv}^{\eta+\alpha}$ to $\dist_{\rzv}^{\eta+\alpha}$, which completes the proof.
\end{proof}

\subsection{Negative Interpolation}
We now use interpolation, the argument and notation being similar to the proof of Lemma \ref{lem: positiveInt}. 

\begin{lemma}\label{jq007}
Let $u:\R^d\to\R$ be a smooth function and define $\nu:\R^d\to\R$ via \eqref{jc71}.
Assume $0<\delta<\frac{1}{2}$ and let $\tilde M\in\R$ be defined as in Lemma \ref{MTPS1}.  For each $\txv\in D$ and $\mop\in(0,\frac15\dist_\txv)$, 
\begin{align}
&\dist_\txv^\eta\lambda^{2-\eta}\inf_{a_{0}\in I,c_0 } |(\partial_{\tv}-a_{0}\Deltap )U_{\txv \lambda}(\txv)-c_0| \nonumber \\
&\lesssim   (\|u\|+\delta^{\alpha}[\Pi])^{\frac{\alpha}{\eta+\alpha} } (\tilde{M}+[u]_{\eta} )^{\frac{\eta}{\eta+\alpha}} +\big ( \|u\|+\delta^{\alpha}[\llhs] \big ) (\delta^{-1} \vee [\llhs]^{\frac{1}{\alpha}})^{\eta}\label{newIntPf1}.
\end{align}
\end{lemma}
\begin{proof}
Let $\txv\in D$ and $\mop\in(0,\frac15\dist_\txv)$, and set $R:=\frac{\lambda}{\dist_\txv}$.

\medskip

\textit{Small scale bounds}: 
We will first argue that
\begin{equation}
\dist_\txv^\eta\lambda^{2-\eta}\inf_{a_{0}\in I,c_0 } |(\partial_{\tv}-a_{0}\Deltap )U_{\txv \lambda}(\txv)-c_0|\lesssim \tilde{M}R^{\alpha}+[\llhs]\sum_{\eta \leq |\beta|<\eta+\alpha}(\|\nu\|_1')^{|\beta_{x}|} R^{|\beta|-\eta}.\label{e15}
\end{equation}
To prove it, note that
\begin{align*}
(\partial_{t}-a(u(\mathbf{x}))\Delta)U_{\mathbf{x}}=F_{\mathbf{x}}-(f_{\eta+\alpha}(\mathbf{x})-f_{\eta}(\mathbf{x})).\Pi_{\mathbf{x}}^{-}+\aff_{0} \nonumber, 
\end{align*}
where the $\aff_{0}:=f_{\eta}(\txv).P_\txv$ is active only if $P_\txv$ from (the footnote of) Assumption \ref{sne1} is non-zero.
We now apply $(\cdot)_{\lambda}$ on both sides and evaluate at $\syv=\txv$.
By \eqref{jq005} and the constraint on $\mop$ we have
\begin{align*}
    \dist_\txv^\eta\mop^{2-\eta} |F_{\txv\mop}(\txv)|\lesssim \dist_\txv^{-\alpha}\tilde M\mop^{\alpha}= \tilde M R^\alpha.
\end{align*}
Furthermore, using \eqref{model7} we have 
\begin{align*}
 \dist_\txv^\eta\mop^{2-\eta}\big| (f_{\eta+\alpha}-f_{\eta})(\mathbf{x} ).\lrhs_{\mathbf{x}\lambda}(\mathbf{x} )| 
 \lesssim  [\llhs]\sum_{\eta \le |\beta|<\eta+\alpha}(\|\nu\|_1')^{|\beta_{x}|} R^{|\beta|-\eta} 
\end{align*}
Since $a(u(\txv))\in I$ and $\aff_0(\txv)\in\R$, this completes the proof of \eqref{e15}.

\medskip

\textit{Large scale bounds: }We now turn to the large-scale estimate and in this case we write
\begin{align*}
 (\partial_{\tv}-a(u(\txv))\Deltap )U_{\txv \mop}&=(\partial_{\tv}-a(u(\txv))\Deltap )u_{\mop} - f_\eta(\txv).(\partial_{\tv}-a_0\Deltap )\llhs_{\txv\mop},
\end{align*}
which implies via the triangle inequality 
\begin{align}
\dist_\txv^\eta\lambda^{2-\eta}\inf_{a_{0}\in I,c_0\in\R } |(\partial_{\tv}-a_{0}\Deltap )U_{\txv \mop}(\mathbf{x})-c_0|
 \lesssim \|u\| R^{-\eta} + [\llhs]\sum_{|\beta|<\eta} (\|\nu\|_{1}')^{|\betan|} R^{|\beta|-\eta} \label{e16}
\end{align}
Hence, combining \eqref{e15} and \eqref{e16} we obtain
\begin{align}
\dist_\txv^{\eta}\lambda^{2-\eta}&\inf_{a_{0}\in I,c_0\in\R } |(\partial_{\tv}-a_{0}\Deltap )U_{\txv \mop}(\mathbf{x})-c_0| \\
 &\lesssim \tilde{M} R^{\alpha}+ \|u\| R^{-\eta} + [\llhs]\sum_{|\beta|<\eta+\alpha} (\|\nu\|_{1}')^{\betan} R^{|\beta|-\eta}. \label{e16a}
\end{align}
If $R\le 1$ satisfies \eqref{e4}, we may use \eqref{e2} to obtain
\begin{align}
\dist_\txv^{\eta}\lambda^{2-\eta}&\inf_{a_{0}\in I,c_0\in\R } |(\partial_{\tv}-a_{0}\Deltap )U_{\txv \mop}(\mathbf{x})-c_0| \nonumber \\
 &\lesssim (\tilde{M}+[u]_\eta) R^{\alpha}+ (\|u\|+\delta^\alpha [\llhs]) R^{-\eta}.\label{e16b}
\end{align}
To complete the proof, we argue similarly as in Lemma \ref{lem: positiveInt}.  Namely, balancing the terms leads us to define $R$ via $R^{\eta+\alpha}=(\|u\|+\delta^\alpha [\llhs])(\tilde{M}+[u]_\eta)^{-1}$.  If \eqref{e4} is satisfied with this choice of $R$, then we obtain 
\begin{equation}
 \dist_\txv^\eta\lambda^{2-\eta}\inf_{a_{0}\in I,c_0 } |(\partial_{\tv}-a_{0}\Deltap )U_{\txv \lambda}(\txv)-c_0|
\lesssim   (\|u\|+\delta^{\alpha}[\Pi])^{\frac{\alpha}{\eta+\alpha} } (\tilde{M}+[u]_{\eta} )^{\frac{\eta}{\eta+\alpha}} \nonumber,   
\end{equation}
which implies \eqref{newIntPf1}.  If \eqref{e4} fails, then we find
\begin{align}
    \tilde{M}+[u]_{\eta} 
     \lesssim \big ( \|u\|+\delta^{\alpha}[\llhs] \big ) (\delta^{-1} \vee [\llhs]^{\frac{1}{\alpha}})^{\eta+\alpha}, \nonumber
\end{align}
so that choosing $R\approx \delta \wedge [\Pi]^{-\frac{1}{\alpha}}$, \eqref{e16b} leads to 
\begin{align}
\dist_\txv\lambda^{2-\eta}\inf_{a_{0}\in I,c_0\in\R } |(\partial_{\tv}-a_{0}\Deltap )U_{\txv \mop}(\mathbf{x})-c_0| 
 \lesssim \big ( \|u\|+\delta^{\alpha}[\llhs] \big ) (\delta^{-1} \vee [\llhs]^{\frac{1}{\alpha}})^{\eta},\nonumber
\end{align}
which again implies \eqref{newIntPf1}.

\end{proof}

\subsection{Application of Integration}

We will now use Proposition \ref{int lem} in order to transform the output \eqref{jq005} of Lemma \ref{MTPS1} into a bound on the solution $u$ to the renormalized equation \eqref{intro3}.

\begin{lemma}\label{jq006}
Let $u$ be a smooth solution to \eqref{intro3} and define $\nu$ as in \eqref{jc71}. For $\eta$ satisfying \eqref{jc40} and $ \delta \in (0,\frac{1}{2})$ it holds
\begin{align}
[u]_{\eta}+[\nu]_{\eta}' &\lesssim (\|u\|+\delta^{\alpha}[\Pi])^{\frac{\alpha}{\eta+\alpha} } \tilde{M} ^{\frac{\eta}{\eta+\alpha}}+   \big (\Inu{\nu}+\opnorm{f_{\eta}} )[\llhs]  \nonumber \\
&
+\big ( \|u\|+\delta^{\alpha}[\llhs] \big ) (\delta^{-1} \vee |[\llhs]^{\frac{1}{\alpha}})^{\eta}
\label{e178}
\end{align}
\end{lemma}
\begin{proof}
\newcounter{MTP} 
\refstepcounter{MTP} 
Recall that the jet in Proposition \ref{int lem} is required to be centered to first order in the sense that $U_\txv(\txv)=\nablap_{\yv}|_{\syv=\txv}U_\txv(\syv)=0$.  For the specific choice \eqref{def_UF}, we find that $U_{\txv}(\txv)=0$ since $\Pi_{\txv \beta}(\txv)=0$ for $|\beta|<2$, cf.~\eqref{model7a}.  Let us now explain how the choice of $\nu$ in \eqref{jc72} ensures 
$\nablap_{\yv}|_{\syv=\txv}U_\txv(\syv)=0$.  Indeed, note that
\begin{align}
\nabla U_{\txv}(\syv)=\nabla u(\syv)-f_{\eta}(\txv).\nabla \Pi_{\txv}(\syv)
=\nabla u(\syv)-f_{\eta}(\txv).\nabla(\id-\projm) \Pi_{\txv}(\syv)-\nu(\txv), 
\end{align}
where we recall that $\id-\projm$ is the projection of $\Tplus$ onto $\Tminus$, cf.\@ Section \ref{sec:ms} and we used  $\nablap \projm\llhs_{\txv}=\poly_x$, cf.\@ \eqref{jc03b} together with $f_\eta(\txv).\poly_x=\nu(\txv)$, cf.\@ \eqref{md0}.
Furthermore, since $\nablap\llhs_{\txv\beta}(\txv)=0$ for $|\beta|>1$, cf.\@ \eqref{model7a}, we may use the definition of $f_{\eta}$, cf.\@ \eqref{md0}, to write
\begin{align*}
\nu(\txv)=\nablap u(\txv)-f_{1}(\txv).\nablap \llhs_{\txv}(\txv) =\nablap u(\txv)-f_{\eta}(\txv). \nabla (\id-\projm)\llhs_{\txv}(\txv),
\end{align*}
and obtain that $\nablap_{\yv}|_{\syv=\txv}U_\txv(\syv)$ vanishes. 

\medskip

\medskip

{\sc Step} \arabic{MTP}.\refstepcounter{MTP}\label{MTPS3} In this step, we establish the three-point continuity condition: for all $\txv\in B_1(0)$, $\syv\in B_{\frac12\dist_{\txv}}(\txv)$ and $\rzv\in B_{\frac12\dist_{\txv}}(\syv)$ it holds
\begin{align}
\dist_{\txv}^{\eta}&\big |U_\txv(\rzv)-U_\txv(\syv)-U_\syv(\rzv)-\gamma_\txv(\syv)\cdot (\zv-\yv) \big | 
 \lesssim \sum_{\alpha \leq |\beta|\leq \eta-\alpha} \opnorm{f_{\eta}}[\Pi]d^{\eta-|\beta|}(\syv,\txv)d^{|\beta|}(\rzv,\syv),\label{mt2}
\end{align}
where $\gamma_\txv(\syv)$ is defined by
\begin{equation}
\gamma_\txv(\syv):=(f_{\eta}(\txv).\Gamma_{\syv\txv}-f_{\eta}(\syv).\id)\poly_x.
\end{equation}
To establish \eqref{mt2}, use \eqref{jc15} to write $\llhs_{\txv}(\rzv)-\llhs_{\txv}(\syv)=\Gamma_{\syv\txv}\llhs_{\syv}(\rzv)$, so that
\begin{align*}
&U_\txv(\rzv)-U_\txv(\syv)-U_\syv(\rzv) - \gamma_\txv(\syv)\cdot (\zv-\yv)\\
&=f_{\eta}(\txv).(\llhs_{\txv}(\rzv)-\llhs_{\txv}(\syv)) - f_{\eta}(\syv).\llhs_{\syv}(\rzv) - \gamma_\txv(\syv)\cdot (\zv-\yv)\\
&=f_{\eta}(\txv).\Gamma_{\syv\txv}\llhs_{\syv}(\rzv) - f_{\eta}(\syv).\llhs_{\syv}(\rzv) - \gamma_\txv(\syv)\cdot (\zv-\yv)\\
&=(f_{\eta}(\txv).\Gamma_{\syv\txv} - f_{\eta}(\syv).\id)(\id-\projm)\llhs_{\syv}(\rzv),
\end{align*}
where we have used $\projm\llhs_{\syv}(\rzv):=\poly_x\cdot(\zv-\yv)$, cf.\@ \eqref{jc03b}.
Since $\dist_\txv\le 2\dist_\syv$, we find by \eqref{md19}
\begin{align*}
\dist_{\txv}^{\eta}&\big |U_\txv(\rzv)-U_\txv(\syv)-U_\syv(\rzv)-\gamma_\txv(\syv)\cdot (\zv-\yv) \big | \\
& \lesssim \sum_{|\beta|=\alpha }^{\eta-\alpha} \opnorm{f_{\eta}} d^{\eta-|\beta|}(\syv,\txv)\dist_\syv^{\langle \beta\rangle \alpha}\|(\id-\projm)\primz_{\rzv\syv}\|_{\Ta{|\beta|}} \\
&\stackrel{\mathclap{\eqref{model7a}}}{\lesssim} \sum_{|\beta|=\alpha}^{\eta-\alpha}\opnorm{f_{\eta}}[\Pi]  d^{\eta-|\beta|}(\syv,\txv)d^{|\beta|}(\rzv,\syv),
\end{align*}
which yields \eqref{mt2}.

\medskip

{\sc Step} \arabic{MTP}.\refstepcounter{MTP}\label{MTPS4} We may apply Proposition \ref{int lem} to the jet $\syv \mapsto U_\txv(\syv)$, as we have verified \eqref{KS1New} in Lemma \ref{jq007} and \eqref{KS2} in Step \ref{MTPS3}, where the set $\Aa{}\cap (0,\eta]$ plays the role of $\Aa{}$. Moreover, $U_\txv(\txv)=\nablap_{\yv}|_{\syv=\txv}U_\txv(\syv)=0$ as observed at the beginning of the proof. 

\medskip

Observe that for $\txv\in B_1(0)$ and $\syv\in B_{\dist_{\txv}}(\txv)$ we have
\begin{align*}
 u(\syv)-u(\txv)-f_{\eta}(\txv).\primz_{\syv\txv} &\stackrel{\mathclap{\eqref{jc15}}}{=} u(\syv)-u(\txv)-f_{\eta}(\txv).\llhs_{\txv}(\syv) = U_\txv(\syv) 
\end{align*}
and thus by \eqref{model7a} and $d(\txv,\syv)\le \dist_{\txv}\le 1$ and the definition of $\Iu{U}$ in \eqref{jc79}
\begin{align*}
 \Iu{U}\lesssim \|u\| + [\llhs] + \Inu{\nu}[\llhs].
\end{align*}
Moreover, the definition of $\gamma_\txv(\syv)$ in Step \ref{MTPS3}, \eqref{model81} and $f_\eta.\poly_x=\nu$, cf.\@ \eqref{md0}, imply
\begin{align*}
\nu(\syv)-\nu(\txv) - f_\eta(\txv).\primo_{\syv\txv} =-\gamma_\txv(\syv).
\end{align*}
Therefore, as a result of the steps above, and taking into account that $\delta^{-1} \vee [\Pi]^{\frac{1}{\alpha}} \geq 1$, and also using Young's inequality, the output \eqref{KS34} implies \eqref{e178}.
\end{proof}

\subsection{Proof of the Main Theorem}

\begin{proof}
The main step is to prove that for $\eta$ satisfying \eqref{jc40}, there exists a universal $\epsilon>0$ such that if $\|u\| \leq \epsilon$, then  
\begin{equation}
    [u]_{\eta}\lesssim  (\|u\|+[\Pi]) (1 \vee [\llhs]^{\frac{\eta}{\alpha}} )\label{e180}.
\end{equation}
We will choose $\epsilon$ later in the proof, but we mention already that we will apply Lemma \ref{lem: positiveInt}, Lemma \ref{jq006} and the \autoref{lem1}  with $\delta:=\frac{1}{2} \wedge (\epsilon [\Pi]^{-1})^{\frac{1}{\alpha}}$, so we may assume throughout that
\begin{equation}
\|u\|+\delta^{\alpha}[\llhs] \leq 2\epsilon \leq 1.  \label{e179}
\end{equation}
We will make use of the following interpolation inequalities:
\begin{align}
[u]_{\kappa} \lesssim   [u]_{\eta}^{\frac{\kappa}{\eta}}, \quad
 \|\nu\|_{1}' \lesssim  [u]_{\eta}^{\frac{1}{\eta}} , \quad [\nu]_{\kappa}' \lesssim ([u]_{\eta} \vee [\nu]_{\eta}' )^{\frac{\kappa}{\eta} }+\delta^{-\kappa} \label{e184}
\end{align}
which follow from \eqref{e174}-\eqref{e175} and \eqref{e179}.  Indeed, for the first two inequalities, notice that we have dropped the second term on the RHS of each of \eqref{e174}-\eqref{e175}, which is possible since if the second term dominates the first term on the RHS \eqref{e174}-\eqref{e175}, then 
\begin{equation}
    [u]_{\eta} \lesssim (\|u\|+\delta^{\alpha}[\Pi] )(\delta^{-1} \vee [\Pi]^{\frac{1}{\alpha}})^{\eta} \lesssim (\|u\|+[\Pi]) (1 \vee [\llhs]^{\frac{\eta}{\alpha}}), \nonumber
\end{equation}
which is precisely \eqref{e180}. 
The third inequality in \eqref{e184} was already justified in the proof of \eqref{e182}.

\medskip
 
Our plan is to deduce \eqref{e180} from \eqref{e178}, so we need to control $\tilde{M}$.
We claim that 
\begin{equation}
    \tilde{M} \lesssim [u]_{\eta}^{\frac{\eta+\alpha}{\eta}}+([\nu]_{\eta}' )^{\frac{\eta+\alpha}{\eta}} +[\Pi]\delta^{-\eta }  \label{e181}.
\end{equation}
In light of \eqref{e184}, it holds
\begin{equation}
    \sup_{ \alpha \leq \kappa \leq \eta} [u]_{\kappa}^{\frac{\eta+\alpha}{\kappa}} 
 +(\|\nu\|_1')^{\eta+\alpha}+[\Pi](1 \vee [\Pi]^{\frac{\eta}{\alpha}} ) \lesssim [u]_{\eta}^{\frac{\eta+\alpha}{\eta}} +[\Pi]\delta^{-\eta}\nonumber. 
\end{equation}
Furthermore, by the \autoref{lem1} and Young's inequality it holds
\begin{align}
\opnorm{f_{\eta+\alpha}}[\Pi] 
\lesssim ([u]_{\eta}+[\nu]'_{\eta}+\delta^{-\eta} )[\Pi]
\lesssim [u]_{\eta}^{\frac{\eta+\alpha}{\eta}}+([\nu]_{\eta}' )^{\frac{\eta+\alpha}{\eta}}+[\Pi]^{\frac{\eta+\alpha}{\alpha}}+[\Pi]\delta^{-\eta}, \nonumber    
\end{align}
which completes the proof of \eqref{e181}.  

\medskip

We now consider the remaining contributions to \eqref{e178}.  Using again the \autoref{lem1}, Young's inequality, and \eqref{e184} we find that
\begin{align}
(\opnorm{f_{\eta}}+\|\nu\|_{1}')[\Pi] &\lesssim ([u]_{\eta-\alpha}+[\nu]_{\eta-\alpha}'+\delta^{-(\eta-\alpha)}+[u]_\eta^{\frac{1}{\eta}} )[\Pi]  \nonumber \\
&\lesssim \big ( [u]_{\eta}+[\nu]_{\eta-\alpha}' \big )^{\frac{\eta-\alpha}{\eta}}[\Pi]+\epsilon^{\eta}[u]_{\eta}+\epsilon^{-\frac{\eta-1}{\eta}}[\Pi]^{\frac{\eta}{\eta-1}}+\delta^{-\eta} [\Pi]  \nonumber\\
& \lesssim \epsilon^{\frac{\eta}{\eta-\alpha}}([u]_{\eta}+[\nu]_{\eta}')+\delta^{-\eta} [\Pi]\nonumber.
\end{align}
We now insert the above estimates into \eqref{e178}, taking into account \eqref{e179} and Young's inequality to obtain
\begin{align}
[u]_{\eta}+[\nu]_{\eta}' 
&\lesssim (\epsilon^{\frac{\alpha}{\eta+\alpha} }+\epsilon^{\frac{\eta}{\eta-\alpha} } )  ( [u]_{\eta}+[\nu]_{\eta}'  )+[\Pi]\delta^{-\eta}+( \|u\|+\delta^{\alpha}[\llhs] \big ) \delta^{-\eta}  \nonumber \\
& \lesssim  (\epsilon^{\frac{\alpha}{\eta+\alpha} }+\epsilon^{\frac{\eta}{\eta-\alpha} } )  ( [u]_{\eta}+[\nu]_{\eta}'  )+(\|u\|+[\Pi])(1 \vee[\Pi]^{\frac{\eta}{\alpha}} ) \nonumber.
\end{align}
Choosing $\epsilon$ sufficiently small (and universal), we obtain \eqref{e180}.

\medskip

We now combine \eqref{e180} with interpolation to conclude the main estimate \eqref{jc63}.  Namely, using \eqref{e174} with $\delta=\frac{1}{2}$, we obtain
\begin{align}
[u]_{\kappa} &\lesssim [u]_{\eta}^{\frac{\kappa}{\eta}}\big ( \|u\|+[\Pi] \big )^{1-\frac{\alpha}{\eta}} +\big ( \|u\|+[\Pi] \big )(1 \vee [\Pi]^{\frac{\kappa}{\alpha}}) \nonumber \\
& \lesssim \big ( \|u\|+[\Pi] \big )(1 \vee [\Pi]^{\frac{\kappa}{\alpha}}) \nonumber.
\end{align}
If $\txv,\syv\in B_{1-r}(0)$ with $r\in(0,1)$, then $\dist_\txv\ge r$. Hence, if $\syv\in B_{\dist_\txv}(\txv)$, the above estimate with $\kappa=\alpha$ implies \eqref{jc63}, and if $\syv\notin B_{\dist_\txv}(\txv)$, then $d(\syv,\txv)\ge \dist_\txv\ge r$ and hence
\begin{align*}
  r^\alpha |u(\syv)-u(\txv)|\lesssim r^\alpha \|u\| \lesssim  \|u\| d^\alpha(\syv,\txv),
\end{align*}
which is also contained in the right-hand side of \eqref{jc63}.
We conclude by arguing the more general claim \eqref{jc72} by essentially the same argument.  No changes are required if $\kappa \leq 1$ and if $\kappa \in (1,2)$ we use that by \eqref{e175} (with $\kappa$ playing the role of $\eta$) and Young's inequality 
\begin{equation}
\|\nu\|_{1}' \lesssim [u]_{\kappa}+(\|u\|+[\Pi] )(1 \vee [\Pi]^{\frac{1}{\alpha}} ), \nonumber
\end{equation}
which implies that
\begin{equation}
r^{\kappa}|\nu(x)| \leq r |\nu(x)| \lesssim \big ( \|u\|+[\Pi] \big )(1 \vee [\Pi]^{\frac{\kappa}{\alpha}}) \nonumber.
\end{equation}

\end{proof}

\section{Acknowledgements}
The third author is very grateful to Joscha Diehl for many useful conversations on rough paths and regularity structures during his tenure at the MPI in Leipzig.  He also thanks Alexandra Blessing and Benjamin Seeger for pointers to the literature on quasi-linear SPDE with colored noise.

\bibliographystyle{plain}

\begin{thebibliography}{10}

\bibitem{AgV20c}
Antonio Agresti and Mark Veraar.
\newblock Stability {P}roperties of {S}tochastic {M}aximal
  {$L^p$}-{R}egularity.
\newblock {\em J. Math. Anal. Appl.}, 482(2), 2020.

\bibitem{AgV22a}
Antonio Agresti and Mark Veraar.
\newblock Nonlinear {P}arabolic {S}tochastic {E}volution {E}quations in
  {C}ritical {S}paces {P}art {I}. {S}tochastic {M}aximal {R}egularity and
  {L}ocal {E}xistence.
\newblock {\em Nonlinearity}, 35(8):4100--4210, 2022.

\bibitem{AgV22b}
Antonio Agresti and Mark Veraar.
\newblock Nonlinear {P}arabolic {S}tochastic {E}volution {E}quations in
  {C}ritical {S}paces {P}art {II}. {B}low-up {C}riteria and {I}nstantaneous
  {R}egularization.
\newblock {\em J. Evol. Equ.}, 22(2):Paper No. 56, 96, 2022.

\bibitem{bailleul2021renormalised}
Ismael Bailleul and Yvain Bruned.
\newblock Renormalised {S}ingular {S}tochastic {PDE}s.
\newblock {\em arXiv preprint arXiv:2101.11949}, 2021.

\bibitem{BDH19}
Isma{\"e}l Bailleul, Arnaud Debussche, and Martina Hofmanov\'{a}.
\newblock Quasilinear {G}eneralized {P}arabolic {A}nderson {M}odel {E}quation.
\newblock {\em Stoch. Partial Differ. Equ. Anal. Comput.}, 7(1):40--63, 2019.

\bibitem{bailleul2023regularity}
Isma{\"e}l Bailleul, Masato Hoshino, and Seiichiro Kusuoka.
\newblock Regularity {S}tructures for {Q}uasilinear {S}ingular {SPDE}s, 2023.

\bibitem{bailleul2019paracontrolled}
Isma{\"e}l Bailleul and Antoine Mouzard.
\newblock Paracontrolled {C}alculus for {Q}uasilinear {S}ingular {PDE}s.
\newblock {\em Stoch. Partial Differ. Equ. Anal. Comput.}, 2022.

\bibitem{bruned2020renormalising}
Yvain Bruned, Ajay Chandra, Ilya Chevyrev, and Martin Hairer.
\newblock Renormalising {SPDE}s in {R}egularity {S}tructures.
\newblock {\em J. Eur. Math. Soc. (JEMS)}, 23(3):869--947, 2021.

\bibitem{bruned2024quasigeneralised}
Yvain Bruned, Máté Gerencsér, and Usama Nadeem.
\newblock Quasi-{G}eneralised {KPZ} {E}quation, 2024.

\bibitem{bruned2016algebraic}
Yvain Bruned, Martin Hairer, and Lorenzo Zambotti.
\newblock Algebraic {R}enormalisation of {R}egularity {S}tructures.
\newblock {\em Invent. Math.}, 215(3):1039--1156, 2019.

\bibitem{chandra2016analytic}
Ajay Chandra and Martin Hairer.
\newblock An {A}nalytic {B}{P}{H}{Z} {T}heorem for {R}egularity {S}tructures.
\newblock {\em arXiv preprint arXiv:1612.08138}, 2016.

\bibitem{dubedat2019stochastic}
Julien Dub\'{e}dat and Hao Shen.
\newblock Stochastic {R}icci {F}low on {C}ompact {S}urfaces.
\newblock {\em Int. Math. Res. Not. IMRN}, (16):12253--12301, 2022.

\bibitem{funaki2020asymptotics}
Tadahisa Funaki, Masato Hoshino, Sunder Sethuraman, and Bin Xie.
\newblock Asymptotics of {PDE} in {R}andom {E}nvironment by {P}aracontrolled
  {C}alculus.
\newblock {\em Ann. Inst. Henri Poincar\'{e} Probab. Stat.}, 57(3):1702--1735,
  2021.

\bibitem{FuG19}
Marco Furlan and Massimiliano Gubinelli.
\newblock Paracontrolled {Q}uasilinear {SPDE}s.
\newblock {\em Ann. Probab.}, 47(2):1096--1135, 2019.

\bibitem{Ger20}
M\'{a}t\'{e} Gerencs\'{e}r.
\newblock Nondivergence {F}orm {Q}uasilinear {H}eat {E}quations {D}riven by
  {S}pace-{T}ime {W}hite {N}oise.
\newblock {\em Ann. Inst. H. Poincar\'{e} Anal. Non Lin\'{e}aire},
  37(3):663--682, 2020.

\bibitem{GeH19}
M\'{a}t\'{e} Gerencs\'{e}r and Martin Hairer.
\newblock A {S}olution {T}heory for {Q}uasilinear {S}ingular {SPDE}s.
\newblock {\em Comm. Pure Appl. Math.}, 72(9):1983--2005, 2019.

\bibitem{Gub04}
Massimiliano Gubinelli.
\newblock Controlling {R}ough {P}aths.
\newblock {\em J. Funct. Anal.}, 216(1):86--140, 2004.

\bibitem{GIP15}
Massimiliano Gubinelli, Peter Imkeller, and Nicolas Perkowski.
\newblock Paracontrolled {D}istributions and {S}ingular {PDE}s.
\newblock {\em Forum Math. Pi}, 3, 2015.

\bibitem{Hai13}
Martin Hairer.
\newblock Solving the {KPZ} {E}quation.
\newblock {\em Ann. of Math. (2)}, 178(2):559--664, 2013.

\bibitem{Hai14}
Martin Hairer.
\newblock A {T}heory of {R}egularity {S}tructures.
\newblock {\em Invent. Math.}, 198(2):269--504, 2014.

\bibitem{Hai16}
Martin Hairer.
\newblock Regularity {S}tructures and the {D}ynamical {$\Phi^4_3$} {M}odel.
\newblock In {\em Current {D}evelopments in {M}athematics 2014}, pages 1--49.
  Int. Press, Somerville, MA, 2016.

\bibitem{HaP15}
Martin Hairer and \'Etienne Pardoux.
\newblock A {W}ong-{Z}akai {T}heorem for {S}tochastic {PDE}s.
\newblock {\em J. Math. Soc. Japan}, 67(4):1551--1604, 2015.

\bibitem{hornung2019quasilinear}
Luca Hornung.
\newblock Quasilinear {P}arabolic {S}tochastic {E}volution {E}quations via
  {M}aximal {$L^p$} {R}egularity.
\newblock {\em Potential Analysis}, 50(2):279--326, 2019.

\bibitem{kuehn2020pathwise}
Christian Kuehn and Alexandra Neam{\c{t}}u.
\newblock Pathwise {M}ild {S}olutions for {Q}uasilinear {S}tochastic {P}artial
  {D}ifferential {E}quations.
\newblock {\em J. Differential Equations}, 269(3):2185--2227, 2020.

\bibitem{LPSX23}
Claudio Landim, Carlos~G. Pacheco, Sunder Sethuraman, and Jianfei Xue.
\newblock On a {N}onlinear {SPDE} {D}erived from a {H}ydrodynamic {L}imit in a
  {S}inai-{T}ype {R}andom {E}nvironment.
\newblock {\em Ann. Appl. Probab.}, 33(1):200--237, 2023.

\bibitem{LOT21}
Pablo Linares, Felix Otto, and Markus Tempelmayr.
\newblock The {S}tructure {G}roup for {Q}uasi-{L}inear {E}quations via
  {U}niversal {E}nveloping {A}lgebras.
\newblock {\em arXiv preprint arXiv:2103.04187}, 2021.

\bibitem{LOTT21}
Pablo Linares, Felix Otto, Markus Tempelmayr, and Pavlos Tsatsoulis.
\newblock A {D}iagram-{F}ree {A}pproach to the {S}tochastic {E}stimates in
  {R}egularity {S}tructures.
\newblock {\em arXiv preprint arXiv:2112.10739}, 2021.

\bibitem{Lyo98}
Terry~J. Lyons.
\newblock Differential {E}quations {D}riven by {R}ough {S}ignals.
\newblock {\em Rev. Mat. Iberoamericana}, 14(2):215--310, 1998.

\bibitem{MoW20}
Augustin Moinat and Hendrik Weber.
\newblock Space-{T}ime {L}ocalisation for the {D}ynamic {$\Phi^4_3$} {M}odel.
\newblock {\em Comm. Pure Appl. Math.}, 73(12):2519--2555, 2020.

\bibitem{otto2018parabolic}
Felix Otto, Jonas Sauer, Scott Smith, and Hendrik Weber.
\newblock Parabolic {E}quations with {R}ough {C}oefficients and {S}ingular
  {F}orcing.
\newblock {\em arXiv preprint arXiv:1803.07884}, 2018.

\bibitem{OtW19}
Felix Otto and Hendrik Weber.
\newblock Quasilinear {SPDE}s via {R}ough {P}aths.
\newblock {\em Arch. Ration. Mech. Anal.}, 232(2):873--950, 2019.

\bibitem{RaS20}
Claudia Raithel and Jonas Sauer.
\newblock The {I}nitial {V}alue {P}roblem for {S}ingular {SPDE}s via {R}ough
  {P}aths.
\newblock {\em arXiv preprint arXiv:2001.00490}, 2020.

\end{thebibliography}

\end{document}